\documentclass[10pt,a4paper]{article}
\title{\bf  On the $W$-entropy and  Shannon entropy power on RCD$(K, N)$ and RCD$(K, n, N)$ spaces\\
 }
\author{
Xiang-Dong Li\thanks{Research of X.-D. Li has been supported by National Key R\&D Program of China (No. 2020YF0712700) and NSFC No. 12171458.},  
	\quad Enrui Zhang
	}
\date{}

\usepackage{amsfonts}
\usepackage{mathrsfs}

\usepackage{amssymb, latexsym}

\usepackage{graphicx}

\usepackage{amsmath}

\usepackage{amssymb,amsmath,amsfonts,amsthm,array,bm,color}

\def\Ric{{\rm Ric}}

\def\<{\langle}
\def\>{\rangle}

\newtheorem{theorem}{Theorem}[section]
\newtheorem{lemma}[theorem]{Lemma}       
\newtheorem{corollary}[theorem]{Corollary}

\newtheorem{remark}[theorem]{Remark}

\newtheorem{definition}[theorem]{Definition}

\usepackage{indentfirst}
\begin{document}

\maketitle
\makeatletter % '@' is now a normal "letter" for TeX
\renewcommand\theequation{\thesection.\arabic{equation}}
\@addtoreset{equation}{section}
\makeatother % '@' is restored as a "non-letter" character for TeX

\vspace{-6mm}
\begin{abstract} In this paper, we prove  the $W$-entropy formula and the monotonicity and rigidity theorem of the $W$-entropy for the heat flow on  RCD$(K, N)$ and RCD$(K, n, N)$ spaces $(X, d, \mu)$, where $K\in \mathbb{R}$, $n\in \mathbb{N}$ is the geometric dimension of $(X, d, \mu)$ and $N\geq n$. We also prove the $K$-concavity of the Shannon entropy power on RCD$(K, N)$ spaces. 
As an application, we derive the Shannon entropy isoperimetric inequality and the Stam type logarithmic Sobolev inequality on RCD$(0, N)$ spaces with maximal volume growth condition. Finally, we prove the rigidity theorem for the Stam type logarithmic Sobolev inequality with sharp constant on noncollapsing RCD$(0, N)$ spaces.
\end{abstract}

\vskip 2mm

\medskip

%\tableofcontents

\section{Introduction}\label{sect1}

 In 2002,  G. Perelman \cite{P1} introduced the $\mathcal{W}$-entropy for the Ricci flow and proved its monotonicity along the conjugate heat equation.  More precisely, let $(M, g)$ be a closed $n$-dimensional Riemannian 
manifold, $f\in C^\infty(M)$. The $\mathcal{W}$-entropy functional  is defined by
\begin{equation*}
	\mathcal{W}(g(t),f(t),\tau):=\int_M \Big[\tau(R+|\nabla f|^2)+f-n\Big]\frac{e^{-f}}{(4\pi\tau)^{\frac{n}{2}}}dv ,
\end{equation*}
where $\tau >0$ satisfies $\frac{d\tau}{dt}=-1$, $R$ and $dv$ denote the scalar curvature and the Riemannian 
 volume form, respectively. Perelman \cite{P1} showed  that if  $g(t)$ and $f(t)$  satisfy following equations
\begin{equation*}
	\begin{split}
		\partial_t g &=-2Ric ,\\
		\partial_t f &=-\Delta f +|\nabla f |^2-R+\frac{n}{2\tau},
	\end{split}
\end{equation*}
then
\begin{equation*}
	\frac{d\mathcal{W}}{dt}=2\tau \int_M\Big\|Ric+\nabla^2f-\frac{g}{2\tau} \Big\|^2_{\operatorname{HS}}\frac{e^{-f}}{(4\pi\tau)^{\frac{n}{2}}}
	dv\geq 0.
\end{equation*}
Here $\|\cdot\|_{\rm HS}$ denotes the Hilbert-Schmidt norm.
In particular, $\frac{d\mathcal{W}}{dt}=0$ at some $t>0$ if and only if $(M, g(t), f(t))$ is a shrinking Ricci soliton 

\begin{equation*}
	Ric+\nabla^2f=\frac{g}{2\tau}. 
\end{equation*}

Inspired by Perelman's groundbreaking contributions to the study of $W$-entropy formula, many authors have extended the $W$-entropy formula to various geometric flows. In \cite{Ni1}, Ni derived the  $W$-entropy for the heat equation on  complete Riemannian manifolds. More precisely, let $(M, g)$ be a complete Riemannian manifold, and 

\begin{equation*}
	u(x, t)=\frac{e^{-f(x, t)}}{(4\pi t)^{\frac{n}{2}}}
\end{equation*}
be a positive solution to the  heat equation
\begin{equation*}
	\partial_t u=\Delta u \label{Heat1}
\end{equation*}
with $\int_M u(x,0)dv=1$. Define the $W$-entropy by
\begin{equation}\label{Wentopy}
	W(f, t):=\int_M \left( t|\nabla f |^2+f-n\right)\frac{e^{-f(x, t)}}{(4\pi t)^{\frac{n}{2}}} dv.
\end{equation}
Then the following $W$-entropy formula holds
\begin{equation*}
	\frac{d}{d t}W(f, t)=-2t\int_M \left(\Big\|\nabla^2 f-\frac{g}{2t} \Big\|_{\operatorname{HS}}^2
	+\Ric(\nabla f,\nabla f)\right)u dv. \label{WNi}
\end{equation*}
In particular,  the $W$-entropy is non-increasing when $\Ric\geq 0$.

Let $(M, g)$ be a Riemannian manifold with a weighted volume measure $d\mu=e^{-V}dv$, where $V\in C^2(M)$, $dv$ is the Riemannian volume measure. Let $\nabla$ be the Levi-Civita connection $(M, g)$, and 
$\Delta={\rm Tr}\nabla^2$ be the Laplace-Beltrami operator on $(M, g)$. Let 
$$L=\Delta-\nabla V\cdot \nabla$$ be the weighted Laplacian (called also the Witten Laplacian) on $(M, g, \mu)$, 
which is the infinitesimal generator of the following Dirichlet form in the sense that
\begin{eqnarray*}
\mathcal{E}(f, g)={1\over 2}\int_M \langle \nabla f, \nabla g\rangle d\mu=-{1\over 2}\int_M (L f)gd\mu=-{1\over 2}\int_M  f(L g)d\mu, \ \ \ \forall ~ f, g\in C_0^\infty(M).
\end{eqnarray*}
In \cite{BE1985}, Bakry and Emery proved the generalized 
Bochner formula holds for all $f\in C^2(M)$
\begin{eqnarray}
{1\over 2}L\|\nabla f\|_2^2=\langle \nabla f, \nabla L\rangle +\|\nabla^2 f\|_{\rm HS}^2+Ric(L)(\nabla f, \nabla f), 
\end{eqnarray}
where $\nabla^2 f$ is the Hessian of $f$ with respect to the Levi-Civita connection on the 
Riemannian manifold $(M, g)$,  $\|\nabla^2 f\|_{\rm HS}$ is its Hilbert-Schmidt norm, and 
\begin{eqnarray}
Ric(L)=Ric+\nabla^2 V
\end{eqnarray}
is the so-called infinite dimensional Bakry-Emery Ricci curvature associated with $L$ on $(M, g, \mu)$.

By the generalized Bochner formula for the weighted Laplacian (i.e., the Witten Laplacian), we have
\begin{eqnarray}
\Gamma_2(f, f)=\|\nabla^2 f\|_{\rm HS}^2+Ric(L)(\nabla f, \nabla f). 
\end{eqnarray}

For any $m\geq n$, the $m$-dimensional Bakry-Emery Ricci curvature associated with $L$ on $(M, g, \mu)$ is defined by
\begin{eqnarray}
Ric_{m, n}(L)=Ric+\nabla^2 V-{\nabla V\otimes \nabla V\over m-n}.
\end{eqnarray}
By \cite{BE1985}, $(M, g, \mu)$ satisfies the CD$(K, m)$ curvature-dimension condition if and only if
\begin{eqnarray}
Ric_{m, n}(L)\geq K.
\end{eqnarray}

Now we introduce the $W$-entropy on RCD$(K, N)$ space. For $K \in \mathbb{R} $ and $N\geq 1$,  following \cite{Li2012, LiLi2015PJM, LiLi2024TMJ}, let
\begin{equation*}
	H_{N,K}(u):=-\int_X u\log u d\mu-\frac{N}{2}\Big(\log(4\pi et)\Big)-\frac{N}{2}Kt\Big(1+\frac{1}{6}Kt\Big),
\end{equation*}
and we define the $W$-entropy on RCD$(K, N)$ space by
\begin{equation*}
	W_{N,K}(u):=\frac{d}{dt}\Big(tH_{N,K}(u)\Big).
\end{equation*}
In particular, when $K=0$, the $W$-entropy on RCD$^*(0, N)$ space has been introduced in \cite{KL,Li2012} as follows
\begin{equation*}
	W_N(u):=\frac{d}{dt}\Big(tH_N(u)\Big),
\end{equation*}
where 
\begin{eqnarray*}
	H_N(u)=-\int_M u(t)\log u(t)d\mu-\frac{N}{2}\Big(1+\log(4\pi t)\Big).                                                                                                                                                                                                                                                                                                                                                                                                                                                                                                                                                                                                                                                                                                                                                                                                                                                                                                                                                                                                                                                                                                                                                                                                                                                                                                                                                                                                                                                                                                                                                                                                                                                                                                                                                                                                                                                                                                                                                                                                                                                                                                                                                                                                                                                                                                                                                                                                                                                                                                                                                                                                                                                                                                                                                                                                                                                                                                                                                                                                                                                                                                                                                                                                                                                                                                                                                                                                                                                                                                                                                                                                                                                                                                                                                                                                                                                                                                                                                                                                                                                                                                                                                                                                                                                           \end{eqnarray*}
Note that when $N\in \mathbb{N}$, $H_N(u)$ is the difference of the Shannon entropy of the heat kernel measure $u(t)d\mu$ and the Shannon entropy of the Gaussian heat kernel measure $\gamma_N(t)dx$.

In \cite{Li2012}, the first named author of this paper proved the $W$-entropy formula for the heat equation $\partial_t u=Lu$ on complete Riemannian manifolds with CD$(0, m)$-condition, i.e., $Ric_{m, n}(L)\geq 0$.  
More precisely, let $(M, g)$ be a complete Riemannian manifold with bounded geometric condition \footnote{Here, we say that $(M, g)$ satisfies the bounded geometry condition if the Riemannian curvature tensor $\mathrm{Riem}$ and its covariant derivatives $\nabla^k \mathrm{Riem}$ are uniformly bounded on $M$ for $k = 1, 2, 3$.}, $V\in C^4(M)$ with $\nabla V\in C_b^k(M)$ for $k=1, 2, 3$, let 
  \begin{equation*}
	u(x, t)=\frac{e^{-f(x, t)}}{(4\pi t)^{\frac{m}{2}}}
\end{equation*}
be a positive and smooth solution to the heat equation
\begin{equation}
	\partial_t u=Lu
	 \label{Heat2}
\end{equation}
with $\int_M u(x, 0)d\mu(x)=1`$. 
%where 
%\begin{equation}
%Lu=\Delta u-\nabla \phi\cdot\nabla u \label{WittenLaplacian}
%\end{equation}
%is the Witten Laplacian on $(M, g, \mu)$, which is the generator of the Dirichlet form
%\begin{equation}
%\mathcal{E}(u, v)=\int_M \langle \nabla u, \nabla v\rangle d\mu.
%	 \label{Dirichletform}
%\end{equation}
Define the $W$-entropy by
\begin{equation}\label{Wentopy}
	W_m(f, t):=\int_M \left( t|\nabla f |^2+f-m\right)\frac{e^{-f(x, t)}}{(4\pi t)^{\frac{m}{2}}} d\mu.
\end{equation}
Then the  following $W$-entropy formula holds
\begin{eqnarray}
	\frac{d}{dt}W_m(f, t)&=&-2t\int_M \left(\Big\|\nabla^2 f-\frac{g}{2t} \Big\|_{\operatorname{HS}}^2
	+\Ric_{m, n}(L)(\nabla f,\nabla f)\right)u d\mu \nonumber\\
	& &\hskip1cm -{2t\over m-n}\int_M\Big|\nabla V\cdot \nabla f-{m-n\over 2t}\Big|^2 u d\mu, \label{WLi}
\end{eqnarray}
In particular,  the $W$-entropy is non-increasing when $\Ric_{m, n}(L)\geq 0$.

In \cite{Li2012}, the first named author gave a probabilistic interpretation of Perelman's $\mathcal{W}$-entropy for the  Ricci flow. More precisely,  Perelman's $\mathcal{W}$-entropy can be understood as follows:
\begin{equation*}
	\mathcal{W}(u,\tau):=\frac{d}{d\tau}(\tau(H(u)-H(\gamma_n)) ),
\end{equation*}
where $\gamma_n:=\frac{e^{-\frac{|x|^2}{4\tau}}}{(4\pi\tau)^{\frac{n}{2}}}$ , $u=\frac{e^{-f}}{(4\pi\tau)^{\frac{n}{2}}}$ and $H(u)=-\int_M u\log udv$ is the Shannon entropy of the probability measure $udv$. Moreover, 
Li  \cite{LiLi2015PJM,Li2012}  proved the rigidity theorem for the $W$-entropy on complete Riemannian manifolds 
 with $\Ric_{m, n}(L)\geq 0$. 

Subsequently, S. Li and Li \cite{LiLi2015PJM, LiLi2018JFA} proved the $W$-entropy formula for the heat equation $\partial_t u=Lu$ on complete Riemannian manifolds with CD$(K, m)$-condition, i.e., $Ric_{m, n}(L)\geq K$.
More precisely, let $u$ be a positive solution to the heat equation $\partial_t u=L u$ on a complete Riemannian manifold $(M, g, \mu)$ with bounded geometry condition and $V\in C^4(M)$ such that $\nabla V\in C_b^3(M)$. Then
\begin{equation}
\begin{aligned}
\frac{d}{d t} W_{m, K}(u)
&= -2t\int_X \Big[\left\|\nabla^2 \log u+{1\over 2}\left({1\over t}-K\right)g\right\|_{\rm HS}^2+(Ric_{m, n}(L)-Kg)(\nabla \log u, \nabla \log u)\Big]ud\mu\\		
& \hskip2cm -{2t\over m-n}\int_X \Big[\nabla V\cdot\nabla \log u-{m-n\over 2}\left({1\over t}-K\right)\Big]^2ud\mu. \label{WfnKm}
\end{aligned}
\end{equation}
In particular, if $Ric_{m, n}(L)\geq K$, then ${d\over dt}W_{m, K}(u)\leq 0$.

When $m\in \mathbb{N}$ and $m>n$, S. Li and the first named author \cite{LiLi2015PJM} gave  an alternative proof of $(\ref{WLi})$ using Ni's $W$-entropy formula $(\ref{WNi})$ and a warped product metric on $\mathcal{M}=M\times \mathbb{R}^{m-n}$. In \cite{LiLi2015PJM, LiLi2018JFA}, S. Li and Li further proved the 
$W$-entropy formula on the $(0, m)$ and $(K, m)$-super Ricci flows. Furthermore, Lei, S. Li, and Li \cite{LiLei} established the corresponding $W$-entropy formula for the porous medium equation $\partial_tu=\Delta u^p$ with $p>1$.

Very recently, S. Li and the first named author \cite{LiLi2015PJM} proved the $K$-concavity of the Shannon entropy power on complete Riemannian manifolds with $\Ric_{m, n}(L)\geq K$. To save the length of the Introduction, we leave the  background of the entropy power inequality in Section 4 below.

It is natural to ask an interesting question whether one can  extend the monotonicity of $W$-entropy to 
 more singular spaces than smooth Riemannian manifolds. In the literature,  the notion of spaces with a lower bound of  Ricci curvature and an upper bound of dimension has been extended from Riemannian manifolds to metric measure spaces by means of optimal transport \cite{AGS2014,LV2009, Sturm2006}.  To facilitate a comparison between our results and those in the smooth case, we begin by reviewing the Bakry-\'Emery curvature-dimension condition for $L$, that is to say
\begin{equation}\label{BE}
	\frac{1}{2}L|\nabla f |^2 \geq \<\nabla f,\nabla L f\>+K |\nabla f|^2+\frac{1}{N}|L f|^2
\end{equation}
holds for any $f\in C^3(M)$, if and only if 
$$\Ric_{N, n}(L) \geq K$$ and $N\geq n$.
Roughly speaking, $m$ plays the role of (an upper bound of) the dimension of the operator $L$ instead of the dimension of the manifold $n={\rm dim}~M$. 
Indeed, $(\ref{BE})$ can be used as an abstract generalization of the condition ``$\Ric\geq K$ and $n={\rm dim} M\leq {\rm dim}~(L)\leq N$''.

 In \cite{KL}. Kuwada and the first named author proved 
 the monotonicity of $W$-entropy on the so-called RCD$(0, N)$ spaces and provided the associated rigidity results. For its precise definition and statement, see Section 2 and Section 6 below. As far as we know, this is the first result on 
the $W$-entropy and related topics on RCD spaces. 

Motivated by very  increasing interest of the study on the geometry and analysis on RCD spaces, it is natural and interesting to ask a question whether one can extend Kuwada-Li's and S.Li-Li's results to RCD$(K, N)$ spaces. 
The purpose of this paper is to study the $W$-entropy  and the Shannon entropy power associated with the heat equation on RCD$(K, N)$ spaces. 
The main results of this paper are the monotonicity of $W$-entropy and the $K$-concavity of Shannon entropy power for heat equation on RCD$(K, N)$ spaces. See Section 3 below. Moreover, as an application of the main results, we prove the entropy isoperimetric inequality and the Stam type logarithmic Sobolev inequality on RCD$(0, N)$ spaces with maximal volume growth condition.

The structure of this paper is as follows:
In Section 2, we introduce some basic notions related to RCD spaces.
In Section 3, we state the main results of this paper.
In Section 4,  we prove the entropy dissipation equalities and inequality on RCD$(K, N)$ spaces. In Section 5, we prove the $W$-entropy formula and its monotonicity on RCD$(K, N)$ spaces.  In Section 6, we prove the $K$-concavity of the Shannon entropy power on RCD$(K, N)$ spaces. In Section 7, as an application,  we prove the Shannon entropy isoperimetric inequality and the Stam type logarithmic Sobolev inequality on RCD$(0, N)$  spaces with maximal volume growth condition.  In Section 8, we prove the rigidity theorem for the sharp Stam type logarithmic Sobolev inequality on noncollapsing RCD$(0, N)$ spaces. In Section 9, we extend two other entropy formulas to RCD spaces. 
 
 \medskip
 
This work has been done since November 30, 2023 when the first named author wrote two pages' notes in a draft and suggested the second name author to complete it for a part of his PhD thesis. On February 11, 2025, when a draft version of this paper is nearly to be completed, the first named author received an email from Dr. Camillo Brena who asked some questions on the $W$-entropy formula in \cite{Li2012} and told us that he is trying to extend the $W$-entropy formula from Riemannian manifolds to RCD spaces.  The first named author mentioned to Dr. Brena that we are nearly finishing this paper but it is still in revision.  Due to some reasons, we have only finished the revision of this paper until now and we were informed by Dr. Yu-Zhao Wang that Brena's paper has been posed on arxiv. See C. Brena: Perelman's entropy and heat kernel bounds on RCD spaces,  arXiv:2503.03017v1 [math.DG] 4 Mar 2025.  We would like to say that our work is independent of Brena's work and we have not yet read his paper. 

\medskip

\noindent{\bf Acknowledgement}. The authors of this paper would like to express their gratitudes to Prof. Banxian Han, Dr. Siqi Jian, Dr. Songzi Li  and Dr. Yu-Zhao Wang for helpful discussions in the preparation of this paper. The first named author would like to thank Dr. C. Brena for his interest on the $W$-entropy formula on Riemannian manifolds and for valuable discussions on related questions on RCD spaces.

%%%%%%%%%%%%%%%--section-----2--%%%%%%%%%%%%%
\section{Definitions and basic facts of RCD spaces}\label{sect2}
In this section, we briefly recall some  definitions and basic facts of RCD spaces. 

 Let $(X, d, \mu)$ be a metric measure space, which means that $(X, d)$ is 
a complete and separable metric space and $\mu$ is a locally finite measure.
Locally finite means that for all $x \in X$ , there is $r>0$ such that 
$\mu\left(B_r(x)\right)<\infty$ and $\mu$ is a  $\sigma$-finite Borel measure on $X$, where $B_r(x)=\{y\in X, d(x, y)<r\}$.

 Let  $P_2(X, d)$ be the $L^2$-Wasserstein space over $(X, d)$, i.e. the set of all Borel probability measures $\mu$ satisfying
$$
\int_X d\left(x_0, x\right)^2 \mu(\mathrm{d} x)<\infty,
$$
where $x_0 \in X$ is a (and hence any) fixed point in $M$. The $L^2$-Wasserstein distance between $\mu_0, \mu_1 \in P_2(X, d)$ is defined by
$$
W_2\left(\mu_0, \mu_1\right)^2:=\inf_{\pi\in \Pi} \int_{X\times X} d(x, y)^2 \mathrm{~d} \pi(x, y),
$$
where $\Pi$ is the set of coupling measures $\pi$ of $\mu_0$ and $ \mu_1$ on $X\times X$, i.e., $\Pi=\{\pi\in P(X\times X), \pi(\cdot, X)=\mu_0, \pi(X, \cdot)=\mu_1\}$, where $P(X\times X)$ is the set of probability measures on $X\times X$.

Fix a reference measure $\mu$ on $(X, d)$, let  $P_2(X, d, \mu)$ be the subspace of  all absolutely continuous measures  with respect to the measure $\mu$.  For any given measure $\nu \in P_2(X, d)$, we can define the relative entropy with respect to $\mu$ as 
$$
\operatorname{Ent}(\nu):=\int_X \rho \log \rho \mathrm{d} \mu,
$$
if $\nu=\rho \mu$ is absolutely continuous w.r.t.\;$\mu$ and $(\rho \log \rho)_{+}$ is integrable  w.r.t.\;$\mu$,  otherwise we set $\operatorname{Ent}(\nu)=+\infty$.
 The Fisher information is defined by
$$
I(\nu):= \begin{cases}\int_X \frac{|D \rho|^2}{\rho} \mathrm{d}\mu & \text { if } \nu=\rho \mu, \\ \infty & \text { otherwise. }\end{cases}
$$
 
Given $N \in$ $(0, \infty)$, Ebar, Kuwada and Sturm \cite{EKS2015} introduced the functional $U_N: P_2(X, d) \rightarrow[0, \infty]$ 

$$
U_N(\nu):=\exp \left(-\frac{1}{N} \operatorname{Ent}(\nu)\right),
$$
which is similar to the Shannon entropy power \cite{Sh}.

 We now follow Bacher  and Sturm \cite{BS2010} and Ambrosio-Gigli-Savar\'e   \cite{AGS2014Invent} to introduce 
  the definition of CD$^*(K, N)$ and RCD$^*(K, N)$ spaces below. Let $P_{\infty}(X, d, \mu)$ be the set of measures in $P_2(X, d, \mu)$ with bounded support.
 
\begin{definition}\cite{EKS2015}\label{Def}
 	For $\kappa \in \mathbb{R}$, and $\theta \geq 0 $ we define the function
 	\begin{equation*}
\mathfrak{s}_\kappa(\theta)= \begin{cases}\frac{1}{\sqrt{\kappa}} \sin (\sqrt{\kappa} \theta), & \kappa>0 ,\\ \theta, & \kappa=0 ,\\ \frac{1}{\sqrt{-\kappa}} \sinh (\sqrt{-\kappa} \theta), & \kappa<0.\end{cases}
\end{equation*}
\begin{equation*}
\mathfrak{c}_\kappa(\theta)= \begin{cases}\cos (\sqrt{\kappa} \theta), & \kappa \geq 0 ,\\ \cosh (\sqrt{-\kappa} \theta), & \kappa<0.\end{cases}
\end{equation*}
Moreover, for $t\in [0,1]$ we set
\begin{equation*}
\sigma_\kappa^{(t)}(\theta)= \begin{cases}\frac{\mathfrak{s}_\kappa(t \theta)}{\mathfrak{s}_\kappa(\theta)}, & \kappa \theta^2 \neq 0 \text { and } \kappa \theta^2<\pi^2 ,\\ t, & \kappa \theta^2=0 ,\\ +\infty, & \kappa \theta^2 \geq \pi^2.\end{cases}
\end{equation*}
 \end{definition} 
 
\begin{definition}[\cite{BS2010}]
	   We say that metric measure space $(X,d,\mu)$ satisfies the reduced curvature-dimension condition CD$^\ast(K, N)$ if and only if for each pair $\mu_0=\rho_0 \mu, \mu_1=\rho_1 \mu \in P_{\infty}(X, d, \mu)$, there exists an optimal coupling $\pi$ of $\mu_0$ and $\mu_1$ such that
\begin{equation}\label{CD}
\begin{aligned}
\int \rho_t^{-\frac{1}{N^{\prime}}} \mathrm{d} \mu_t \geq & \int_{X \times X}\left[\sigma_{K / N^{\prime}}^{(1-t)}\left(d\left(x_0, x_1\right)\right) \rho_0\left(x_0\right)^{-\frac{1}{N^{\prime}}}\right. \\
& \left.+\sigma_{K / N^{\prime}}^{(t)}\left(d\left(x_0, x_1\right)\right) \rho_1\left(x_1\right)^{-\frac{1}{N^{\prime}}}\right] \mathrm{d} \pi\left(x_0, x_1\right),
\end{aligned}
\end{equation}
 where  $\left(\mu_t\right)_{t \in[0,1]}$ in $P_{\infty}(X, d, \mu)$ is a geodesic connecting $\mu_0$ and $\mu_1$ and $N^{\prime} \geq N$.
If inequality \eqref{CD} holds for any geodesic $\left(\mu_t\right)_{t \in[0,1]}$ in $P_{\infty}(X, d, \mu)$,  we say that $(X, d, \mu)$ is a strong CD$^*(K, N)$ space.
\end{definition}

  To introduce the RCD spaces and  consider the canonical heat flow on $(X, d, \mu)$, we need several  notions  including the Cheeger energy functional.

\begin{definition}[minimal relaxed gradient\cite{AGS2014Invent}]
		We say that $G \in L^2(X, \mu)$ is a relaxed gradient of $f \in L^2(X, \mu)$ if there exist Borel d-Lipschitz functions $f_n \in L^2(X, \mu)$ such that: 
		
		(a) $f_n \rightarrow f$ in $L^2(X, \mu)$ and $|Df_n|$ weakly converge to  $\tilde{G}$ in $L^2(X, \mu)$; 
		
		(b) $\tilde{G} \leq G $. m-a.e. in X.  We say that G is the minimal relaxed gradient of f if its $L^2(X, \mu)$ norm is minimal among relaxed gradients. 
		
 We use $|Df |_{\ast} $ to denote  the minimal relaxed gradient.
\end{definition}
%	\begin{definition}[minimal weak upper gradient\cite{AGS2014Invent}]
%		Given $f : X \rightarrow \mathbb{R}$, a m-measurable function $G : X \rightarrow [0, \infty]$ is a weak upper gradient of  $f$ if 	the following estimte holds
%		 $$ |f(\gamma_s)-f(\gamma_s)|\leq \int_s^t  G(\gamma_r)|\dot{\gamma_r}|dr.$$
%		
%		The minimal weak upper gradient is defined by $|\nabla f |_w:=\inf_{\gamma_t}{G}$.
%		\end{definition}
 Ambrosio et al \cite{AGS2014Invent} proved that $|Df |_{\ast}=|\nabla f |_w , \mu-a.e$ where  $|\nabla f |_w $ denotes the so called minimal weak upper
gradient of $f$ (cf \cite{AGS2014Invent})

The Cheeger energy functional \cite{EKS2015} is  defined by  
\begin{equation*}
	\operatorname{Ch}(f):=\int_X |\nabla f |^2_w d\mu,
\end{equation*}
  and inner product is given by $$
\langle\nabla f, \nabla g\rangle:=\lim _{\varepsilon \searrow 0} \frac{1}{2 \varepsilon}\left(|\nabla(f+\varepsilon g)|_w^2-|\nabla f|_w^2\right).
$$
  We now have a strongly local Dirichlet form $(\mathcal{E}, D(\mathcal{E}))$ on $L^2(X, \mu)$ by setting $\mathcal{E}(f, f)=\operatorname{Ch}(f)$ and $D(\mathcal{E})=W^{1,2}(X, d, \mu)$ being a Hilbert space 
  and $L^2$-Lipschitz 
functions are dense in the usual sense.  In this case, $\mathrm{H}_t$ is a semigroup of the self-adjoined linear operator on $L^2(X, \mu)$ with the Laplacian $\Delta$ as its generator. The previous result implies that for $f, g \in W^{1,2}(X, d, \mu)$,  the Dirichlet form is defined by
$$
\mathcal{E}(f, g):=\int_X\langle\nabla f, \nabla g\rangle \mathrm{d} \mu.
$$
Moreover, for $f \in W^{1,2}$ and $g \in D(\Delta)$, the integration by parts formula holds
$$
\int_X\langle\nabla f, \nabla g\rangle \mathrm{d} \mu=-\int_X f \Delta g \mathrm{~d} \mu.
$$

Ambrosio et al \cite{AGS2014} proved that the Cheeger energy ${\rm Ch}$ is quadratic
 is equivalent to the linearity of the heat semigroup $\mathrm{H}_t$ defined by solving the 
 heat equation below:
\begin{equation*}
	\frac{\partial}{\partial t}u=\Delta u, \ \ \  u(0)=f.
\end{equation*}

%
%For any $f, g\in D(\Delta)\cap W^{1, 2}(X, \mu)$ with $\Delta f, \Delta g\in W^{1, 2}(X, \mu)$, the iterated carr\'e du champs {\bf measure} is defined  by 
%		\begin{eqnarray*}
%	{\bf \Gamma_2}(f,  g):={1\over 2}\Delta \langle \nabla f, \nabla g\rangle-{1\over 2}\left(\langle\nabla f, \nabla \Delta g\rangle+\langle\nabla g, \nabla \Delta f\rangle\right)\mu.
%	\end{eqnarray*}
%	In particular,  we have
%	\begin{eqnarray*}
%	{\bf \Gamma_2}(f, f):={1\over 2}\Delta |\nabla f|^2-\langle\nabla f, \nabla \Delta f\rangle 
%	\mu.
%	\end{eqnarray*}

\begin{definition}[\cite{EKS2015, KL}]
	We say that a metric measure space $(X, d, \mu)$ is infinitesimally Hilbertian if the associated
	 Cheeger energy is quadratic. Moreover,
we call $(X, d, \mu)$ an RCD$^*(K, N)$ space if it satisfies the Riemannian Curvature-Dimension condition CD$^*(K, N)$ and satisfies infinitesimally Hilbertian condition. 	
\end{definition}

By \cite{AGS2014,EKS2015,KL},  for infinitesimally Hilbertian $(X, d, \mu)$, CD$^*(K, N)$  condition is equivalent to the following  conditions:

(i) There exists $C>0$ and $x_0 \in X$ such that
\begin{equation*}
\int_X \mathrm{e}^{-C d\left(x_0, x\right)^2} \mu(\mathrm{~d} x)<\infty.
\end{equation*}

(ii) For $f \in \mathcal{D}(\mathrm{Ch})$ with $|\nabla f|_w \leq 1 \quad \mu$-a.e.

(iii) For all $f \in \mathcal{D}(\Delta)$ with $\Delta f \in \mathcal{D}(\Delta)$ and $g \in \mathcal{D}(\Delta) \cap L^{\infty}(\mu)$ with $g \geq 0$ and $\Delta g \in L^{\infty}(\mu)$
$$
\begin{aligned}
& \frac{1}{2} \int_X|\nabla f|_w^2 \Delta g \mathrm{~d} \mu-\int_X\langle \nabla f, \nabla \Delta f\rangle g \mathrm{d}\mu 
 \geq K \int_X|\nabla f|_w^2 g \mathrm{d}\mu+\frac{1}{N} \int_X(\Delta f)^2 g \mathrm{d}\mu.
\end{aligned}
$$

We now give the following important examples of RCD$^\ast(K,N)$ space:
\medskip

\begin{itemize}

\item Let $\left(M^n, g\right)$ be a complete Riemannian manifold, $f: M \rightarrow \mathbb{R}$ a $C^2(M)$ function, $d_g$ the Riemannian distance function, and vol $g_g$ the Riemannian volume measure on $M$. Set $\mathfrak{m}:=e^{-f}$ vol $_g$. Then the metric measure space ( $M, d_g, \mathfrak{m}$ ) satisfies $\operatorname{RCD}(K, N)$ condition for $N>n$ if and only if 
$$
\operatorname{Ric}_N:=\operatorname{Ric}_g+\operatorname{Hess}_f-\frac{d f \otimes d f}{N-n} \geq K g
$$
holds. For $N=n$, the $\operatorname{RCD}(K, n)$ condition is equivalent to $d f=0$ and $\operatorname{Ric}_g \geq K$.

\item Let $\left\{\left(X_i, d_i, \mathfrak{m}_i\right)\right\}_i$ be a family of $\operatorname{RCD}^*\left(K_i, N\right)$ spaces. For $x_i \in X_i$, assume $\mathfrak{m}_i\left(B_1\left(x_i\right)\right)=1, K_i \rightarrow K$ and $\left(X_i, d_i, \mathfrak{m}_i, x_i\right) \xrightarrow{p m G}\left(X_{\infty}, d_{\infty}, \mathfrak{m}_{\infty}, x_{\infty}\right)$ as $i \rightarrow \infty$, where $\xrightarrow{p m G}$ means the pointed measured Gromov convergence (see \cite{Gigili2015} ). Then $\left(X_{\infty}, d_{\infty}, \mathfrak{m}_{\infty}\right)$ satisfies the $\operatorname{RCD}^*(K, N)$ condition. Moreover a family of $\mathrm{RCD}^*(K, N)$ spaces with the normalized measures is precompact with respect to the pmG-convergence.

\end{itemize}

  We now explain some basic results on RCD spaces. 
  For $f, g \in \mathcal{D}(\Delta)\cap \mathcal{L}^{\infty}(\mu)$ and $\varphi \in C^1(\mathbb{R})$ with $\varphi(0)=0$, we have $\varphi(f) \in \mathcal{D}(\Delta) \cap L^{\infty}(\mu)$ and  the following chain rule \eqref{nabla} (see \cite{FOT} ) and the Leibniz rule \eqref{Delta} for the Laplacian (see \cite{Gigili2015MAMS} ) hold :
\begin{equation}\label{nabla}
\begin{aligned}
	\langle \nabla \varphi(f), \nabla g\rangle &=\varphi^{\prime}(f)\langle \nabla f, \nabla g\rangle \quad \mu \text {-a.e. }\\
	\Delta(\phi(g) )& =\phi'(g) \Delta g +\phi''(g) | \nabla g |_w\quad \mu \text {-a.e. }
\end{aligned}
\end{equation}
\begin{equation}\label{Delta}
	\Delta (f\cdot g)
= f \Delta g +g \Delta f +2 <\nabla f, \nabla g> 
\end{equation}

By  Cavaletti-Milman \cite{CavMim} and Z. Li \cite{LZH}, the notion of RCD$^*(K, N)$ space is indeed equivalent to the one of RCD$(K, N)$ space.  So we will only say RCD$(K, N)$ space throughout this paper. 
 
 Ambrosio et al. \cite{AGS2014} proved that for $\mu \in \mathcal{P}_2(X), t \mapsto \operatorname{Ent}\left(P_t \mu\right)$ is absolutely continuous on $(0, \infty)$ and $\mu_t=P_t \mu$ satisfies the energy dissipation identity, i.e. $\mu_t \rightarrow \mu_0$ as $t \rightarrow 0$ and for $0<s<t$,
\begin{equation}\label{EnDI}
\operatorname{Ent}\left(\mu_s\right)=\operatorname{Ent}\left(\mu_t\right)+\frac{1}{2} \int_s^t\left|\dot{\mu}_r\right|^2 \mathrm{~d} r+\frac{1}{2} \int_s^t I\left(\mu_r\right) \mathrm{d} r  \text { a.e. } t.
\end{equation}
 The energy dissipation identity \eqref{EnDI} is equivalent to the following equality
\begin{equation*}
-\frac{\mathrm{d}}{\mathrm{~d} t} \operatorname{Ent}\left(\mu_t\right)=\left|\dot{\mu}_t\right|^2=I\left(\mu_t\right)<\infty \quad \text { a.e.}\quad t. 
\end{equation*}

To end this section, let us briefly summarize the history of the curvature-dimension condition on metric measure spaces. 
As pointed out in \cite{EKS2015},   Bakry and Emery \cite{BE1985} used the inequality 
\eqref{BE}  to introduce the definition of the  
energetic curvature-dimension, denoted as BE$(K,N)$, to 
 study Dirichlet forms and 
the corresponding Markov semigroup via the {\it carr\'e du champ operator}. The 
inequality \eqref{BE} leads to a variety of elegant results concerning entropy in the context of optimal transport. 
However, defining these inequalities directly in general metric measure spaces poses significant difficulties. 
To this end, many mathematicians have focused on exploring the concept of 
curvature in non-smooth spaces. 
The curvature-dimension condition CD$(K,N)$ on metric measure spaces was 
first introduced by K.-T. Sturm \cite{Sturm2006}, where $K$ represents the lower bound of Ricci curvature and $N$ the upper bound of the dimension. On the other hand,  the CD$(K,\infty)$ condition on metric measure 
spaces was independently introduced by Sturm \cite{Sturm2006} and Lott-Villani \cite{LV2009}. However, because the CD$(K,N)$ condition lacks the local-to-global property, Bacher and Sturm \cite{BS2010} proposed  the reduced curvature-dimension condition, denoted by CD$^\ast(K,N)$,  to overcome this shortcoming.  
Ambrosio et al. \cite{AGS2014Invent,AGS2014, AGS2015} further introduced the concept of infinitesimally Hilbertian, which assumes that the 
heat flow is linear. On metric measure space with the  infinitesimally Hilbertian assumption, Ambrosio et al \cite{AGS2014} proved the equivalence of 
CD$(K,\infty)$ and BE$(K,\infty)$ for $K\in \mathbb{R}$, Ebar, Kuwada and Sturm \cite{EKS2015} 
further established the equivalence of CD$^*(K,N)$ and BE$(K,N)$ for $K\in \mathbb{R}$ and $N\in [1, \infty)$ along with the corresponding weak Bochner inequality. This leads the definition of
RCD$^*(K, N)$ space as a metric measure space that satisfies both the infinitesimally Hilbertian condition and the CD$^*(K, N)$ condition.

%%%%%%%%%%%%%%%--section-----3--%%%%%%%%%%%%%

\section{ Statement of main results}

We now state the main results of this paper. Our first result is the following entropy dissipation equalities and inequality on RCD$(K, N)$ spaces.

\begin{theorem}\label{1HH2}
	Let $(X, d, \mu)$ be an 
 RCD space and $u$ be a solution to the heat equation $\partial_t u=\Delta u$. Let $H(u)=-\int_X u \log ud\mu$ be the Shannon entropy. Let $u^\delta$ be the regularization of $u$ which will be precisely defined in Section 4 below. 
 We have
 
 (1) Let
 $$H_\varepsilon(u)=-\int e_\varepsilon(u)d\mu,$$ 
 where $e_\varepsilon :[0,\infty)\rightarrow \mathbb{R}$ with $e'_\varepsilon(r)=\log(\varepsilon+r)+1$ and
		  $e_\varepsilon(0)=0$.
  %$$H_{\varepsilon, \delta}(u_t^\delta)=-\int_X u_t^\delta \log u_t^\delta d\mu.$$
 Suppose that there exists a $\delta_0>0$ such that
  \begin{equation}\label{C0}
 \int_X \sup_{\delta \in (0, \delta_0)} u_t^\delta |\log u_t^\delta| d\mu \leq \infty.
 \end{equation}
 Then $H_{\varepsilon, \delta}(u_t))$ is continuous in $\varepsilon$ and $\delta$, and 
 \begin{equation}\label{H0}
 	\begin{aligned}
 		\lim_{\delta\rightarrow 0}\lim_{\varepsilon\rightarrow 0}H_\varepsilon(u_t^\delta)
 	=H(u_t).
 	\end{aligned}
 \end{equation} 
 
 (2) Suppose that there exists a $\delta_0>0$ such that
	\begin{equation}\label{C1}
		\begin{aligned}
			\int_X \sup_{\delta\in (0, \delta_0)} \frac{|\nabla u_t^\delta |^2}{u_t^\delta} d\mu< \infty.
		\end{aligned}
	\end{equation}
Then $ {d\over dt}H_{\varepsilon, \delta}(u_t))$ is continuous in $\varepsilon$ and $\delta$, and 
the  first order entropy dissipation formula holds
	\begin{equation}
	\label{H1}
	\frac{d}{dt}H(u)=\lim_{\delta\rightarrow 0}\lim_{\varepsilon\rightarrow 0}{d\over dt}H_\varepsilon(u_t^\delta)=-\int_X \log u
\Delta u d\mu=-\int_X \Delta \log u ud\mu,
\end{equation}

(3)	Suppose that there exists a $\delta_0>0$ such that
\begin{equation}\label{C4}
	\begin{aligned}
		\int_X \sup_{\delta\in (0, \delta_0)}\Big[\frac{|\nabla u^\delta |^2}{u^\delta}+ \frac{|\Delta u^\delta|^2}{u^\delta}+\frac{|\nabla \Delta u^\delta |^2 }{u^\delta}\Big]d\mu<\infty.
	\end{aligned}
\end{equation}
Then the second  order entropy dissipation formula holds 
	\begin{equation}
	\frac{d^2}{dt^2}H(u)=\lim_{\delta\rightarrow 0}\lim_{\varepsilon\rightarrow 0}{d^2\over dt^2}H_\varepsilon(u_t^\delta)=-2\int_X {\bf \Gamma_2}( \log u, \log u)ud\mu,\label{HH2}
\end{equation}
and the following second order entropy dissipation inequality holds
		\begin{eqnarray}
		\frac{d^2}{dt^2}H(u)
		\leq -\frac{2}{N} \int_X u\Big(\Delta \log u\Big)^2d\mu-2K\int_X u |\nabla \log u|_w^2 d\mu.\label{H2}
	\end{eqnarray}

\end{theorem}

\begin{remark} In the case $X=M$ is a compact Riemannian manifold, Bakry and \'Emery \cite{BE1985} proved that for any positive solution to the heat equation $\partial_t u=Lu$, the  entropy dissipation formulas hold
	\begin{equation}\label{1H1}
		\begin{split}
			\frac{d}{dt}H(u)
			=\int_M |\nabla\log u |^2u d\mu,
\end{split}		
	\end{equation}
	and
	\begin{equation}\label{2H2}
		\begin{split}
			\frac{d^2}{dt^2}H(u)
			=-2\int_M [\|\nabla^2 \log u \|_{\rm HS}^2+\Ric(L)(\nabla \log u , \nabla \log u )] 
			ud\mu.
\end{split}		
	\end{equation}

	In the case $X=M$ is a complete Riemannian manifold, the first author \cite{Li2012} proved the entopy dissipation formulas in Theorem 4.1 on complete Riemannian manifolds with bounded geometry condition. More precisely, the result is stated as follows: Let $(M, g)$ be a complete Riemannian manifold with bounded geometry condition, and $\phi \in C^2(M)$. Let $u$ be a positive solution to the heat equation $\partial_t u=L u$ satisfying the integrability condition

\begin{equation}\label{L1}
	\int_M\left[\frac{|\nabla u|^2}{u}+\frac{|L u|^2}{u}+\frac{|\nabla L u|^2}{u}\right] d \mu(x)<+\infty
\end{equation}
%
%
%Let
%
%$$
%H(u(t))=-\int_M u \log u d \mu
%$$
%
Then entropy dissipation formulas $(\ref{1H1})$ and $(\ref{2H2})$ hold. 
%
%\begin{equation}\label{L2}
%	\begin{aligned}
%\frac{d}{d t} H(u(t)) & =\int_M|\nabla \log u|^2 u d \mu \\
%\frac{d^2}{d t^2} H(u(t)) & =-2 \int_M \Gamma_2(\nabla \log u, \nabla \log u) u d \mu
%\end{aligned}
%\end{equation}
%
%where
%
%\begin{equation*}
%	\Gamma_2(\nabla \log u, \nabla \log u)=\left\|\nabla^2 \log u\right\|_{\mathrm{HS}}^2+\operatorname{Ric}(L)(\nabla \log u, \nabla \log u) .
%\end{equation*}
In \cite{Li2012, Li2016SPA}, Li proved that the fundamental solution to the heat equation $\partial_t u=L u$  satisfies 
the condition $(\ref{L1})$ on any complete Riemannian manifold 
	with the $CD(K, m)$-condition, i.e., $Ric_{m, n}(L)\geq K$ for $m\geq n$ and 
	$K\in \mathbb{R}$.  Thus, $(\ref{1H1})$ and $(\ref{2H2})$ hold for the fundamental solution to the heat equation $\partial_t u=L u$ satisfies the condition \eqref{L1} on any complete Riemannian manifold on which the Riemannian curvature tensor as well as its $k$-th order covariant derivatives and $\phi \in C^4(M)$ with $\nabla \phi \in C_b^3(M)$. Recently, it has been further noticed by S. Li and Li \cite{LiLi2024TMJ} that for any positive solution to the heat equation $\partial_t u=L u$, the condition $(\ref{L1})$ holds on any complete Riemannian manifold with bounded geometry condition, and $\phi \in C^4(M)$ with $\nabla \phi \in C_b^3(M)$.  Hence the entropy dissipation formulas $(\ref{1H1})$ and $(\ref{2H2})$ hold for any positive solution to the heat equation $\partial_t u=L u$ on complete Riemannian manifolds with the above bounded geometry condition. 
\end{remark}
%\begin{remark} 
%In the case $X$ is a Riemannian manifold, i.e., $X=M$, we have (\cite{BE1985})
%		
%		
%	\begin{eqnarray}\label{Bochneridentity}
%	\Gamma_2(g_t^\delta,  g_t^\delta)=\|\nabla^2 g_t^\delta\|_{\rm HS}^2+\Ric(\nabla g_t^\delta, \nabla g_t^\delta).
%	\end{eqnarray}
%In this case, it holds
%\begin{equation}\label{d2t2H}
%		\begin{split}
%			\frac{d^2}{dt^2}H_\varepsilon(P_t\mu_\delta)
%			=-2\int_X[\|\nabla^2 g_t^\delta \|_{\rm HS}^2+\Ric(\nabla g_t^\delta, \nabla g_t^\delta)] u^\delta d\mu-\varepsilon \int_X |\Delta g_t^\delta|^2 d\mu,
%\end{split}		
%	\end{equation}
%	In the case $X=M$ is a compact Riemannian manifold, Bakry and \'Emery proved that 
%	\begin{equation}\label{1H1}
%		\begin{split}
%			\frac{d}{dt}H(u_t)
%			=\int_M |\nabla\log u_t |^2u_t d\mu,
%\end{split}		
%	\end{equation}
%	and
%	\begin{equation}\label{2H2}
%		\begin{split}
%			\frac{d^2}{dt^2}H(u_t)
%			=-2\int_M [\|\nabla^2 \log u_t \|_{\rm HS}^2+\Ric(\nabla \log u_t , \nabla \log u_t )] 
%			u_t d\mu.
%\end{split}		
%	\end{equation}
%	In \cite{Li2016SPA}, Li proved that $(\ref{1H1})$ holds  on any complete Riemannian manifold 
%	with the $CD(K, m)$-condition, i.e., $Ric_{m, n}(L)\geq K$ for $m\geq n$ and 
%	$K\in \mathbb{R}$.  Moreover,  Li \cite{Li2016SPA} proved that $(\ref{2H2})$ also holds on 
%	any complete Riemannian manifold with reasonable  bounded geometry condition on 
%	the Riemannian curvature tensor as well as its $k$-th order covariant derivatives. For details, see \cite{Li2016SPA}.  
%\end{remark}

\begin{remark} By \cite{JLZ},  the fundamental solution to the heat equation $\partial_t u=\Delta u$ satisfies the following two sides estimates on RCD$(-K, N)$ space: Let  $(X, d, \mu)$ be an RCD$
(-K,N)$ space with $K\geq 0$ and $N\in [1, \infty)$. Given
any $\varepsilon>0$, there exist positive constants $C_1(\varepsilon)$ and $C_2(\varepsilon)$, 
depending also on $K$ and $N$, such that for all $x, y\in X$ and $t>0$, it holds
 \begin{equation*}\label{JLZB}
 {1\over C_1(\varepsilon) V_x(\sqrt{t}) }\exp\left({-\frac{d^2(x, y)} {(4-\epsilon)t}-C_2(\varepsilon) t } \right)\leq p_t(x, y)\leq 
 { C_1(\varepsilon)\over V_x(\sqrt{t}) }\exp\left({-\frac{d^2(x, y)} {(4+\epsilon)t}+C_2(\varepsilon) t } \right),
 \end{equation*}
 where $V_x(\sqrt{t})=\mu(B(x, \sqrt{t})$ is the volume of the ball $B(x, \sqrt{t})=\{y\in X: d(x, y)\leq \sqrt{t}\}$. 
 This yields
 \begin{equation*}
 \log p_t(x, y)\geq 
 -\log C_1(\varepsilon) -C_2(\varepsilon) t -\frac{d^2(x, y)}{(4-\varepsilon)t}-\log V_x(\sqrt{t}),
 \end{equation*}
and

 \begin{equation*}
 \log p_t(x, y)\leq 
 \log C_1(\varepsilon) +C_2(\varepsilon) t -\frac{d^2(x, y)}{(4+\varepsilon)t}-\log V_x(\sqrt{t}),
 \end{equation*}
 Therefore
 \begin{equation*}
| \log p_t(x, y)|\leq  \log C_1(\varepsilon)+C_2(\varepsilon) t +\frac{d^2(x, y)}{(4-\varepsilon)t}+\log V_x(\sqrt{t}).
 \end{equation*}
For any fixed $x\in X$ and $t>0$, we have
 \begin{equation*}
 \int_X d^2(x, y)p_t(x, y)d\mu(y)<\infty.
 \end{equation*}   
Hence, for any fixed $x\in M$ and $t>0$, the fundamental solution $p_t(x, \cdot)$ on RCD$(K, N)$ space satisfies the  condition \eqref{H0} required in Theorem \ref{1HH2}. On the other hand, the Li-Yau Harnack inequality on RCD$(K, N)$ (see \cite{zhuxiping} and  Lemma \ref{zhulemma} below) says that, for every $\alpha>1$, it holds
\begin{equation*}
 \int_X {|\nabla u|^2\over u}d\mu\leq \int_X\left[\alpha{\partial_t u\over u} +C_{N, K, \alpha, t} \right]u\mu<\infty.
 \end{equation*}
 where $C_{K, N, \alpha, t}=\left(1+{2Kt\over a(\alpha-1)}\right){N\alpha^2\over 2t}$. 
 Hence the condition $(\ref{C1})$ in Theorem \ref{1HH2} is also verified on RCD$(K, N)$ space. It remains only to verify the condition 
 \begin{equation*}\label{C4B}
	\begin{aligned}
		\int_X \Big[\frac{|\Delta u|^2}{u}+\frac{|\nabla \Delta u |^2 }{u}\Big]d\mu<\infty
	\end{aligned}
\end{equation*}
to ensure the validity of the second order entropy dissipation formula in Theorem \ref{1HH2}. This has been proved for the fundamental solution of the heat equation $\partial_t u=Lu$ in \cite{Li2016SPA} on complete Riemannian manifolds with bounded geometry condition and CD$(-K, N)$ condition, and extended to all  positive solution of the heat equation $\partial_t u=Lu$ in \cite{Li2016SPA}.  
 
\end{remark}

The next theorem extends the $W$-entropy formula from Riemannian manifolds to RCD$(K, N)$ spaces. 
\begin{theorem} \label{WNK0} 
Let $(X, d, \mu)$ be a metric measure space satisfying the RCD$(K, N)$-condition. Let $u$ be a positive solution to the heat equation $\partial_t u=\Delta u$ which satisfies the condition as required in Theorem \ref{1HH2}. Then
\begin{equation}
\begin{aligned}
\frac{d}{d t} W_{N, K}(u)
= -2t\int_X \Big[{\bf \Gamma_2}(\log u, \log u )+\left({1\over t}-K\right) \Delta \log u+{N\over 4}\Big({1\over t}-K \Big)^2-K|\nabla \log u |_w^2\Big]ud\mu. \label{W-Gamma2}
\end{aligned}
\end{equation}
%where for any given $a, K\in \mathbb{R}$
%\begin{equation*}
%\begin{aligned}
%\Gamma_{2, a, K, N}(f, f):=\Gamma_2(f, f)+\left(a-K\right) \Delta f+{1\over 4}\Big(a-K \Big)^2N. 
%\end{aligned}
%\end{equation*}
Moreover, it holds
\begin{equation*}
	\frac{d}{dt}W_{N,K}(u)\leq -\frac{2t}{N}\int_X \Big[\Delta \log u+\left({N\over 2t}-{NK\over 2}\right)\Big]^2u d\mu.
\end{equation*}
In particular,  ${d\over dt}W_{N,K}(u)\leq0$, and if ${d\over dt}W_{N,K} (u)=0$ holds at some $t>0$ then 
\begin{equation*}
	\Delta \log u+{N\over 2t}-{NK\over 2}=0.
	\end{equation*}
\end{theorem}

When $K=0$, Theorem \ref{WNK0} is a natural extension and improvement of the  following result due to 
Kuwada and Li \cite{KL}. The proof of Theorem \ref{WNK0} is different from the proof of Theorem \ref{KL0} given in Kuwada-Li  \cite{KL}. 

\begin{theorem} [\cite{KL}]\label{KL0}
Let $(X, d, \mu)$ be a metric measure space satisfying the 
 RCD$(0, N)$-condition. Then
\begin{equation*}
	\frac{d}{dt}W_N(u)\leq 0.
\end{equation*}
\end{theorem}

In order to extend the $W$-entropy formulas (see Ni \cite{Ni1},  Li \cite{Li2012} and S. Li-Li \cite{LiLi2015PJM, LiLi2018JFA, LiLi2018SCM}) from complete Riemannian manifolds to RCD spaces, we introduce the following definition. Recall that in a series of papers of Gigli and his collaborators, the Hessian of a nice function can be defined in the sense of distribution on RCD spaces and the notion of the Ricci curvature measure has been introduced by Gigli et al \cite{Gigili2015MAMS, BGZ, BG, Han2018A, Han2018B}, the notion of the local geometric dimension $n$ of RCD spaces  have been by Han \cite{ Han2018A, Han2018B}. In \cite{BS}, Bru\`e and Semola proved that the local geometric dimension $n$ is indeed a global constant. In this paper, we use $n$ to denote the global geometric dimension of an RCD space $(X, d, \mu)$. 

For the convenience of the readers, we briefly describe the definitions and notations here. Following \cite{Gigili2015MAMS, BGZ, BG, Han2018A, Han2018B},  the tangent module $L^2(TX)$ and the cotangent module $L^2(T^*X)$ of an RCD$(K, N)$ space $(X, d, \mu)$ have been defined as $L^2$-normed modules. The pointwise inner product $\langle\cdot, \cdot\rangle: L^2(T^*X)\times L^2(T^*X)\rightarrow L^1(X) $ is defined by
$$ \langle df, dg\rangle={1\over 4}\left(|\nabla(f+g)|^2-|\nabla(f-g)|^2)\right)$$
for all $f, g\in W^{1, 2}(X)$. For any $g\in W^{1, 2}(X)$, its gradient $\nabla g$ is the unique 
element in $L^2(TX)$ such that
$$\nabla g(df)=\langle df, dg\rangle,\ \ \ \mu-a.e. $$
for all $f\in W^{1, 2}(X)$. Therefore, $L^2(TX)$ inherits a pointwise inner product $\langle\cdot, \cdot\rangle$ from the above inner product $\langle\cdot, \cdot\rangle$  on $L^2(T^*X)$.  To keep the standard notation as used  in  Riemannian geometry, we  use  ${\bf g}$ to denote this inner product $\langle\cdot, \cdot\rangle$ on $L^2(TX)$. 
%
%The notion of local dimension $n$ of an RCD space $(X, d, \mu)$ is introduced in \cite{Han2018A, Han2018B} as follows: We say that $L^2(TM)$ is finitely generated if there is a finite family $v_1, ..., v_n$
%spanning $L^2(TM)$ on $(X, d, \mu)$, and locally finitely generated if there is a partition $\{E_i\}$ of $X$ such that 
%$\left.L^2(TM)\right|_{E_i}$ is finitely finitely generated for every $i \in \mathbb{N}$. If $L^2(TM)$ has a basis of cardinality $n$ on a Borel set $A\subset X$, we say that it has dimension $n$ on $A$,
%or that its local dimension on $A$ is $n$. 
%From now on, we use $n$ to denote the global geometric dimension of an RCD space $(X, d, \mu)$. See \cite{Han2018A, Han2018B, BS}. 
%

The Hessian of a nice function $f$, denoted by ${\rm H}_f$ in \cite{Gigili2015MAMS, Han2018A, Han2018B} and  denoted by $\nabla^2f$ in this paper for keeping the standard notation as in Riemannian geometry, is defined as  the unique bilinear map
$$\nabla^2 f={\rm H}_f: \{\nabla g: g\in {\rm Test} F(X)\}^2\mapsto L^0(X)$$
such that
$$2\nabla^2 f(\nabla g, \nabla h)=\langle \nabla g, \nabla\langle \nabla f, \nabla h\rangle\rangle+
\langle \nabla h, \nabla\langle \nabla f, \nabla g\rangle\rangle
-\langle \nabla f, \nabla\langle \nabla g, \nabla h\rangle\rangle
$$
for any $g, h\in {\rm Test} F(X)$, where 
${\rm Test} F(X)=\{f\in D(\Delta)\cap L^\infty: |\nabla f|\in L^\infty, \Delta f \in W^{1, 2}(X)\}$ is the space of test functions.  Note that
\begin{eqnarray*}
\nabla^2 f(\nabla f, \nabla g)={1\over 2}\langle \nabla |\nabla f|^2, \nabla g\rangle.
\end{eqnarray*}

In \cite{Gigili2015MAMS}, Gigli defined measure valued Ricci tensor on RCD metric measure space
(see also \cite{Han2018A, Han2018B}) as
\begin{equation*}
	{\bf Ric}(\nabla f, \nabla f):={\bf \Gamma_2}(f, f)-\|\nabla^2 f\|_{\rm HS}^2\mu,
\end{equation*}
where
\begin{equation*}
	{\bf \Gamma_2}(f, f):={1\over 2}{\bf \Delta} |\nabla f|^2-\langle \nabla f, \nabla \Delta f\rangle\mu
\end{equation*}
for all nice functions $f$ on RCD space $(X, d, \mu)$, where ${\bf \Delta}$ is the 
Laplacian in the distribution sense.

For nice function $f$ whose Hessian  $\nabla^2 f$ has finite Hilbert-Schmidt norm, i.e., 
$\|\nabla^2 f\|_{\rm HS}<\infty$, the trace of $\nabla^2 f$, denoted by ${\rm Tr}\nabla^2 f$ in this paper 
as in Riemannian geometry, is introduced by Han \cite{Han2018A, Han2018B} as follows: Let $e_1, \ldots, e_n$ be a basis of the $L^2$-tangent module $L^2(TX)$. Then
$${\rm Tr}\nabla^2 f=\sum\limits_{1\leq i, j\leq n}\nabla^2 f(e_i, e_j)\langle e_i, e_j\rangle.$$ 
In our notation, it reads as follows
$${\rm Tr}\nabla^2 f=\langle\nabla^2 f, {\bf g}\rangle.$$

The $N$-dimensional Bakry-Emery Ricci curvature {\bf measure} of the Laplacian 
$$\Delta={\rm Tr}\nabla^2+(\Delta-{\rm Tr}\nabla^2)$$ on 
an $n$-geometric dimensional RCD space $(X, d, \mu)$ is defined by Han \cite{ Han2018A, Han2018B} 
as follows

\begin{equation*}
	{\bf Ric_{N, n}}(\Delta)(\nabla f, \nabla f):={\bf Ric}(\nabla f, \nabla f)
	-{|{\rm Tr}\nabla^2 f-\Delta f|^2\over N-n}.
\end{equation*}

\medskip

With the help of the above definitions and notations,  we introduce the following

	\begin{definition}\label{defRCDKnN} Let $(X, d, \mu)$ be an RCD space, and $n={\rm dim} X$ be the global geometric dimension of $X$. Let $\Delta={\rm Tr}\nabla^2+(\Delta-{\rm Tr}\nabla^2)$  be the splitting  of the Laplacian on $(X, d, \mu)$. We say that  $(X, d, \mu)$ satisfies the Riemannian Bochner formula RBF$(n, N)$  if for all nice $f\in W^{1, 2}(X)\cap D(\Delta)$ and $a\in \mathbb{R}$, it holds
\begin{eqnarray}
{\bf \Gamma_2}(f, f)+2a{\rm Tr}\nabla^2 f+na^2=\left\|\nabla^2 f+a{\bf g}\right\|^2_{\rm HS}+
\operatorname{\bf Ric_{N, n}}(\Delta)(\nabla f, \nabla f)+{|{\rm Tr}\nabla^2f-\Delta f|^2\over N-n}\label{RBFna}\end{eqnarray} 
in the sense of distribution.
%
%where ${\bf g}=\langle\cdot, \cdot\rangle$ is the inner product on the $L^2$-tangent 
%module of $(X, d, \mu)$ (see \cite{Gigili2015MAMS, Han2018A, Han2018B} for its definition), ${\rm Tr}\nabla^2 f$ is the trace of $\nabla^2 f$ which is the Hessian of $f$ which we assume  to be in the trace class with finite the Hilbert-Schmidt norm, i.e., $\|\nabla^2 f\|_{\rm HS}<\infty$,	 and ${\bf Ric_{N, n}}(\Delta)$ is the $N$-dimensional  Bakry-Emery Ricci curvature {\bf measure} of the Laplacian 
%$\Delta={\rm Tr}\nabla^2+(\Delta-{\rm Tr}\nabla^2)$ on the $n$-geometric dimensional RCD space $(X, d, \mu)$. 
%	 

	 	We say that an $n$-geometric dimensional metric measure space $(X, d, \mu)$ satisfies the RCD$(K, n, N)$-condition, if it is an RCD space on which the Riemannain Bochner formula $(\ref{RBFna})$ holds in the sense of distribution and the $N$-dimensional Bakry-Emery Ricci curvature measure ${\bf Ric_{N, n}}(\Delta)$ of $\Delta$ satisfies the following inequality in the sense of distribution 
	
	\begin{equation}\label{RKNn}
	{\bf Ric_{N, n}}(\Delta)(\nabla f, \nabla f)\geq K|\nabla f |^2,\ \ \ \forall f\in W^{1, 2}(X).
\end{equation}

\end{definition}

Note that, if the Riemannian Bochner formula holds, thee following formula holds in distribution%for $a=-{{\rm Tr}\nabla^2 f\over n}$, we have
\begin{eqnarray*}\label{RBFnN}
		{\bf \Gamma_{2}}(f, f)=\left\|\nabla^2 f-{{\rm Tr}\nabla^2 f\over n}{\bf g}\right\|_{\rm HS}^2+{|{\rm Tr}\nabla^2 f|^2\over n}+
		{\bf Ric_{N, n}}(\Delta)(\nabla f, \nabla f) +{|{\rm Tr}\nabla^2f-\Delta f|^2\over N-n}.
 \end{eqnarray*}

\medskip

Now we state the $W$-entropy formula on RCD$(K, n, N)$ spaces. It extends the $W$-entropy formula of the Witten Laplacian from $n$-dimensional complete Riemannian manifolds with CD$(K, N)$-condition (\cite{Ni1, Li2012, Li2016SPA,  LiLi2015PJM, LiLi2018JFA}) to $n$-geometric dimensional RCD$(K, n, N)$ metric measure spaces. 

\begin{theorem} \label{WnNK}  
Let $(X, d, \mu)$ be  an RCD$(K, n, N)$ space, where $n\in \mathbb{N}$, $N\geq n$ and $K\in \mathbb{R}$. Let $u$ be a positive solution to the heat equation $\partial_t u=\Delta u$ satisfying the condition as required in Theorem \ref{1HH2}. Then
\begin{equation}
\begin{aligned}
\frac{d}{d t} W_{N, K}(u)
&= -2t\int_X \Big[\left\|\nabla^2 \log u+{1\over 2}\left({1\over t}-K\right){\bf g}\right\|_{\rm HS}^2+({\bf Ric_{N, n}}(\Delta)-K{\bf g})(\nabla \log u, \nabla \log u)\Big]u \mu\\		
& \hskip2cm -{2t\over N-n}\int_X \Big[({\rm Tr}\nabla^2 -\Delta )\log u-{N-n\over 2}\left({1\over t}-K\right)\Big]^2u d\mu. \label{WfnKN}
\end{aligned}
\end{equation}
Moreover, we have
\begin{equation*}
	\frac{d}{dt}W_{N,K}(u)\leq -\frac{2t}{N}\int_X \Big[\Delta \log u+{N\over 2}\left({1\over t}-K\right)\Big]^2u d\mu.
\end{equation*}
In particular,  ${d\over dt}W_{N,K}(u)\leq0$, and ${d\over dt}W_{N,K} (u)=0$ holds at some $t>0$ if and only if at this $t$,

\begin{equation*}
\nabla^2 \log u+{1\over 2}\left({1\over t}-K\right){\bf g}=0, \ \ \ 
{\bf Ric_{N, n}}(\Delta)=K{\bf g},
\end{equation*}
and 
\begin{equation*}
({\rm Tr}\nabla^2 -\Delta ) \log u-{N-n\over 2}\left({1\over t}-K\right)=0.
	\end{equation*}

\end{theorem}

When $K=0$, denoting $W_{N, 0}$ by $W_N$ for simplicity of notation, we have the following 

\begin{theorem} \label{WnNK=0}  
Let $(X, d, \mu)$ be  an RCD$(0, n, N)$ space, where $n\in \mathbb{N}$ and $N\geq n$. Let $u$ be a positive solution to the heat equation $\partial_t u=\Delta u$ satisfying the condition as required in Theorem \ref{1HH2}. Then
\begin{equation}
\begin{aligned}
\frac{d}{d t} W_{N}(u)
&= -2t\int_X \Big[\left\|\nabla^2 \log u+{{\bf g}\over 2t}\right\|_{\rm HS}^2+{\bf Ric_{N, n}}(\Delta)(\nabla \log u, \nabla \log u)\Big]u d\mu\\		
& \hskip2cm -{2t\over N-n}\int_X \Big[({\rm Tr}\nabla^2 -\Delta )\log u-{N-n\over 2t}\Big]^2u d\mu,\label{WfnKN}
\end{aligned}
\end{equation}
and
\begin{equation}\label{dW=0}
	\frac{d}{dt}W_{N}(u)\leq -\frac{2t}{N}\int_X \Big[\Delta \log u+{N\over 2t}\Big]^2u d\mu \leq 0.
\end{equation}
In particular,  $W_N(u(t))$ is decreasing along the heat equation $\partial_t u=\Delta u$ on $(0, \infty)$, and ${d\over dt}W_{N,K} (u)=0$ holds at some $t>0$ if and only if at this $t$,

\begin{equation*}
\nabla^2 \log u+{{\bf g}\over 2t}=0, \ \ \ {\bf Ric_{N, n}}(\Delta)=0,
\end{equation*}
and 
\begin{equation*}
({\rm Tr}\nabla^2 -\Delta ) \log u-{N-n\over 2t}=0.
	\end{equation*}

\medskip
Moreover, ${d\over dt}W_{N} (u(t))=0$ holds at some $t=t_*>0$ for the  fundamental solution of the heat equation $\partial_t u=\Delta u$ if and only if $(X, d, \mu)$ is one of the following rigidity models:

(i) If $N \geq 2$,  $(X, d, \mu)$ is $(0, N-1)$-cone over an $\operatorname{RCD}(N-2, N-1)$ space and $x$ is the vertex of the cone.

(ii) If $N<2$,  $(X, d, \mu)$ is isomorphic to either $\left([0, \infty), d_{\text {Eucl }}, x^{N-1} \mathrm{~d} x\right)$ or  $\left(\mathbb{R}, d_{\mathrm{Eucl}},|x|^{N-1} \mathrm{~d} x\right)$, where $d_{\text {Eucl }}$ is the canonical Euclidean distance. 

In each of the above cases, $W_N(u(t))$ is a constant on $(0, \infty)$, the Fisher information  $I(u(t))$ is given by $I(u(t))={N\over 2t}$ for all $t\in (0, \infty)$, and there exists some $x_0\in M$ such that
$$\Delta d^2(\cdot, x_0)=2N.$$ 

%
%at this $t>0$
%\begin{equation}\label{GGGG}
%\nabla^2 \log u+{I\over 2t}=0, \ \ \ {\bf Ric_{N, n}}(\Delta)=0,\ \ \ 
%({\rm Tr}\nabla^2 -\Delta ) \log u-{N-n\over 2t}=0.
%	\end{equation}
	
\end{theorem}

 Theorem \ref{WnNK=0}  is a natural extension of the previous results due to Ni \cite{Ni1} for the heat equation of the Laplace-Beltrami operator on compact Riemannian manifolds, and due to the first named author \cite{Li2012} for the heat equation of the Witten Laplacian on $n$-dimensional weighted complete Riemannian manifolds with bounded geometry condition and CD$(0, N)$-condition for $N\geq n$. See also Li \cite{Li2016SPA} and 
 S. Li-Li \cite{LiLi2015PJM}. More precisely, we have the following 

\begin{theorem}\label{Lirigidity}
 (See Theorem 2.3 in \cite{Li2012}) Let $(M, g)$ be a complete Riemannian manifold with bounded geometry condition, $d\mu=e^{-V}dv$, where 
$dv$ is the standard Riemannian volume measure on $(M, g)$, and
$V\in C^4(M)$ with $\nabla V\in C_b^3(M)$. Let $L=\Delta-\nabla V\cdot \nabla$, where $\Delta={\rm Tr}\nabla^2$ is the Laplace-Beltrami operator on $(M, g)$. Let  $u$ be the fundamental solution to the heat equation $\partial_t u=Lu$. Then the following $W$-entropy formula  holds

\begin{eqnarray*}
	\frac{d}{dt}W_N(u)&=&-2t\int_M \left(\Big\|\nabla^2 \log u+\frac{g}{2t} \Big\|_{\operatorname{HS}}^2
	+\Ric_{N, n}(L)(\nabla \log,\nabla \log u)\right)u d\mu \nonumber\\
	& &\hskip1cm -{2t\over N-n}\int_M\Big|\nabla V\cdot \nabla \log u-{N-n\over 2t}\Big|^2 u d\mu, \label{WLi2}
\end{eqnarray*}
In particular, under the  CD$(0, N)$-condition, i.e., $Ric_{N, n}(L)\geq 0$, we have ${d\over dt}W_N(u(t))\leq 0$ for all $t>0$. Moreover, ${d\over dt}W_N(u(t))=0$  holds at some $t=t_*>0$ if and only if $(M, g)$ is isometric to $(\mathbb{R}^n, g_{\rm Euclid})$, $N=n$, $V$ is a constant and $u(x, t)={e^{-{\|x\|^2\over 4t}}\over (4\pi t)^{n\over 2}}$ is the Gaussian heat kernel on $\mathbb{R}^n$ for all $x\in M=\mathbb{R}^n$ and $t>0$. 
  \end{theorem}

Theorem \ref{WnNK=0} can also regarded as an improvement of the monotonicity and rigidity theorem of the $W$-entropy on RCD$(0, N)$ spaces originally proved  by  Kuwada-Li \cite{KL}. The 
new part of Theorem  \ref{WnNK=0} is the $W$-entropy formula \eqref{WfnKN} on RCD$(0, n, N)$ spaces, which implies the monotonicity of $W_N(u(t))$ proved in  Kuwada-Li \cite{KL} by a different way. The rigidity part of  Theorem \ref{WnNK=0} is indeed as the same as in  Kuwada-Li \cite{KL}. More precisely, we have

\begin{theorem} (See Theorem 4.1 in \cite{KL})\label{KLRigidity} Let $(X, d, \mu)$ be an RCD$(0, N)$ space. Then $W_N(u(t))$ is decreasing along the heat equation $\partial_t u=\Delta u$. If  
	 the right upper derivative of $W\left(P_t \mu\right)$ is 0 
	at $t=$ $t_* \in(0, \infty)$, that is,
$$
\varlimsup_{t \downarrow t_*} \frac{W\left(P_t \mu\right)-W\left(P_{t_*} \mu\right)}{t-t_*}=0,
$$
for some $\mu \in \mathcal{P}_2(X)$ and $t_* \in(0, \infty)$, where $P_t=e^{t\Delta}$ is the 
heat semigroup of $\Delta$, then $(X, d, \mu)$ must be the RCD space  as stated in Theorem \ref{WnNK=0}. Moreover,  we have
	
(i) $I(P_t\mu)=\frac{N}{2t}$ holds  for any $t \in\left(0, t_*\right]$.

(ii) For $\mu$-a.e. $x_0 \in X$,
\begin{equation}\label{Fishers}
	I\left(P_t \delta_{x_0}\right)=\frac{N}{2 t},
\end{equation}
holds for any $t \in\left(0, t_*\right]$. In particular, there exists $x_0 \in X$ satisfying \eqref{Fishers}.
\end{theorem}

We now turn to state the second part of results on the concavity of the Shannon entropy power on RCD spaces. 
For the relationship between the $W$-entropy, the Fisher information $I$ and the Shannon entropy power $N$, see Theorem \ref{thmNIW} in Section 6 below. 

\begin{theorem}
Let $(X, d, \mu)$ be a metric measure space satisfying the Riemannian curvature dimension condition
 RCD$(K, N)$. Let $\mathcal{N}(u(t))=e^{2H(u(t))\over N}$ be the Shannon entropy power for the heat equation $\partial_t u=\Delta u$. Then
	\begin{equation*}
		\begin{split}
			\frac{N}{2\mathcal{N}}\Big[\frac{d^2\mathcal{N}}{dt^2}+2K\frac{d\mathcal{N}}{dt}\Big]&\leq-\frac{2}{N}\int_X \Big[\Delta \log u-\int_X u\Delta \log ud\mu \Big]^2 ud\mu.
		\end{split}
	\end{equation*}
In particular, 
\begin{equation*}
	\frac{N}{2\mathcal{N}}\Big[\frac{d^2\mathcal{N}}{dt^2}+2K\frac{d\mathcal{N}}{dt}\Big]\leq 0.
\end{equation*}

\end{theorem}

As a corollary,  we have the following result due to Ebar, Kuwada and Sturm \cite{EKS2015}. In the setting of complete Riemannian manifolds, see S. Li and Li \cite{LiLi2024TMJ}. 

\begin{corollary}\label{EKSN} Let  $\mathcal{N}(u)=e^{2H(u)\over N}$ be the Shannon entropy power associated with the heat equation $\partial_t u=\Delta u$
	 on an RCD$(0, N)$ space. Then 
	 \begin{equation*}
	\frac{d^2\mathcal{N}(u)}{dt^2} \leq 0.
\end{equation*}
That is to say,  the Shannon entropy power is concave on RCD$(0,N)$ spaces.
\end{corollary}

 As an application of the concavity of Shannon entropy power on RCD$(0, N)$ spaces, 
 we have the following  entropy isoperimetric inequality and the Stam type Log-Sobolev inequality on RCD$(0, N)$ spaces with maximal volume growth condition.
\begin{theorem}\label{StamLSIA}
	Let $(X, d, \mu)$ be an 
 RCD$(0, N)$ space satisfying the maximal volume growth condition: for all $ x \in X$ and $r>0$, there exists a constant $ C_n > 0$ such that
		\begin{equation*}\label{deng1}
		\mu(B(x,r))\geq C_n r^N,
	\end{equation*} where $B(x, r)=\{y\in X: d(x, y)\leq r\}$ for any $x\in X$ and $r>0$.
	Then
	\begin{equation}\label{deng4}
		I(f)\mathcal{N}(f)\geq\gamma_N:=2\pi Ne\kappa^{\frac{2}{N}},
	\end{equation}
	where
	\begin{equation*}
		\kappa:=\lim_{r\rightarrow\infty}\frac{\mu(B(x,r))}{\omega_N r^N},
	\end{equation*}
	and $\omega_N$ denotes the volume of the unit ball in $\mathbb{R}^N$.
	Moreover, inequality \eqref{deng4} is equivalent to the Stam type  Log-Sobolev inequality: for all nice $f
	>0$ with $\int_X {|\nabla f |^2_w \over f}d\mu<+\infty$, it holds
\begin{equation}
	\int_X f\log f d\mu \leq \frac{N}{2}\log\Big(\frac{1}{\gamma_N}\int_X {|\nabla f |^2_w \over f}d\mu\Big).\label{StamLSIA}
\end{equation}
\end{theorem}

The following result gives the rigidity theorem of the Stam logarithmic Sobolev inequality with the sharp constant 
on non-collapsing RCD$(0, N)$ space. 
 Recall that an RCD$(0, N)$ space $(X, d, \mu)$ is called non-collapsing if for any $x\in X$, it holds (see \cite{BGZ, Gigli2016, Honda20})
\begin{equation}\label{Noncollapsing}
	\lim\limits_{r\rightarrow 0} \frac{\mu(B(x,r))}{\omega_N r^N}=1.
\end{equation}

\begin{theorem}\label{Rigidity}
	Let $(X, d, \mu)$ be a non-collapsing  RCD$(0,N)$ space and the Stam logarithmic Sobolev inequality \eqref{StamLSIA}
holds with the sharp constant $\gamma_N=2\pi e N$ (i.e., $\kappa=1$ in Theorem \ref{StamLSIA})  as on the $N$-dimensional Euclidean space. Then $(X, d, \mu)$ must be one of the rigidity models in Theorem \ref{WnNK=0}. 
%
%(i) If $N \geq 2$,  $(X, d, \mu)$ is $(0, N-1)$-cone over an $\operatorname{RCD}(N-2, N-1)$ space and $x$ is the vertex of the cone.
%
%(ii) If $N<2$,  $(X, d, \mu)$ is isomorphic to either $\left([0, \infty), d_{\text {Eucl }}, x^{N-1} \mathrm{~d} x\right)$ or  $\left(\mathbb{R}, d_{\mathrm{Eucl}},|x|^{N-1} \mathrm{~d} x\right)$, where $d_{\text {Eucl }}$ is the canonical Euclidean distance. 
\end{theorem}

\begin{remark}
	 We would like to point out that, after we proved the Shannon entropic isometry inequality \eqref{stamLogS} which is  equivalent to the Stam type Log-Sobolev inequality \eqref{NXY5}. We find the recent published paper by Balogh et al \cite{BK, BKT}, in which the authors proved very similar results  as the Stam type Log-Sobolev inequality. Our work is independent of \cite{BK, BKT}.  Our proof uses the argument of Li-Li  \cite{LiLi2024TMJ}  and is different  from \cite{BK, BKT}. 

\end{remark}

\section{The entropy dissipation formulas on RCD$(K,N)$ spaces}
In this section, we prove the entropy dissipation equalities and  inequality on 
 RCD$(K, N)$ spaces. In the case $K=0$,  the first order entropy dissipation equality and the second order entropy dissipation inequality  have been  essentially proved in \cite{EKS2015} even though the authors did not precisely formulated the integrability conditions for the validity of these results. We will modify the proof  in \cite{EKS2015} and prove the entropy dissipation equalities and  inequality  under the integrability conditions  precisely formulated in Theorem \ref{1HH2}. 
 
 To prove our results, we need following Li-Yau Harnack inequality on RCD spaces. 
\begin{theorem}\cite{zhuxiping}\label{zhulemma}
	 Let $K \geq 0$ and $N \in [1,\infty )$, and let $(X,d,\mu)$ be a metric measure space
satisfying RCD$^\ast(-K, N)$. Let $T_\ast \in (0, \infty ] $ and $u(x, t)$ is a positive solution of the heat equation on $X \times\left(0, T_*\right)$. Then, for almost all $T \in\left(0, T_*\right)$, the following gradient estimate holds
$$
\sup _{x \in X}\left(|\nabla f|^2-\alpha \cdot \frac{\partial}{\partial t} f\right)(x, T) \leq\left(1+\frac{K T}{2(\alpha-1)}\right) \cdot \frac{N \alpha^2}{2 T}
$$
for any $\alpha>1$, where $f=\log u$.
\end{theorem}

In order to apply the weak Bochner inequality,  we regularize $\mu$ as described in \cite{EKS2015} through a series of steps in several steps.
		First,  put $P_\delta u=\tilde{u_\delta}$, let
		 $$u_\delta=c_\delta h^\delta(\tilde{u_\delta}\wedge \delta^{-1}),$$ where 
		 $c_\delta$ is a normalizing constant and 
		 $$h^\varepsilon u=\int^\infty_0\frac{1}{\varepsilon}\eta\left(\frac{t}{\varepsilon}\right)P_t u dt$$ with a non-negative 
		 kernel $\eta \in C_c^\infty (0,\infty)$ satisfying $\int^\infty_0 \eta(t)dt=1$, 
		 then put $P_tu_\delta=u_t^\delta $,  we also regularize entropy functional as follow. 
		 
		 Define $e_\varepsilon :[0,\infty)\rightarrow \mathbb{R}$ with $e'_\varepsilon(r)=\log(\varepsilon+r)+1$ and
		  $e_\varepsilon(0)=0$. Set $$H_\varepsilon(u)=-\int e_\varepsilon(u)d\mu.$$ 
		  Note that $u_t^\delta \in L^\infty(X,\mu) \bigcap D(\Delta)$ and
		   $\Delta u_t^\delta \in L^\infty(X,\mu) \bigcap D(\Delta)$. Let
		    $$g_t^{\varepsilon,\delta}=e'_\varepsilon(u_t^\delta)-1-\log(\varepsilon),$$ 
			we obtain from chain rule that $g_t^{\varepsilon,\delta} ,|\nabla g_t^{\varepsilon,\delta}|,\Delta g_t^{\varepsilon,\delta}\in L^\infty(X,\mu) \bigcap D(\Delta)$.
			
			Note that
			
			\begin{equation*}
\nabla g_t^{\varepsilon,\delta} =e''_\varepsilon (u_t^\delta)\nabla u_t^\delta ={\nabla u_t^\delta\over u_t^\delta+\varepsilon},
\end{equation*}
		    and 
\begin{eqnarray*}
\Delta g_t^{\varepsilon,\delta} =e''_\varepsilon (u_t^\delta) \Delta u_t^\delta+e'''_\varepsilon (u_t^\delta)|\nabla u_t^\delta|_w^2={\Delta u_t^\delta \over u_t^\delta+\varepsilon}-{|\nabla u_t^\delta|_w^2\over (u_t^\delta+\varepsilon)^2}.
\end{eqnarray*}
		    
Now we are ready to prove Theorem \ref{1HH2}. We will divide our proof into three steps.
		    
%%%%%%%%%%% 1  %%%%%%%%%%%%%%%

%%%%%%%%%%% 1  %%%%%%%%%%%%%%%

%%%%%%%%%%% 2  %%%%%%%%%%%%%%%
%\begin{theorem}\label{HHH12}
%	Let $(X, d, \mu)$ be an 
% RCD$^*(K, N)$ space and $u$ be a solution to the heat equation $\partial_t u_t=\Delta u_t$. Let $H(u_t)=-\int_X u_t \log u_t d\mu$ be the Shannon entropy. 
%Suppose that there exists a $\delta_0>0$ such that
%\begin{equation}\label{Condition2}
%	\begin{aligned}
%		\int_X \sup_{\delta\in (0, \delta_0)}\Big[\frac{|\nabla\Delta u_t^\delta |^2}{u_t^\delta}+\frac{|\nabla u_t^\delta |^2}{u_t^\delta} +\frac{|\Delta u_t^\delta |^2}{u_t^\delta}+\frac{|\nabla |\nabla u_t^\delta |^2||\nabla u_t^\delta | }{(u_t^\delta)^2}+ \Big|\frac{\Delta|\nabla u_t^\delta |^2 }{u_t^\delta}\Big|\Big]d\mu<\infty .
%	\end{aligned}
%\end{equation}
%
%Then the following entropy dissipation formulas hold
%	\begin{equation}\label{H1}
%	\frac{d}{dt}H(u_t)=-\int_X \log u_t
%\Delta u_t d\mu=-\int_X \Delta \log u_t u_td\mu,
%\end{equation}
%and
%\begin{equation}
%	\frac{d^2}{dt^2}H(u_t)=-2\int_X \Gamma_2(\nabla \log u_t, \nabla\log u_t)u_td\mu. \label{H12}
%\end{equation}
%where 
%\begin{equation*}
%	\Gamma_2(\nabla \log u, \nabla \log u):={1\over 2}\Delta \left\|\nabla\log u\right\|^2-\langle \nabla\log u, \Delta \nabla \log u\rangle.
%\end{equation*}
%Furthermore, under the RCD$^*(K, N)$-condition, we have
%		\begin{eqnarray}
%		\frac{d^2}{dt^2}H(u_t)
%		\leq -\frac{2}{N} \int_X u_t\Big(\Delta \log u_t\Big)^2d\mu-2K\int_X u_t|\nabla \log u_t |_w^2 d\mu.\label{H2}
%	\end{eqnarray}
%
%\end{theorem}

\medskip

\noindent{\bf Step 1: The continuity of entropy in $\varepsilon$ and $\delta$}. 
\medskip
This follows the integrability condition \eqref{C0} and the Lebesgue dominated convergence theorem. 

\medskip

\noindent{\bf Step 2: The first order entropy dissipation formula}. 
By \cite{EKS2015}, it holds
		\begin{equation}\label{d1t1H}
		\begin{aligned}
			\frac{d}{dt}H_\varepsilon(u_t^\delta)&=-\int_X e'_\varepsilon(u_t^\delta)\partial_t u_t^\delta d\mu=-\int_X e'_\varepsilon(u_t^\delta)\Delta u_t^\delta d\mu\\
			&=-\int_X \Delta e'_\varepsilon(u_t^\delta) u_t^\delta d\mu=\int_X \langle \nabla e'_{\varepsilon}(u_t^\delta), \nabla u_t^\delta \rangle d\mu\\
			 &=\int_X \langle\nabla g_t^{\varepsilon,\delta}, \nabla u_t^\delta \rangle d\mu=  \int_X \frac{|\nabla u_t^\delta |^2}{u_t^\delta+\varepsilon} d\mu.
		\end{aligned}
	\end{equation}
	Note that
%	\begin{equation*}
%		\lim_{\delta\rightarrow 0}\lim_{\varepsilon\rightarrow 0}\Delta g_t^{\varepsilon,\delta} u_t^\delta =\Delta \log u_t u_t,
%	\end{equation*}
%	and
	\begin{equation*}
	\frac{|\nabla u_t^\delta |^2}{u_t^\delta+\varepsilon}\leq \sup_{\delta>0} \frac{|\nabla u_t^\delta |^2}{u_t^\delta}.
\end{equation*}
By the Lebesgue Dominated Convergence Theorem, if \eqref{C1} holds,  we have
\begin{equation*}
\begin{aligned}
	\lim_{\delta\rightarrow 0}\lim_{\varepsilon\rightarrow 0}\int_X \frac{|\nabla u_t^\delta |^2}{u_t^\delta+\varepsilon} d\mu
	&=\int_X \lim_{\delta\rightarrow 0}\lim_{\varepsilon\rightarrow 0}\frac{|\nabla u_t^\delta |^2}{u_t^\delta+\varepsilon} d\mu=\int_X \frac{|\nabla u_t |^2}{u_t} d\mu\\
	&=\int_X \langle \nabla \log u_t,  \nabla u_t\rangle d\mu=-\int_X \log u_t
\Delta u_t d\mu\\
&=-\int_X \Delta \log u_t u_td\mu.
\end{aligned}
\end{equation*}
Therefore 
\begin{equation*}
	\frac{d}{dt}H(u_t)=-\int_X \log u_t
\Delta u_t d\mu=-\int_X \Delta \log u_t u_td\mu.
\end{equation*}

\medskip

\noindent{\bf Step 3: The second order entropy dissipation formula and inequality}. We have
 	\begin{equation}\label{d2t2H1}
		\begin{split}
			\frac{d^2}{dt^2}H_\varepsilon(u_t^\delta)
		&=-{d\over dt}\int_X e'_\varepsilon(u_t^\delta)\Delta u_t^\delta d\mu\\
			&=-\int_X \Big[e''_\varepsilon(u_t^\delta)(\Delta u_t^\delta)^2+e'_\varepsilon(u_t^\delta)\Delta^2u_t^\delta \Big]d\mu\\
			&=-\int_X e''_\varepsilon(u_t^\delta)(\Delta u_t^\delta)^2d\mu -\int_X\Delta e'_\varepsilon(u_t^\delta)\Delta u_t^\delta d\mu\\
			& \ \  ({\rm using\ the\ fact}\ \Delta e'_\varepsilon(u_t^\delta)=e''_\varepsilon(u_t^\delta)\Delta u_t^\delta+e'''_\varepsilon(u_t^\delta)|\nabla u_t^\delta |_w^2\\
			&=-2\int_X e''_\varepsilon(u_t^\delta)(\Delta u_t^\delta)^2d\mu-\int_X e'''_\varepsilon(u_t^\delta)|\nabla u_t^\delta |_w^2\Delta u_t^\delta d\mu\\
			&=-2\int_X e''_\varepsilon(u_t^\delta)(\Delta u_t^\delta)^2d\mu-2\int_X e'''_\varepsilon(u_t^\delta)|\nabla u_t^\delta |_w^2\Delta u_t^\delta d\mu+\int_X e'''_\varepsilon(u_t^\delta)|\nabla u_t^\delta |_w^2\Delta u_t^\delta d\mu \\
			& \ \  ({\rm using\ again\ the\ fact}\ \Delta e'_\varepsilon(u_t^\delta )=e''_\varepsilon(u_t^\delta)\Delta u_t^\delta+e'''_\varepsilon(u_t^\delta)|\nabla u_t^\delta |_w^2\\
			&=-2\int_X \Delta e'_\varepsilon(u_t^\delta)\Delta u_t^\delta d\mu+ \int_X e'''_\varepsilon(u_t^\delta)|\nabla u_t^\delta |_w^2\Delta u_t^\delta d\mu\\
			&=\int_X \Big[2 \<\nabla u_t^\delta,\nabla \Delta e'_\varepsilon(u_t^\delta) \>+e'''_\varepsilon(u_t^\delta)|\nabla u_t^\delta |_w^2\Delta u_t^\delta\Big] d\mu\\
			& \ \  (\ \ {\rm using\ the\ fact}\ e'''_{\varepsilon}(r)=-(e''_{\varepsilon}(r)^2)\ \ \ )\\
			&=\int_X \Big[2 (u_t^\delta+\varepsilon)\<\nabla \log(u_t^\delta+\varepsilon),\nabla \Delta e'_\varepsilon(u_t^\delta)\>d\mu-(e''_{\varepsilon}(u_t^\delta))^2 |\nabla u_t^\delta |_w^2\Delta u_t^\delta\Big] d\mu \\
			&=\int_X \Big[2 (u_t^\delta+\varepsilon)\<\nabla e'_\varepsilon(u_t^\delta),\nabla \Delta e'_\varepsilon(u_t^\delta)\>d\mu- |\nabla e'_\varepsilon(u_t^\delta)|_w^2\Delta u_t^\delta\Big] d\mu\\
			&= \int_X\Big[2 (u_t^\delta+\varepsilon)\<\nabla g_t^{\varepsilon,\delta},\nabla \Delta g_t^{\varepsilon,\delta} \>-|\nabla g_t^{\varepsilon,\delta} |^2_w\Delta (u_t^\delta+\varepsilon )\Big] d\mu\\
			&=\int \left[ 2\<\nabla (u_t^\delta+\varepsilon), \nabla \Delta g_t^{\varepsilon,\delta} \>-|\nabla g_t^{\varepsilon,\delta} |^2_w \Delta(u_t^\delta+\varepsilon )\right]d\mu\\
			&=-\int \left[ 2\Delta(u_t^\delta+\varepsilon) \Delta g_t^{\varepsilon,\delta}+|\nabla g_t^{\varepsilon,\delta} |^2_w \Delta(u_t^\delta+\varepsilon )\right]d\mu.
			%&= \int_X\Big[2 (u_t^\delta+\varepsilon)\<\nabla g_t^{\varepsilon,\delta},\nabla \Delta g_t^{\varepsilon,\delta} \>-{\bf \Delta}|\nabla g_t^{\varepsilon,\delta} |^2_w (u_t^\delta+\varepsilon )\Big] d\mu.
%			&=-2\int_X {\bf\Gamma}_2(g_t^{\varepsilon,\delta}, g_t^{\varepsilon,\delta})
%	(u_t^\delta+\varepsilon) d\mu.
		\end{split}		
	\end{equation}
	Note that
	\begin{equation*}
	\begin{aligned}
	\Big|\Delta (u_t^\delta +\varepsilon ) \Delta g_t^{\varepsilon,\delta}\Big|
	=|\Delta u_t^\delta|\left| \frac{\Delta u_t^\delta}{u_t^\delta} -\frac{|\nabla u_t^\delta|^2}{ (u_t^\delta)^2}\right|\leq \frac{(\Delta u_t^\delta )^2}{u_t^\delta}+\frac{|\nabla u_t^\delta |^2|\Delta u_t^\delta |}{(u_t^\delta )^2}.
	\end{aligned}
	\end{equation*}
By the Li-Yau Harnack inequality in Theorem \ref{zhulemma}, we have
\begin{equation}\label{DDDD}
\begin{aligned}
\frac{|\nabla u_t^\delta |^2 |\Delta u_t^\delta |}{(u_t^\delta)^2}&\leq \left[ \alpha  \frac{\Delta u_t^\delta}{u_t^\delta}+C_{K,N,T,\alpha}\right] |\Delta u_t^\delta | \\
%& \leq \alpha \frac{|\Delta u_t^\delta|^2}{u_t^\delta}+C_{K,N,T,\alpha} \frac{|\Delta u_t^\delta | }{\sqrt{u_t^\delta}}\sqrt{u_t^\delta}\\
&\leq \alpha \frac{|\Delta u_t^\delta|^2}{u_t^\delta}+C_{K,N,T,\alpha} \Big(\frac{|\Delta u_t^\delta|^2}{u_t^\delta}+u_t^\delta \Big)\\
&=(\alpha+C_{K,N,T,\alpha})\frac{|\Delta u_t^\delta|^2}{u_t^\delta}+C_{K,N,T,\alpha}u_t^\delta.
\end{aligned}
\end{equation}

On the other hand
\begin{equation}\label{EEEE}
	\begin{aligned}
	|\Delta u_t^\delta ||\nabla g_t^{\varepsilon,\delta}|^2
	=|\Delta u_t^\delta| \frac{|\nabla u_t^\delta|^2}{ (u_t^\delta+\varepsilon)^2}\leq\frac{|\nabla u_t^\delta|^2 |\Delta u_t^\delta| }{ (u_t^\delta)^2}.	\end{aligned}
	\end{equation}
	By \eqref{DDDD} and \eqref{EEEE}, under the integrability condition \eqref{C4}, 
we have 	
	\begin{equation}
	\int_X \sup\limits_{\varepsilon, \delta}[|\Delta (u_t^\delta+\varepsilon ) \Delta g_t^{\varepsilon,\delta}|+|\nabla g_t^{\varepsilon,\delta} |^2_w| |\Delta(u_t^\delta+\varepsilon )|]d\mu<+\infty.
	\end{equation}
%	That is to say
%	\begin{equation*}
%	\begin{aligned}
%	\sup\limits_{\varepsilon, \delta}|C_{1, \varepsilon,\delta} |
%	& \leq \sup\limits_{\varepsilon, \delta}\left[\frac{(\Delta u_t^\delta )^2}{u_t^\delta}+\frac{|\nabla u_t^\delta |^2|\Delta u_t^\delta |}{(u_t^\delta )^2}\right]
%	\end{aligned}
%	\end{equation*}
%	Hence, under the integrability condition \eqref{C4}, 
% the integrant in  the last term of \eqref{d2t2H1} is uniformly integrable with respect to $\varepsilon$ and $\delta$.
 By the Lebesgue dominated convergence theorem, we  have
	\begin{equation*}\label{HHHHH1}
		\begin{split}
			\lim_{\delta\rightarrow 0}\lim_{\varepsilon\rightarrow 0}\frac{d^2}{dt^2}H_\varepsilon(u_t^\delta)
						&=-\lim_{\delta\rightarrow 0}\lim_{\varepsilon\rightarrow 0}\int_X\left[ 2\Delta(u_t^\delta+\varepsilon) \Delta g_t^{\varepsilon,\delta}+|\nabla g_t^{\varepsilon,\delta} |^2_w \Delta(u_t^\delta+\varepsilon )\right]d\mu\\
						&=-\int_X \lim_{\delta\rightarrow 0}\lim_{\varepsilon\rightarrow 0}\left[ 2\Delta(u_t^\delta+\varepsilon) \Delta g_t^{\varepsilon,\delta}+|\nabla g_t^{\varepsilon,\delta} |^2_w \Delta(u_t^\delta+\varepsilon )\right]d\mu\\
&=-\int_X \left[ 2\Delta u_t \Delta \log u_t+|\nabla \log u_t |^2_w \Delta u_t\right]d\mu\\
&=-\int_X \left[ 2\Delta \log u_t+|\nabla \log u_t |^2_w \right]\Delta u_t d\mu\\
&(\ {\rm by\ the\ integration\ by\ parts\ formula\ and\ the\ duality\ of}\ {\bf \Delta})\\
&= \int_X \left[2\<\nabla \Delta \log u_t, \nabla \log u_t\>-{\bf \Delta} |\nabla \log u_t |^2_w \right]u_td\mu\\
&=-\int_X {\bf \Gamma_2}(\nabla \log u_t, \nabla \log u_t)u_td\mu.					
								\end{split}		
	\end{equation*}
By the continuity of $H_{\varepsilon, \delta}(u_t^\delta)$ and ${d\over dt}H_{\varepsilon, \delta}(u_t^\delta)$ in $\varepsilon$ and $\delta$, we have
\begin{equation*}
	\frac{d^2}{dt^2}H(u_t)=-2\int_X {\bf \Gamma_2}(\nabla \log u_t, \nabla\log u_t)u_td\mu.
	\end{equation*}
% Finally,  by the RCD$^*(K, N)$-condition, we have
%			\begin{equation}\label{d2t2H}
%		\begin{split}
%			\frac{d^2}{dt^2}H(u_t)\leq -\frac{2}{N} \int_X u_t(\Delta \log u_t)^2d\mu-2K\int_X u_t|\nabla \log u_t |_w^2 d\mu.\\
%		\end{split}		
%	\end{equation}
	 \hfill $\square$

%%%%-------------------------------

The following result gives two equivalent representations of the Fisher information. 

\begin{lemma}\label{Fisher}
	If $u$ satisfies the condition $(\ref{C1})$ in Theorem \ref{1HH2}, we have
	 \begin{equation}
	 	\int_X u|\nabla \log u |_w^2 d\mu=-\int_X u\Delta\log u  d\mu.
	 \end{equation}
	 \begin{proof}
	 	In order to apply the chain rule on RCD-spaces, we regularize entropy functional as before.
		 Define $u^\delta$ as before and $e_\varepsilon :[0,\infty)\rightarrow \mathbb{R}$ with $e'_\varepsilon(r)=\log(\varepsilon+r)+1$, $e_\varepsilon(0)=0$. By the chain rule, we have
	 	\begin{equation*}
	 		\Delta e'_\varepsilon(u^\delta )=e''_\varepsilon(u^\delta)\Delta u^\delta+e'''_\varepsilon(u^\delta)|\nabla u^\delta |_w^2.
	 	\end{equation*}
	 	 Note that $e''_\varepsilon(r)=(r+\varepsilon)^{-1}$ and $e_\varepsilon'''(r) =-(e''_\varepsilon(r))^2$. Then
	 	\begin{equation*}
	 		\begin{split}
	 			\int_X \Delta e'_\varepsilon(u^\delta)u^\delta d\mu &=\int_X e''_\varepsilon(u^\delta)\Delta u^\delta ud\mu+\int_X e'''_\varepsilon(u^\delta)|\nabla u^\delta |_w^2ud\mu\\
	 			&=\int_X \frac{u^\delta}{u^\delta+\varepsilon}\Delta u^\delta d\mu -\int_X \frac{1}{(u^\delta+\varepsilon)^2}|\nabla u^\delta |_w^2u^\delta d\mu\\
	 			&=\int_X \frac{u^\delta}{u^\delta+\varepsilon}\Delta u^\delta d\mu -\int_X u^\delta|\nabla e'_\epsilon(u^\delta) |_w^2 d\mu .
	 		\end{split}
	 	\end{equation*}
		 Taking  $\varepsilon\rightarrow 0$ and then $\delta\rightarrow 0$, we complete the proof.
	 \end{proof}
\end{lemma}

%---------------------------------------------------------------

We now use the first order entropy dissipation equality and the second order entropy dissipation inequality in Theorem \ref{1HH2} and Lemma \ref{Fisher} to prove the  Ricatti type entropy differential inequality (EDI) for the Shannon entropy of the heat semigroup on RCD$(K,N)$ spaces. It will play a key role in the proofs of the main results of this paper. In the setting of complete Riemannian manifolds with CD$(K, N)$-condition, it is due to Li-Li \cite{LiLi2015PJM,LiLi2018SCM}.

\begin{theorem}\label{EDIKM}
	Let $(X, d, \mu)$ be an 
 RCD$(K, N)$ space and $u$ be a solution to the heat equation $\partial_t u=\Delta u$.  Let $H(u)=-\int_X u \log u d\mu$ be the Shannon entropy.  
	 Then  the Entropy Differential Inequality (EDI$(K,N)$) holds 
	\begin{equation}\label{EDI}
		\frac{d^2}{dt^2}H(u)+\frac{2}{N}\Big(\frac{d}{dt}H(u)\Big)^2+2K\frac{d}{dt}H(u)\leq 0,
	\end{equation}
%	Equivalently, 
%the Shannon entropy power $\mathcal{N}(u)=e^{2H(u)\over N}$ is $(-2K)$-concave on RCD$^\ast(K,N)$ spaces, i.e., 
%\begin{equation*}
%{d^2\mathcal{N} \over dt^2} \leq -2K\mathcal{N}'.		
%	\end{equation*}
\end{theorem}
\begin{proof}
Substituting \eqref{H1} and \eqref{H2} into \eqref{EDI}, we have
	\begin{equation*}
		\begin{split}
			&\frac{d^2}{dt^2}H(u)+\frac{2}{N}\Big(\frac{d}{dt}H(u)\Big)^2+2K\frac{d}{dt}H(u)\\
			\leq & \frac{2}{N}\Big(\int_X\Delta \log u u d\mu\Big)^2-\frac{2}{N} \int_X u\Big(\Delta \log u\Big)^2 d\mu\\
			\leq & 0,
		\end{split}
	\end{equation*}
	where the last inequality holds by application of the Cauchy-Schwarz inequality.

\end{proof}
\begin{remark} By S. Li-Li \cite{LiLi2024TMJ}, when $(X, d, \mu)=(M, g, \mu)$ is an $n$-dimensional complete Riemannian manifold with CD$(K, N)$-condition and bounded geometry condition, we have
\begin{align*}
	-\frac{1}{2}H''%&=\int_X \Gamma_2(\nabla \log u,\nabla \log u )ud\mu\\
	&=\int_X\Big[\|\nabla^2 \log u \|^2_{\operatorname{HS}}+\Ric(L)(\nabla \log u,\nabla \log u)\Big] ud\mu\\
	&=\int_X\Big[\frac{(L\log u)^2}{N}+\Ric_{N, n}(L)(\nabla \log u,\nabla \log u)+\left\|\nabla^2 \log u-\frac{\Delta\log u}{n}\right\|^2_{\operatorname{HS}}\Big]ud\mu\\
	& \hskip2cm +{N-n\over Nn}\int_M \Big[L\log u+{N\over N-n}\nabla V\cdot\nabla \log u\Big]^2ud\mu\\
	&=\frac{1}{N}\Big(\int_X|\nabla \log u |^2 ud\mu\Big)^2 +\frac{1}{N} \int_X \Big[L \log u-\int_Xu  L\log u d\mu \Big]^2ud\mu\\
	&\hskip1cm +\int_X\Big[\Ric_{N, n}(L)(\nabla \log u,\nabla \log u)+\left\|\nabla^2 \log u-\frac{\Delta\log u}{n}g\right\|^2_{\operatorname{HS}}\Big]ud\mu\\
	& \hskip2cm +{N-n\over Nn}\int_M \Big[L\log u+{N\over N-n}\nabla V\cdot\nabla \log u\Big]^2ud\mu.
\end{align*}
This yields
\begin{align*}
	\frac{1}{2}H''+\frac{{H'}^2}{N} &=
-\frac{1}{N} \int_X \Big[L \log u-\int_Xu  L\log u d\mu \Big]^2ud\mu-\int_X
\Ric_{N, n}(L) (\nabla \log u, \nabla \log u) ud\mu\\
&\hskip1cm -\int_M \left\|\nabla^2 \log u-\frac{\Delta\log u} {n}g\right\|^2_{\rm HS}ud\mu-{N-n\over Nn}
\int_M \Big[L\log u+{N\over N-n}\nabla V\cdot\nabla \log u\Big]^2ud\mu.
\end{align*}
Under the CD$(K, N)$-condition, we have
\begin{align*}
	\frac{1}{2}H''+\frac{{H'}^2}{N} &\leq -K\int_X|\nabla \log u|^2 ud\mu -\int_X \Big[\left\|\nabla^2 \log u-\frac{\Delta\log u}{n}\right\|^2_{\operatorname{HS}}\Big]ud\mu\\
	& \ \ \ -\frac{1}{N} \int_X\Big[L \log u-\int_X L\log u ud\mu \Big]^2ud\mu.
\end{align*}
Thus, on Riemannian manifold with CD$(K, N)$-condition and the bounded geometry condition, the Ricatti EDI reads 
\begin{align*}
	\frac{1}{2}H''+\frac{{H'}^2}{N} +KH'&\leq  -\int_X \Big[\left\|\nabla^2 \log u-\frac{\Delta\log u}{n}\right\|^2_{\operatorname{HS}}\Big]ud\mu-\frac{1}{N} \int_X\Big[L \log u-\int_X L\log u ud\mu \Big]^2ud\mu.
\end{align*}
\quad\quad

In general case of non-smooth  RCD metric measure space, we have
\begin{eqnarray*}
	\frac{1}{2}H''+ \frac{H'^2}{N}&=&-\int_X {\bf\Gamma}_2(\nabla \log u, \nabla \log u)ud\mu+{1\over N}\left(\int_X \Delta \log u ud\mu\right)^2\\
	&=&-\int_X {\bf\Gamma}_2(\nabla \log u, \nabla \log u)ud\mu+{1\over N}\int_X |\Delta \log u|^2 ud\mu
	-{1\over N}\int_X \left|\Delta \log u-\int_X \Delta\log u ud
	\mu\right|^2 ud\mu\\
	&=&\int_X \left[{|\Delta \log u|^2\over N} -{\bf\Gamma}_2(\nabla \log u, \nabla \log u)\right]ud\mu
	-{1\over N}\int_X \left|\Delta \log u-\int_X \Delta\log u ud
	\mu\right|^2 ud\mu.
\end{eqnarray*}
In particular, if the RCD$(K, N)$-condition holds, using the weak Bochner inequality, we have 
%\begin{eqnarray}
%{\bf \Gamma_2}(\nabla \log u, \nabla \log u) \geq {1\over N}|\Delta \log u|^2+K|\nabla \log u|^2_w. \ \ \ \label{BKN3}
%\end{eqnarray}
%Hence
\begin{eqnarray*}
	\int_X \left[{|\Delta \log u|^2\over N}-{\bf \Gamma_2}(\nabla \log u, \nabla \log u)\right]ud\mu
	\leq -\int_X K|\nabla\log u|_w^2 ud\mu=-KH'.\label{BKN4}
\end{eqnarray*}
This yields
%\begin{eqnarray*}
%	\frac{1}{2}H''+ \frac{H'^2}{N}
%	&=&\int_X \left[{|\Delta \log u|^2\over N}-{\bf \Gamma_2}(\nabla \log u, \nabla \log u)\right]ud\mu-{1\over N}\int_X \left|\Delta \log u-\int_X \Delta\log u ud
%	\mu\right|^2 ud\mu\nonumber\\
%	& &\label{BKN5}\\
%	&\leq& -KH'-{1\over N}\int_X \left|\Delta \log u-\int_X \Delta\log u ud
%	\mu\right|^2 ud\mu.\label{BKN5}
%\end{eqnarray*}
%In particular, we have
\begin{align*}
	\frac{1}{2}H''+\frac{H'^2}{N}+KH'\leq -{1\over N}\int_X \left|\Delta \log u-\int_X \Delta\log u ud
	\mu\right|^2 ud\mu,\label{BNK6}
\end{align*}
which implies the Riccatti EDI $(\ref{EDI})$ in Theorem \ref{EDIKM}. 

\end{remark}
We can also derive an upper bound for the Fisher information on RCD$(K, N)$ spaces.
 
\begin{theorem}\label{Ibound}On RCD$(K, N)$ space, 
	if $u$ satisfies the conditions in Theorem \ref{EDIKM}, then
\begin{equation}
	I(u(t))=\frac{dH}{dt}(u(t))\leq \frac{NK}{e^{2Kt}-1}. \label{IKbound}
\end{equation}
In particular, on RCD$(0, N)$ space (i.e., $K=0$), we have
\begin{equation}
	I(u(t))=\frac{dH}{dt}(u(t))\leq \frac{N}{2t}. \label{IKbound0}
\end{equation} 
\end{theorem}
\begin{proof} Based on the Riccatti Entropy Differential Inequality $(\ref{EDI})$,  the proof of Theorem \ref{Ibound} has been essentially given by Li-Li \cite{LiLi2024TMJ}. To save the length of the paper, we omit the detail.
\end{proof}

\begin{remark} By the same argument as used in \cite{LiLi2024TMJ}, we can also prove $(\ref{IKbound0})$ by  the Li-Yau Harnack inequality on RCD$(0, N)$ spaces. Indeed, by Theorem 1.1 in Jiang \cite{Jiang}, the Li-Yau Harnack inequality holds on RCD$(0, N)$ spaces
\begin{equation*}
	{|\nabla u|^2\over u^2}-{\partial_t u\over u}\leq {N\over 2t}, \label{LYH}
\end{equation*}
which implies
\begin{equation*}
I(u(t))=\int_X \frac{\left|\nabla u\right|^2}{u} d\mu=\int_X\left[\frac{\left|\nabla u\right|^2}{u^2}-\frac{\partial_t u}{u}\right] u d\mu \leq \frac{N}{2 t}.
\end{equation*}

\end{remark} 
%%%%%%%%%%%%%%%--section-----4--%%%%%%%%%%%%%

\section{$W$-entropy formula and its monotonicity on RCD spaces }\label{sect3}

We now prove  Theorem \ref{WNK0} and Theorem \ref{WnNK}. 
 \medskip

\noindent{\bf Proof of Theorem \ref{WNK0}}. By the definition formula of $W_{K, N}$ and the entropy dissipation identities $(\ref{H1})$ and $(\ref{HH2})$ in Theorem \ref{1HH2}, we have  
\begin{equation}
\begin{aligned}
\frac{d}{d t} W_{N, K}(u) &= \frac{d}{d t} W_N(u)+N K\left(1-\frac{K t}{2}\right). \\
&=tH''+2H'-\frac{N}{2t}+N K\left(1-\frac{K t}{2}\right)\\
			&=-2t\int_X {\bf \Gamma_2}(\log u, \log u )ud\mu+2\int_X\|\nabla\log u \|^2_w ud\mu-\frac{N}{2t}+N K\left(1-\frac{K t}{2}\right)\\
			&= -2t\int_X \Big[{\bf \Gamma_2}(\log u, \log u )+\left(\frac{1}{t}-K\right)\Delta \log u+{N\over 4}\Big(\frac{1}{t}-K \Big)^2 -K|\nabla \log u |^2\Big]ud\mu. \label{W-Gamma2A}
\end{aligned}
\end{equation}
Under the RCD$(K, N)$-condition, the weak Bochner inequality yields
\begin{equation*}
\begin{aligned}
\frac{d}{d t} W_{N, K}(u) 
&\leq -2t\int_X \Big[ {|\Delta \log u|^2\over N}+K |\nabla \log u|^2+\left(\frac{1}{t}-K\right)\Delta \log u+{N\over 4}\Big(\frac{1}{t}-K \Big)^2 -K|\nabla \log u |^2\Big] ud\mu\\
			&= -\frac{2t}{N} \int_X\left[\Delta \log u +{N\over 2}\left(\frac{1}{t}-K\right)\right]^2ud\mu.
\end{aligned}
\end{equation*} This finishes the proof of Theorem \ref{WNK0}. \hfill $\square$

\medskip

As a corollary, we have the following result which was oringinally proved by Kuwada and Li \cite{KL}. 

\begin{corollary}
Let $(X, d, \mu)$ be an 
 RCD$(0, N)$ space and $u$ be a positive solution to the heat equation $\partial_t u=\Delta u$. Then
\begin{equation}
	\frac{d}{dt}W_N(u)\leq -\frac{2t}{N}\int_X u\Big(\Delta \log u+\frac{N}{2t}\Big)^2d\mu.\label{WW}
\end{equation}
In particular, we have
\begin{equation*}
	\frac{d}{dt}W_N(u)\leq 0.
\end{equation*}
Moreover, $\frac{d}{dt}W_N(u)=0$ at some $t=\tau$ if and only if 
\begin{equation*}
	\Delta \log u+\frac{N}{2\tau}=0.
\end{equation*}
\end{corollary}

\medskip

\noindent{\bf Proof of Theorem \ref{WnNK}}. 
Under the condition that the Riemannian Bochner formula $(\ref{RBFnN})$ holds, we can prove Theorem \ref{WnNK}  by the same argument as used in Li \cite{Li2012} and S. Li-Li \cite{LiLi2015PJM} for the $W$-entropy formulas (see \eqref{WLi} and \eqref{WfnKm})  on Riemannian manifolds with CD$(K, m)$-condition. For the completeness of the paper, we  reproduce the proof as follows.  By $(\ref{W-Gamma2A})$, we have
	\begin{equation*}
\begin{aligned}
\frac{d}{d t} W_{N, K}(u) 
%&= \frac{d}{d t} W_N(u)+N K\left(1-\frac{K t}{2}\right). \\
%&=tH''+2H'-\frac{N}{2t}+N K\left(1-\frac{K t}{2}\right)\\
%			&=-2t\int_X \Gamma_2(\log u, \log u )ud\mu_V+2\int_X\|\nabla\log u \|^2_w ud\mu-\frac{N}{2t}+N K\left(1-\frac{K t}{2}\right)\\
			%&=-2t\int_X \Gamma_2(\log u, \log u )ud\mu_V-2\int_X\Delta\log u ud\mu-\frac{N}{2t}+N K\left(1-\frac{K t}{2}\right)\\
			= -2t\int_X \Big[{\bf \Gamma_2}(\log u, \log u )+\left(\frac{1}{t}-K\right)\Delta \log u+{N\over 4}\Big(\frac{1}{t}-K \Big)^2 -K|\nabla \log u |^2\Big]ud\mu.
%			&=-2t\int_X\left[\left\|\nabla^2 \log u\right\|^2_{\rm HS}+\operatorname{Ric}(L)(\nabla \log u, \nabla \log u ) \right]d\mu_V+2\int_X\|\nabla\log u \|^2_w ud\mu\\
%			& \hskip2cm -\frac{N}{2t}+N K\left(1-\frac{K t}{2}\right).
\end{aligned}
\end{equation*}
Splitting 
$$\Delta \log u={\rm Tr}\nabla^2 \log u+(\Delta-{\rm Tr} \nabla^2) \log u,$$
we have
\begin{equation*}
\begin{aligned}
\frac{d}{d t} W_{N, K}(u) &= -2t\int_X \Big[{\bf \Gamma_2}(\log u, \log u )
+\left(\frac{1}{t}-K\right){\rm Tr}\nabla^2 \log u+{n\over 4}\Big(\frac{1}{t}-K \Big)^2-K|\nabla \log u |^2\Big]ud\mu\\
&\hskip1cm -2t\int_X \left[\left(\frac{1}{t}-K\right)(\Delta -{\rm Tr}\nabla^2)\log u 
+{N-n\over 4}\Big(\frac{1}{t}-K \Big)^2 \right]ud\mu.
\end{aligned}
\end{equation*}
Under the assumption that the Riemannian Bochner formula $(\ref{RBFnN})$ holds, we have
\begin{eqnarray*}
& &{\bf \Gamma_2}(\log u, \log u)+\left({1\over t}-K\right){\rm Tr}\nabla^2 \log u+{n\over 4}\left({1\over t}-K\right)^2\\
& &\hskip1.5cm =\left\|\nabla^2 \log u+{1\over 2}\left({1\over t}-K\right){\bf g}\right\|^2_{\rm HS}+\operatorname{\bf Ric_{N, n}}(\Delta)(\nabla\log u, \nabla\log u)+{|(\Delta-{\rm Tr}\nabla^2)\log u|^2\over N-n}.\end{eqnarray*}
This yields 
\begin{equation*}
\begin{aligned}
\frac{d}{d t} W_{N, K}(u, t)&=  -2t \int_M \left[\left\|\nabla^2 \log u+{1\over 2}\left({1\over t}-K\right){\bf g}\right\|^2_{\rm HS}+(\operatorname{\bf Ric_{N, n}}(\Delta)-K{\bf g})(\nabla \log u, \nabla \log u) \right]u d \mu\\
& \hskip0.5cm -2t\int_X \left[{|(\Delta-{\rm Tr}\nabla^2)\log u)|^2\over N-n} +\left({1\over t}-K\right)(\Delta-{\rm Tr}\nabla^2) \log u+{N-n\over 4}\left(\frac{1}{ t}-K\right)^2 \right]ud\mu\\
&=-  2 t\int_X \left[ \left\|\nabla^2 \log u+{1\over 2}\left({1\over t}-K\right){\bf g}\right\|^2_{\rm HS}+\left(\operatorname{\bf Ric_{N, n}}(\Delta)-K{\bf g}\right)(\nabla \log u, \nabla \log u)\right] u d \mu\\
& \hskip0.5cm -\frac{2 t}{N-n} \int_M\left[({\rm Tr}\nabla^2-\Delta)\log u-\frac{N-n}{2} \left(\frac{1}{ t}-K\right)\right]^2 u d \mu.
\end{aligned}
\end{equation*}
This completes the proof of Theorem \ref{WnNK}. \hfill $\square$

%
%\begin{theorem}\label{WWKN1}
%	Let $(X, d, \mu)$ be an 
% RCD$^\ast(K, N)$ space and $u$ be a solution to the heat equation $\partial_t u=\Delta u$. Then
%\begin{equation}
%	\frac{d}{dt}W_{N,K}(u)\leq -\frac{2t}{N}\int_X u\Big(\Delta \log u+\frac{N}{2t}-\frac{NK}{2}\Big)^2d\mu.\label{WKN}
%\end{equation}
%In particular, it holds
%\begin{equation*}
%	\frac{d}{dt}W_{N,K}(u)\leq 0.
%\end{equation*}
%Moreover, $\frac{d}{dt}W_{N, K}(u)=0$ at some $t=\tau$ if and only if 
%\begin{equation*}
%	\Delta \log u+\frac{N}{2\tau}-{NK\over 2}=0.
%\end{equation*}
%\end{theorem}
%\begin{proof}
% The proof is essentially provided  by Li-Li in \cite{LiLi2015PJM}. For the convenience of the readers,  we reproduce
% it below. By a direct calculation, we have 
% 
% \begin{equation*}
%	\frac{d}{dt}W_{N,K}(u)=t\frac{d^2}{dt^2}H+2\frac{d}{dt}H-\frac{N}{2t}+NK-\frac{1}{2}NK^2t.
%\end{equation*} 
%Using  the first order entropy dissipation formula\eqref{H1} and the second order  entropy dissipation inequality\eqref{H2}  on RCD$^\ast(-K,N)$ spaces, we have
% 
%	\begin{equation*}
%		\begin{split}
%			\frac{d}{dt}W_{N,K}(u)
%			&\leq-\frac{2t}{N}\int_X u\Big(\Delta \log u+\frac{N}{2t}\Big)^2d\mu+2Kt\int_X u|\nabla\log u |_w^2d\mu +NK-\frac{1}{2}NK^2t\\
%			&=-\frac{2t}{N}\int_X u\Big(\Delta \log u+\frac{N}{2t}-\frac{NK}{2}\Big)^2d\mu+2Kt\int_X u\Delta \log u d\mu+2Kt\int_X u|\nabla \log u |^2_wd\mu\\
%			&=-\frac{2t}{N}\int_X u\Big(\Delta \log u+\frac{N}{2t}-\frac{NK}{2}\Big)^2d\mu.
%		\end{split}
%	\end{equation*}
%	This finishes the proof. 
%\end{proof}

\medskip

\noindent{\bf Proof of the rigidity part of Theorem \ref{WnNK=0} }. The proof of the rigidity part  of Theorem \ref{WnNK=0} has been essentially given in Kuwada-Li \cite{KL} (which modifies a part of the proof of the rigidity part of Theorem \ref{Lirigidity} in \cite{Li2012}). For the convenience of the readers, we reproduce it here. By \eqref{dW=0}, we see that ${d\over dt}W_N(u(t))\leq 0$ for all $t>0$.  If 
${d\over dt}W_{N} (u)=0$ holds at some $t=t_*>0$ for the fundamental solution of the heat equation $\partial_t u=\Delta u$,  we have

\begin{equation*}\label{GGGG3}
\Delta \log u+{N\over 2t}=0, \ \ \ {\bf Ric_{N, n}}(\Delta)=0.
	\end{equation*}
	In particular, the Fisher information at $t=t_*$ satisfies 

\begin{equation}\label{GGGG4}
I(u(t_*))=-\int_X \Delta \log u ud\mu={N\over 2t}.
	\end{equation}
	As in \cite{KL}, we can further prove that 
\begin{equation}\label{GGGG5}
I(u(t))={N\over 2t}\ \ \ \forall t\in (0, t_*).
	\end{equation}
	Indeed, following \cite{KL}, we introduce 
	$$h(t)={N\over 2}-tI(u(t)).$$
	By \eqref{IKbound0} in Theorem \ref{Ibound}, we have
	$h(t)\geq 0$ for all $t>0$.  
	Moreover, from the definition formula of $W_N(u(t))$, we have
%	$$W_N(u(t))={d\over dt}(tH_N(u(t))=H(u(t))+t {d\over dt}H(u(t))-H(u_N(t))-t{d\over dt} H(u_N(t)),$$
%	where $$H(u_N(t))={N\over 2}\log(4\pi et).$$
%	Thus
%	
%\begin{equation*}\label{WHh}
%\begin{aligned}
%W_N(u(t))&=H(u(t))+tI(u(t))-{N\over 2}\log(4\pi e t)-{N\over 2}\\
%&=H(u(t))-h(t)-{N\over 2}\log(4\pi e t).
%\end{aligned}
%\end{equation*}
%Hence

\begin{equation*}\label{WHh}
\begin{aligned}
{d\over dt}W_N(u(t))={d\over dt}H(u(t))-{N\over 2t}-h'(t)= I(u(t))-{N\over 2t}-h'(t)=-{h(t)\over t}-h'(t).
\end{aligned}
\end{equation*}
Thus ${d\over dt}(th(t))=h(t)+th'(t))=-t{d\over dt}W_N(u(t))\geq 0$. This yields, for any $0<t<t_*$, we have $0\leq th(t)\leq t_*h(t_*)$. 
By \eqref{GGGG4}, $I(u(t_*))={N\over2t_*}$. Thus $h(t)=0$ for all $t\in (0, t_*]$. Equivalently, \eqref{GGGG5} holds. 
%\begin{equation*}\label{NtI}
%I(u(t))={N\over 2t}, \ \ \ \forall\ t\in (0, t_*).
%\end{equation*}
By \cite{JLZ}, the  Li-Yau Harnack inequality holds on RCD$(0, N)$ space
\begin{equation}\label{LYHNK0}
{|\nabla u|^2\over u^2}-{\Delta u\over u}\leq {N\over 2t}.
  \end{equation}
  This yields, for all $t>0$ we have (see \eqref{IKbound0} in Theorem \ref{Ibound})
  \begin{equation*}\label{LYHNK=0}
I(u(t))\leq {N\over 2t}.
  \end{equation*}
Thus, the equality $I(u(t))={N\over 2t}$ on $(0, t_*]$ implies that the Li-Yau Harnack inequality \eqref{LYHNK0} is indeed an equality for $\mu$-a.e. $x\in X$ and for all $t\in (0, t_*]$. From this fact, as already proved  in \cite{KL} (see $(4.4)$ in \cite{KL}), for any nice function 
$f\in D(\Delta)\cap L^1(X, \mu)$  with $df, d|\nabla f|\in L^1(X, \mu)$, we have
\begin{eqnarray*}\label{ILY}
-\int_X \langle \nabla d^2(x, x_0), \nabla f \rangle d\mu= 2N\int_X fd\mu.
  \end{eqnarray*}

%
%\begin{eqnarray*}\label{ILY}
%\int_X \log  u \Delta f_0 d\mu=-\int_X \Big\langle \nabla\log u, \nabla f_0\Big\rangle d\mu=-\int_X \Big\langle {\Delta u\over u}-{|\nabla u|^2\over u^2}\Big\rangle fd\mu.
%  \end{eqnarray*}
%Thus  for all $t>0$, it holds
%\begin{equation}\label{ILYB}
%\int_X \log  u \Delta f_0d\mu={N\over 2t}\int_X fd\mu.
%  \end{equation}
%  
%  On the other hand,  the two-sides estimates $(\ref{JLZB})$ of the fundamental solution to the heat equation $\partial_t u=\Delta u$ on RCD$(0, N)$ spaces says that, for any $\varepsilon>0$ and any $x, y\in X$, we have
%  \begin{equation}\label{JLZB0}
% {1\over C_1(\varepsilon)V_x(\sqrt{t})}\exp\left(-\frac{d^2(x, y)} {(4-\epsilon)t}\right)
% \leq p_t(x, y)\leq 
% {C_1(\varepsilon)\over \mu(B(x, \sqrt{t})}\exp\left(-\frac{d^2(x, y)} {(4+\epsilon)t}\right).
% \end{equation}
% Recalll the Bishop-Gromov  volume comparison inequality on RCD$(0, N)$ space
%\begin{equation}\label{KL212}
% {\mu(B(x, R))\over \mu(B(x, r))}\leq \left({R\over r}\right)^N, \ 
% \ \ \forall 0<r<R, x\in X,
% \end{equation}
% Combining \eqref{JLZB0} and \eqref{KL212}, we have the Varadhan type small time heat kernel asymptotic behavior on RCD$(0, N)$ spaces
%
%\begin{equation}\label{small}
%\lim\limits_{t\rightarrow 0}t\log p_t(x, o)=-{d^2(x, o)\over 4}\ \ \ {\rm uniformly\ on\ each\ bounded\ set}.
%\end{equation}\\
%Hence
%
%\begin{equation*}\label{ILYB}
%-{N\over 2}\int_X f d\mu=\lim\limits_{t\rightarrow 0+} \int_X (t\log  p_t(x, o) \Delta f (x)d\mu(x)=-{1\over 4}
%\int_X d^2(x, o)\Delta fd\mu.
%  \end{equation*}
In other words, it holds
  $$\Delta d^2(x, x_0)=2N.$$

By Theorem 6.1 in \cite{GV2023} and Theorem 2.8 in \cite{HP}, this proves the  rigidity theorem on the $W$-entropy on RCD$(0, N)$ spaces.  \hfill $\square$

\begin{remark}
In \cite{LiLi2015PJM} (see \cite{Li2012} for $K=0$), it was proved that, in the case where $(M, g,\mu)$ is an $n$-dimensional 
compact weighted Riemannian manifold, $d\mu=e^{-\phi}$ with $\phi\in C^4(M)$, and if we assume  that 
$$Ric_{m, n}(L)=Ric+\nabla^2\phi-{\nabla\phi\otimes \nabla\phi\over m-n}\geq K,$$
where $L=\Delta-\nabla\phi\cdot\nabla$ is the Witten Laplacian on the weighted Riemannian manifold $(M,g,\mu)$, $K\in \mathbb{R}$ and $N\geq n$, then there exists a positive and smooth
function $u=e^{-\frac v2}$  such that $v$ achieves the optimal logarithmic Sobolev
constant $\mu_K(t)$ defined by
\begin{equation}\label{LSI}
\mu_K(t):=\inf\left\{W_{N,K}(u,t):\int_M\frac{e^{-v}}{(4\pi t)^{N/2}}d\mu=1\right\},
\end{equation}
where
$$
W_{N,K}(u)=\int_M\left(t|\nabla v|^2+v-N\left(1-{1\over 2}Kt\right)^2\right)\frac{e^{-v}}{(4\pi t)^{\frac N2}}d\mu
$$
 is the $W$-entropy on the weighted Riemannian manifold with CD$(K,N)$ condition.
Indeed, by a
similar argument as used in Perelman \cite{P1}, it is shown in \cite{Li2012, LiLi2015PJM} that the minimization
problem \eqref{LSI} has a non-negative minimizer $u \in H^1(M, \mu)$, which satisfies the following 
Euler-Lagrange equation
$$
-4tLu-2u\log u-N\left(1-\frac {K}{2t}\right)^2u=\mu_K(t)u.
$$
In the proof of this result, S. Li and Li used the regularity theory of elliptic PDEs on Riemannian manifolds to derive $u\in C^{1,\alpha}(M)$ for $\alpha\in(0,1)$. Then, an argument due to Rothaus \cite{Roth}
allows them to prove that $u$ is strictly positive and smooth. Hence $v = -2\log u$ is also
smooth. It would be interesting to study the similar question and to see what happens on RCD$(K,N)$ spaces. To this end, one  need to establish  the regularity theory of elliptic PDEs on RCD spaces. It will be an interesting problem for the study in the future.  See also Kuwada-Li \cite{KL} for the case $K=0$.
\end{remark}

\section{Shannon entropy power on RCD spaces}\label{sect4}

In his seminal work on information theory \cite{Sh} in 1948, Shannon established the channel coding theorem through the introduction of entropy power. Since then, entropy power has been a fundamental topic of research in information theory. Shannon was the first one to propose and prove the entropy power inequality for discrete random variables \cite{Sh} and also introduced the concept of entropy power for continuous random vectors. 

More precisely,  let $X$ be an  random vector on $\mathbb{R}^n$ whose  the probability distribution is $f(x)dx$, and 
   \begin{equation*}
   	H(X):= - \int_{\mathbb{R}^n} f (x) \log f (x)dx ,
   \end{equation*}
    be the  Shannon  entropy of  the law of $X$.  The Shannon entropy power of $X$ is defined as follows 
\begin{equation*}
	\mathcal{N} (X) := e ^{\frac{2}{n} H(X)}.
\end{equation*}

Then the Shannon entropy power inequality states that, for two independent random vectors $X$ and $Y$ on $\mathbb{R}^n$, the following inequality holds:  
\begin{equation}
	\mathcal{N}(X+Y)\geq \mathcal{N}(X)+\mathcal{N}(Y), \label{NXY1}
\end{equation}
with equality if and only if the covariance matrix of $X$ is proportional to the covariance matrix of $Y$. For more 
rigorous proof of the Shannon entropy power inequality, see 
Stam  \cite{Stam} and Blachman \cite{Bla}. See also \cite{CT}.

By Dembo \cite{Dem1} and Dembo-Cover-Thomas \cite{Dem2}, the  Shannon-Stam Entropy Power Inequality $(\ref{NXY1})$  is equivalent to the following inequalities for $H$ and for $\mathcal{N}$:  
for any $\lambda\in (0, 1)$, it holds
\begin{eqnarray}
H(\sqrt{\lambda} X+\sqrt{1-\lambda}Y)\leq \lambda H(X)+(1-\lambda) H(Y),  \label{HXY}
\end{eqnarray}
or 
\begin{eqnarray}
\mathcal{N}(\sqrt{\lambda} X+\sqrt{1-\lambda}Y)\geq \lambda \mathcal{N}(X)+(1-\lambda) \mathcal{N}(Y). \label{NXY2}
\end{eqnarray}

The Shannon-Stam Entropy Power Inequality  $(\ref{NXY1})$ is an analogue of the Brunn-Minkowski isoperimetric inequality in convex geometry: for any 
$A, B\in  \mathscr{B}(\mathbb{R}^n)$, letting $A+B:=\{x+y: x\in A,  y\in B\}$, then 
\begin{eqnarray*}
V^{1/n}(A+B)\geq V^{1/n}(A)+V^{1/n}(B),
\end{eqnarray*}
where $V(A)$, $V(B)$ and $V(A+B)$ denote the volume of the sets $A$, $B$ and $A+B$ respectively. Indeed,  EPI $(\ref{NXY1})$ is equivalent to 
\begin{eqnarray*}
e^{2H(X+Y)\over n}\geq e^{2H(X)\over n}+e^{2H(Y)\over n}. 
\end{eqnarray*}
For a unified  proof of the Shannon-Stam Entropy Power Inequality and the Brunn-Minkowski isoperimetric inequality, see Dembo  et al. \cite{Dem2}. See also \cite{CT}.

 In 1985, Costa \cite{Co} proved the concavity of the Shannon entropy power along the 
heat equation on $\mathbb{R}^n$. More precisely, let $u(x, t)$ be a solution to
 the heat equation on $\mathbb{R}^n$, i.e., 
 \begin{equation*}
 	\partial_t u = \Delta u, \quad u(x, 0) = f (x).
 \end{equation*} 
 Let $H(u(t)) = -\int_M  u \log udx$ be the Shannon entropy associated to the heat distribution $u(x, t)dx$ at time $t$ and $\mathcal{N} (u(t)) = e ^{\frac{2}{n} H(u(t))}$ be the Shannon entropy power. Then 
 \begin{equation}\label{Entropy power}
 	 \frac{d^2 \mathcal{N}(u(t))}{dt^2} \leq 0.
 \end{equation}

Using an argument based on the Blachman-Stam inequality \cite{Bla}, the original proof of the Entropy Power Concave Inequality (EPCI) $(\ref{Entropy power})$ has been simplified by Dembo et al. \cite{Dem1, Dem2} and Villani \cite{villani2000}. 
In \cite{villani2000},  Villani pointed out the possibility of extending the Entropy Power Concave Inequality (EPCI) $(\ref{Entropy power})$ to Riemannian manifolds with non-negative Ricci curvature using the $\Gamma_2$-calculation. In \cite{ST2014}, Savar\'{e} and Toscani  proved 
that the R\'{e}nyi entropy power is also concave along the porous media equation and the fast diffusion equation, given by $\partial_t u=\Delta u^p$ with $ p> 1-\frac{2}{n}  $ on $\mathbb{R}^n $.

In a very recent paper \cite{LiLi2024TMJ},  Li and Li  derived the corresponding Entropy Differential Inequality ($\operatorname{EDI}$)  and proved the $K$-concavity of the entropy power for the heat equation associated with the Witten Laplacian on complete Riemannian manifolds with CD$(K, m)$-condition.
They also discovered the NIW formula, which links the $W$-entropy to entropy power $N$,the Fisher information $I$ and the Perelman $W$-entropy for the heat equation associated with the Witten Laplacian on complete Riemannian manifolds with CD$(K, m)$-condition.  One of the important feature is that one can use the NIW formula to provide a second proof of the concavity of the Shannon entropy power. As an application, they established the corresponding entropy isoperimetric inequality on Riemannian manifolds. Furthermore, Li and Li \cite{LiLiRenyi} extended these results to 
R\'enyi entropy, proving analogous formulas for the porous media equation and fast diffussion equation associated with the Witten Laplacian on complete Riemannian manifolds with CD$(K, m)$-condition.

The purpose of this paper is to extend Li-Li's $K$-concavity of the Shannon entropy power from complete Riemannian manifolds with CD$(K, m)$-condition to metric measure spaces satisfying the Riemannian curvature-dimension. condition RCD$(0, N)$. Let us mention that Erbar, Kuwada and Sturm \cite{EKS2015} 
have proved the concavity of the Shannon entropy power for the heat semigroup on metric measure spaces satisfying the Riemannian curvature-dimension condition RCD$(0, N)$. Our main result can be also regarded as a natural extension of Ebar-Kuwada-Sturm's result from RCD$(0, N)$ spaces to RCD$(K, N)$ spaces.

\medskip
 The following NIW formula was essentially proved in  Li-Li \cite{LiLi2024TMJ} on Riemannian manifolds and can be easily extended to  RCD spaces. It indicats an interesting relationship between the $W$-entropy $W$, the Fisher information $I$ and the Shannon entropy power $N$. 

\begin{theorem}\label{thmNIW} Let $\mathcal{N}(u)=e^{2H(u)\over N}$ be the Shannon entropy power associated with the linear heat equation $\partial_t u=\Delta u$
	 on RCD$(K, N)$ space. Then  the following NIW formulas hold on  metric measure spaces
	\begin{equation*}
		\frac{d^2\mathcal{N}}{dt^2}=\frac{2\mathcal{N}}{N}\Big
		[\frac{2}{N}\Big(I-\frac{N(1-Kt)}{2t}\Big)^2+\frac{1}{t}\frac{dW_{N,K}}{dt}\Big].
	\end{equation*}
	In particular, on  RCD$(0, N)$ space, we have
	\begin{equation*}
		\frac{d^2\mathcal{N}}{dt^2}=\frac{2\mathcal{N}}{N}\Big[\frac{2}{N}\Big(I-\frac{N}{2t}\Big)^2
		+\frac{1}{t}\frac{dW_N}{dt}\Big].
	\end{equation*}
	\end{theorem}
		\begin{proof}
		{The proof follows the same argument as in Li-Li \cite{LiLi2024TMJ}.}
\end{proof}

We now state the main result of this section. 

\begin{theorem}\label{K-concavity} Let  $\mathcal{N}(u)=e^{2H(u)\over N}$ be the Shannon entropy power associated with the linear heat equation $\partial_t u=\Delta u$ on an RCD$(K, N)$ space. Then 
\begin{equation*}
	\frac{d^2\mathcal{N}}{dt^2}+2K\frac{d\mathcal{N}}{dt} \leq 0.
\end{equation*}
%Equivalently,
%\begin{equation*}
%	{d\over dt}\left[e^{-2Kt}{d\over dt}\left(e^{2Kt}\mathcal{N}\right) \right]\leq 0.
%\end{equation*}
\end{theorem} 
\begin{proof} Note that 
	\begin{eqnarray*}
	\mathcal{N}'&=&\frac{2H'}{N}\mathcal{N},\\
	\mathcal{N}''&=&\frac{2H''}{N}\mathcal{N}+\frac{4H'^2}{N^2}\mathcal{N}=\frac{2\mathcal{N}}{N}\Big(H''+\frac{2H'^2}{N}\Big). 
	\end{eqnarray*}
Using $\operatorname{EDI}(K,N)$ in \eqref{EDI}, we have
\begin{eqnarray*}
\mathcal{N}''\leq -\frac{4K}{N}\mathcal{N}H'=-2K\mathcal{N}'.
	\end{eqnarray*}
\end{proof}

\medskip

We can now use the NIW formula to provide an alternative proof of Theorem \ref{K-concavity}.

%
% The following results provide us a new way to understand the connection between the
% Shannon entropy power and the $W$-entropy.
%\begin{theorem}
%	Let $\mathcal{N}(u)=e^{2H(u)\over N}$. Then  the following $\operatorname{NIW}$ formula holds on  general metric measure spaces
%	\begin{equation*}
%		\frac{d^2\mathcal{N}}{dt^2}=\frac{2\mathcal{N}}{N}\Big[\frac{2}{N}\Big(I-
%		\frac{N}{2t}\Big)^2+\frac{1}{t}\frac{dW_N}{dt}\Big],
%	\end{equation*}
%	and
%	\begin{equation*}
%		\frac{d^2\mathcal{N}}{dt^2}+2K\frac{d\mathcal{N}}{dt}=\frac{2\mathcal{N}}{N}
%		\Big[\frac{2}{N}\Big(I-\frac{N(1-Kt)}{2t}\Big)^2+\frac{1}{t}\frac{dW_{N,K}}{dt}\Big].
%	\end{equation*}
%\end{theorem}

\begin{theorem}\label{ECPKN}
Let $(X, d, \mu)$ be an 
 RCD$(K,N)$ space. Then
\begin{equation*}
	\begin{split}
		\frac{N}{2\mathcal{N}}\frac{d^2\mathcal{N}}{dt^2}\leq& -2K\int_X |\nabla \log u |_w^2d\mu
		-\frac{2}{N}\int_X \Big[\Delta \log u-\int_X u \Delta \log u  d\mu \Big]^2 u d\mu.
		\end{split}
	\end{equation*}
As a consequence, we have
	\begin{eqnarray*}
	\frac{N}{2\mathcal{N}}\Big[\frac{d^2\mathcal{N}}{dt^2}+2K\frac{d\mathcal{N}}{dt}\Big]\leq -\frac{2}{N}\int_X \Big[\Delta \log u-\int_X u\Delta \log u d\mu \Big]^2 ud\mu.
	\end{eqnarray*}
	In particular, on 
 RCD$(K,N)$ space, we have
 \begin{eqnarray*}
	\frac{d^2\mathcal{N}}{dt^2}+2K\frac{d\mathcal{N}}{dt}\leq 0,
		\end{eqnarray*}
	and the left hand side vanishes if and only if  
	$$\Delta \log u=\int_X u \Delta \log u d\mu.$$
	\end{theorem}
%	\begin{align*}
%		W_{N,K}&=W_N-N(1-\frac{Kt}{2})^2\\
%		&=W_N-N-\frac{K^2t^2}{4}+NKt\\
%		\frac{dW_{N,K}}{dt}&=\frac{dW_N}{dt}-\frac{NK^2t}{2}+NK\\
%		\frac{1}{t}\frac{dW_{N,K}}{dt}&=\frac{1}{t}\frac{dW_N}{dt}-\frac{NK^2}{2}+\frac{NK}{t}
%	\end{align*}
%	\begin{align*}
%		&\frac{d^2\mathcal{N}}{dt^2}+2K\frac{d\mathcal{N}}{dt}\\&=\frac{2\mathcal{N}}{N}\Big
%		[\frac{2}{N}\Big(I-\frac{N(1-Kt)}{2t}\Big)^2+\frac{1}{t}\frac{dW_{N,K}}{dt}\Big]\\
%		&=\frac{2\mathcal{N}}{N}\Big
%		[\frac{2}{N}\Big(I-\frac{N}{2t}+\frac{NK}{2}\Big)^2+\frac{1}{t}\frac{dW_N}{dt}-\frac{NK^2}{2}+\frac{NK}{t}\Big]\\
%		&=\frac{2\mathcal{N}}{N}\Big
%		[\frac{2}{N}\Big(I-\frac{N}{2t}\Big)^2+\frac{2}{N}\frac{N^2K^2}{4}+\frac{2}{N}\Big(I-\frac{N}{2t}\Big)NK+\frac{1}{t}\frac{dW_N}{dt}-\frac{NK^2}{2}+\frac{NK}{t}\Big]\\
%		&=\frac{2\mathcal{N}}{N}\Big[\frac{2}{N}\Big(I-\frac{N}{2t}\Big)^2
%		+\frac{1}{t}\frac{dW_N}{dt}\Big]+2K\frac{d\mathcal{N}}{dt}
%	\end{align*}
\begin{proof} Note that
	\begin{align*}
		H_{N,K}(u)&=H_N+\frac{N}{2}Kt\Big(1-\frac{1}{6}Kt\Big),
			\end{align*}
			and
	\begin{align*}
		W_{N,K}(u)&=\frac{d}{dt}(tH_{N,K})\\
		&=H_{N,K}+t\frac{d}{dt}H_{N,K}\\&=H_N+\frac{N}{2}Kt\Big(1-\frac{1}{6}Kt\Big)+t\frac{d}{dt}H_{N}+t\frac{d}{dt}\left[\frac{N}{2}Kt\Big(1-\frac{1}{6}Kt\Big)\right]\\
		&=H_N+{1\over t}\frac{d}{dt}H_N+NKt-\frac{1}{4}NK^2t^2.
	\end{align*}
	This yields
	\begin{align*}
		\frac{d}{dt}W_{N,K}(u)=\frac{d}{dt}W_{N}(u)+NK-\frac{1}{2}NK^2t,
	\end{align*}
	and
	\begin{align*}
		\frac{1}{t}\frac{d}{dt}W_{N,K}(u)=\frac{1}{t}\frac{d}{dt}W_{N}+\frac{NK}{t} -\frac{1}{2}NK^2.
	\end{align*}By the entropy dissipation formulas, we have
		\begin{eqnarray}
			\frac{d}{dt}W_N(u)&=&tH''+2H'-\frac{N}{2t}\nonumber\\
			&=&-2t\int_X {\bf\Gamma}_2(\nabla\log u,\nabla \log u )ud\mu+2\int_X\|\nabla\log u \|^2_w ud\mu-\frac{N}{2t}\nonumber\\
			&=& -2t\int_X \Big[{\bf\Gamma}_2(\nabla\log u,\nabla \log u )+\frac{\Delta \log u}{t}+\frac{N}{4t^2} \Big]ud\mu. \label{W-Gamma2}
	\end{eqnarray}	
	When $X=M$ is a Riemannian manifold, the Bochner identity formula implies	
		\begin{eqnarray}
\frac{d}{dt}W_N(u)
			&=&-2t\int_X \Big[\|\nabla^2 \log u \|^2_{\operatorname{HS}}+{\bf Ric}(\nabla\log u,\nabla \log u ) +\frac{\Delta \log u}{t}+\frac{N}{4t^2} \Big]ud\mu\nonumber\\
			&=&-2t\int_X \Big[\|\nabla^2 \log u+\frac{1}{2t} \|^2_{\operatorname{HS}}+{\bf Ric}(\nabla\log u,\nabla \log u )  \Big]ud\mu.\label{W-Ni}
		\end{eqnarray}
		This is the $W$-entropy formula when $K=0$ due to Ni \cite{Ni1} and Li \cite{Li2012}. Moreover, when $K\neq 0$, we have
		\begin{eqnarray*}
		\frac{d}{dt}W_{N,K}(u)&=&\frac{d}{dt}W_{N}+NK-\frac{1}{2}NK^2t\\
		&=&-2t\int_X \Big[\|\nabla^2 \log u+\frac{1}{2t} \|^2_{\operatorname{HS}}+{\bf Ric}(\nabla\log u,\nabla \log u )  \Big]ud\mu+NK-\frac{1}{2}NK^2t\\
		&=&-2t\int_X \Big[\|\nabla^2 \log u+\frac{1}{2t} \|^2_{\operatorname{HS}}-{NK\over 2t}+{NK^2\over 4}+{\bf Ric}(\nabla\log u,\nabla \log u )  \Big]ud\mu\\
		&=&-2t\int_X \Big[\|\nabla^2 \log u+\frac{1}{2t}-{K\over 2}\|^2_{\operatorname{HS}}+K\Delta \log u+{\bf Ric}(\nabla\log u,\nabla \log u )  \Big]ud\mu\\
		&=&-2t\int_X \Big[\|\nabla^2 \log u+\frac{1}{2t}-{K\over 2}\|^2_{\operatorname{HS}}+({\bf Ric}-K)(\nabla\log u,\nabla \log u )  \Big]ud\mu.
	\end{eqnarray*}
This is the $W$-entropy formula when $K\neq 0$ due to Li-Xu \cite{LiXu} and Li-Li \cite{LiLi2015PJM, LiLi2018JFA}.

	Next, the NIW formula implies
	\begin{equation*}\label{N1}
			\begin{split}
				\frac{N}{2\mathcal{N}}\frac{d^2\mathcal{N}}{dt^2}
				&=\frac{2}{N}\Big(I-\frac{N}{2t}\Big)^2+\frac{1}{t}\frac{dW_{N}}{dt}\\
				&= \frac{2}{N}\Big(I-\frac{N}{2t}\Big)^2	-2\int_X \Big[\|\nabla^2 \log u+\frac{1}{2t} \|^2_{\operatorname{HS}}+{\bf Ric}(\nabla\log u,\nabla \log u \Big]u d\mu	\\
				&=\frac{2}{N}I^2-2\int_X \Big[\|\nabla^2\log u \|^2_{\operatorname{HS}} +{\bf Ric}(\nabla\log u,\nabla \log u )\Big]u d\mu\\
				&=\frac{2}{N}\Big(\int_X u\Delta\log u d\mu\Big)^2-2\int_X \Big[{\bf\Gamma}_2(\nabla\log u,\nabla \log u )\Big]ud\mu.
							\end{split}
		\end{equation*}
		This yields, if the CD$(0, N)$-condition holds, i.e., the Bochner inequality formula holds\begin{eqnarray}
{\bf\Gamma}_2(\nabla \log u, \nabla \log u) \geq {1\over N}|\Delta \log u|^2,\ \ \ \label{BKNK=0}
\end{eqnarray}
we have
\begin{align*}
	\frac{N}{2\mathcal{N}}\frac{d^2\mathcal{N}}{dt^2}&\leq \frac{2}{N} \Big(\int_X u\Delta\log u d\mu\Big)^2-\frac{2}{N}\int_X(\Delta\log u )^2u d\mu\\
	&=-\frac{2}{N}\int_X \Big[\Delta \log u-\int_X u\Delta \log u d\mu \Big]^2 ud\mu.
\end{align*}

In general, for $K\neq 0$, we have
\begin{align*}
\frac{N}{2\mathcal{N}}\Big[\frac{d^2\mathcal{N}}{dt^2}+2K\frac{d\mathcal{N}}{dt}\Big]
&=\frac{N}{2\mathcal{N}}\frac{d^2\mathcal{N}}{dt^2}+{NK\over\mathcal{N}}\frac{d\mathcal{N}}{dt}\\
&=\frac{2}{N}\Big(\int_X u\Delta\log u d\mu\Big)^2-2\int_X \Big[{\bf\Gamma}_2(\nabla\log u,\nabla \log u )-K|\nabla\log u |_w^2\Big]ud\mu
\end{align*}
This yields, if the CD$(K, N)$-condition holds, i.e., the Bochner inequality formula holds\begin{eqnarray}
{\bf\Gamma}_2(\nabla \log u, \nabla \log u) \geq {1\over N}|\Delta \log u|^2+K|\nabla \log u|^2_w,\ \ \ \label{BKN3}
\end{eqnarray}
we have
\begin{align*}
	\frac{N}{2\mathcal{N}}\Big[\frac{d^2\mathcal{N}}{dt^2}+2K\frac{d\mathcal{N}}{dt}\Big] 
	%&\leq \frac{2}{N} \Big(\int_X u\Delta\log u d\mu\Big)^2-\frac{2}{N}\int_X(\Delta\log u )^2u d\mu\\
	\leq -\frac{2}{N}\int_X \Big[\Delta \log u-\int_X u\Delta \log u d\mu \Big]^2 ud\mu.
\end{align*}

%		\begin{eqnarray}
%			\|\nabla^2 \log u+\frac{1}{2t} \|^2=\frac{1}{N}(\Delta\log u+\frac{N}{2t})^2+\|\nabla^2 \log u-\frac{\Delta\log u}{N} \|^2_{\operatorname{HS}}
%		\end{eqnarray}
%		Note that
%		\begin{equation*}\label{N3}
%			\Big(I-\frac{N}{2t}\Big)^2-\int_X \Big(\Delta \log u+\frac{N}{2t}\Big)^2ud\mu=-\int_X\Big[\Delta \log u-\int_X u \Delta \log u d\mu\Big]^2ud\mu.
%		\end{equation*}
%		 Hence
%		\begin{align*}
%			\label{2NK}
%			%\begin{split}
%		\frac{N}{2\mathcal{N}}\Big[\frac{d^2\mathcal{N}}{dt^2}+2K\frac{d\mathcal{N}}{dt}\Big]&= 	-\frac{2}{N}\int_X \Big[\Delta \log u-\int_X u\Delta \log u d\mu \Big]^2 ud\mu-2\int_X\|\nabla^2 \log u-\frac{\Delta\log u}{N} \|^2_{\operatorname{HS}}ud\mu\\
%		&-2\int_X\Big[ \Ric(\nabla\log u,\nabla \log u )-K|\nabla\log u |_w^2  \Big]u d\mu\\
%		&=-\frac{2}{N}\int_X \Big[\Delta \log u-\int_X u\Delta \log u d\mu \Big]^2 ud\mu+\frac{2}{N}\int_X (\Delta\log u)^2d\mu\\
%		&-2\int_X\Big[\|\nabla^2\log u\|^2_{\operatorname{HS}}+\Ric(\nabla\log u,\nabla \log u )-K|\nabla\log u |_w^2  \Big]u d\mu\\
%		&=-\frac{2}{N}\int_X \Big[\Delta \log u-\int_X u\Delta \log u d\mu \Big]^2 ud\mu\\
%		&-2\int_X\Big[\Gamma_2(\nabla\log u,\nabla\log u) -K|\nabla\log u |_w^2-\frac{1}{N}\Delta\log u   \Big]u d\mu
%		%\end{split}
%		\end{align*}
%\begin{eqnarray*}
%	\frac{N}{2\mathcal{N}}\Big[\frac{d^2\mathcal{N}}{dt^2}+2K\frac{d\mathcal{N}}{dt}\Big]\leq -\frac{2}{N}\int_X \Big[\Delta \log u-\int_X u\Delta \log u d\mu \Big]^2 ud\mu-8\int_X\frac{(\Delta \log u)^2}{N}ud\mu
%\end{eqnarray*}

	\end{proof}
%%%%%%%%%%%%%%%--section-----5--%%%%%%%%%%%%%

\section{Shannon entropy isoperimetric inequality on RCD$(0,N)$ spaces}

In this section, as an application of the concavity of the Shannon entropy power, 
we prove the Shannon entropy isoperimetric inequality and the Stam logarithmic Sobolev inequality on RCD$(0,N)$ spaces with maximal volume condition, which extend the previous results due to Li-Li \cite{LiLi2024TMJ} on complete Riemannian manifold with CD$(0,N)$ condition and maximal volume condition.

\begin{theorem}\label{EII}
	{Let $(X, d, \mu)$ be an 
 RCD$(0,N)$ space satisfying the maximal volume growth condition: there exists a constant $ C_N> 0$ such that}
		\begin{equation*}
		\mu(B(x,r))\geq C_N r^N,\quad \forall x \in X,\ r>0.
	\end{equation*}
	Then the Shannon entropy isoperimetric inequality  holds
		\begin{equation}\label{deng}
		I(f)\mathcal{N}(f)\geq\gamma_N:=2\pi Ne\kappa^{\frac{2}{N}},
	\end{equation}
	where
	\begin{equation*}
		\kappa:=\lim_{r\rightarrow\infty}\frac{\mu(B(x,r))}{\omega_N r^N}
	\end{equation*}
	is the asymptotic volume ratio on $(X, d, \mu)$,  and $\omega_N$ denotes the volume of the unit ball in $\mathbb{R}^N$. Equivalently, under the above condition, the Stam logarithmic Sobolev inequality holds: for all nice $f
	>0$ with $\int_X {|\nabla f |^2_w \over f}d\mu<+\infty$, it holds
\begin{equation*}
	\int_X f\log f d\mu \leq \frac{N}{2}\log\Big(\frac{1}{\gamma_N}\int_X {|\nabla f |^2_w\over f} d\mu\Big).
\end{equation*}
\end{theorem}
\begin{proof}
{The proof has been essentially given by Li-Li \cite{LiLi2024TMJ}. For the convenience of the reader,  we reproduce
 it here. We  know that, by Theorem \ref{EKSN} or Theorem \ref{ECPKN}, on RCD$(0,N)$ spaces, the concavity of entropy power holds:}
\begin{equation*}
	\frac{d^2}{dt^2}\mathcal{N}(u(t))\leq 0 .
\end{equation*}
 This means that
\begin{equation*}
	\frac{d}{dt}\mathcal{N}(u(t))
\end{equation*}
is non-increasing in $t$ and hence $\lim\limits_{t\rightarrow \infty}\frac{d}{dt}\mathcal{N}(u(t)) $ exists.

The first order entropy dissipation formula \eqref{H1} implies
\begin{equation*}
	\frac{d}{dt}\mathcal{N}(u(t))\leq \frac{d}{dt}\Big\vert_{t=0}\mathcal{N}(u(t))=\frac{2}{N}I(f)\mathcal{N}(f).
\end{equation*}
Thus
\begin{equation}
\frac{2}{n} I(f) \mathcal{N}(f) \geq \lim _{t \rightarrow \infty} \frac{d}{d t} \mathcal{N}(u(t)).
\end{equation}
Again, the first order entropy dissipation formula \eqref{H1} implies

\begin{equation}
\frac{d}{d t} \mathcal{N}(u(t))=\frac{2}{N} I(u(t)) \mathcal{N}(u(t)) \geq 0 .
\end{equation}
Thus $\mathcal{N}(u(t))$ is non-increasing in t. By the L'Hospital rule, we have
\begin{equation*}
	\lim_{t\rightarrow \infty}\frac{d}{dt}\mathcal{N}(u(t))=\lim_{t\rightarrow \infty}\frac{\mathcal{N}(u(t))}{t}. 
\end{equation*}
Let
\begin{equation*}
	H_N(u(t))=H(u(t))-\frac{N}{2}\log (4\pi e t).
\end{equation*}
 Then
\begin{equation*}
	\frac{\mathcal{N}(u(t))}{t}=(4\pi e) e^{\frac{2}{N}H_N(u(t))}.
\end{equation*} 

By \cite{KL}, we know that
$$
\frac{d}{d t} H_N(u(t)) \leq 0.
$$

Thus the limit $\lim\limits_{t \rightarrow \infty} H_N(u(t))$ exists and we need only to prove $\lim\limits_{t \rightarrow \infty} H_N(u(t))$ is finite.

Let $P_t=e^{t \Delta}$, and $f, h $ such that $\int_X f d\mu=\int_X h d\mu=1$. Let $H\left(P_t f\right)=$ $-\int_M P_t f \log P_t f d v$, and $H\left(P_t h\right)=-\int_X P_t h \log P_t h d\mu$. 

By Theorem 3.28 \cite{EKS2015}, the CD$(K,N)$ condition means that
\begin{align*}
	\frac{U_N\left(\mu_1\right)}{U_N\left(\mu_0\right)} \leq \mathfrak{c}_{K / N}\left(W_2\left(\mu_0, \mu_1\right)\right)+\frac{1}{N} \mathfrak{s}_{K / N}\left(W_2\left(\mu_0, \mu_1\right)\right) \sqrt{I\left(\mu_0\right)} .
\end{align*}
For the defenition of $\mathfrak{c}_{K / N}$ and $\mathfrak{s}_{K / N}$, one can see Definition \ref{Def}.

In particular, when $K=0$, We have
\begin{align*}
	e^{-\frac{1}{N}(H(\mu_0)-H(\mu_1))}\leq 1+ \frac{W_2(\mu_0,\mu_1)\sqrt{I(\mu_0)}}{N}.
\end{align*}
That is to say
\begin{align*}
	H(\mu_1)-H(\mu_0)\leq N\log\left(1+\frac{W_2(\mu_0,\mu_1)\sqrt{I(\mu_0)}}{N}\right).\end{align*}
Using the inequality $\log(1+x)\leq x, \forall x>0$, we have
\begin{align*}
	H(\mu_1)-H(\mu_0)\leq  \sqrt{I(\mu_0)}W_2(\mu_0,\mu_1).
\end{align*}
Let $\mu_0=P_tf$ and $\mu_1=P_th$, we have 
\begin{equation}\label{HWI1}
	 H(P_tf) -H(P_th) \leq \sqrt{I(P_th)}W_2(P_tfd\mu,P_thd\mu).
\end{equation}
Exchanging  $P_tf$ and $P_th$ in \eqref{HWI1}, we have 
\begin{equation}\label{HWI2}
	 H(P_th) -H(P_tf) \leq \sqrt{I(P_tf)}W_2(P_tfd\mu,P_thd\mu).
\end{equation}
We obtain following estimation by combining \eqref{HWI1} and \eqref{HWI2}
\begin{equation*}
-\sqrt{I(P_tf)}W_2(P_tfd\mu,P_thd\mu)	\leq H(P_tf) -H(P_th)\leq \sqrt{I(P_th)}W_2(P_tfd\mu,P_thd\mu)
\end{equation*}
Thus
\begin{equation*}
\left|H\left(P_t f\right)-H\left(P_t h\right)\right| \leq \max \left\{ \sqrt{I(P_tf)}, \sqrt{I(P_th)}\right\} W_2\left(P_t f d\mu, P_t h d\mu\right),\label{LLEst1}
\end{equation*}                                                                                                     
By Ebar, Kuwada  and Sturm \cite{EKS2015}, we have the following inequality on RCD$(0, N)$ spaces
\begin{equation*}
	W_2\left(P_t f d\mu, P_t h d\mu\right) \leq W_2(f d\mu, h d\mu).\label{W2cont}
\end{equation*}
By Theorem \ref{Ibound}, on RCD$(0, N)$ we have
\begin{equation*}
I(P_tf)\leq \frac{N}{2 t}.\label{LLEst2}
\end{equation*}
Therefore

$$
\left|H\left(P_t f\right)-H\left(P_t h\right)\right| \leq \sqrt{\frac{N}{2 t}} W_2(f d\mu, h d\mu).$$
Taking an approximating sequence $h_k(x) d\mu(x)$ to the Dirac mass $\delta_o(\cdot)$, we have

$$
\left|H\left(P_t f\right)-H\left(p_t(o, \cdot)\right)\right| \leq \sqrt{\frac{N}{2 t}} W_2\left(f d\mu, \delta_o(\cdot)\right).
$$

 Hence
\begin{equation*}
	\lim_{t\rightarrow\infty}(H(u(t))-H(p_t))=0.
\end{equation*}

 On the other hand, H. Li \cite{HLi} proved that  on RCD$(0, N)$ 
 spaces it holds

\begin{equation*}
	\lim_{t\rightarrow \infty} H_N(p_t)= \log \kappa .
\end{equation*}

{Consequently, we obtain}
\begin{equation*}
	I(f)\mathcal{N}(f)\geq\gamma_N:=2\pi Ne\kappa^{\frac{2}{N}},
\end{equation*}
which is equivalent to the Stam type Log-Sobolev inequality: for all $f
	>0$ with $\int_X {|\nabla f |^2_w \over f}d\mu<+\infty$, it holds
\begin{equation}
	\int_X f\log f d\mu \leq \frac{N}{2}\log\Big(\frac{1}{\gamma_N}\int_X{ |\nabla f |^2_w\over f} d\mu\Big). 
	\label{stamLogS}
\end{equation}
Replacing $f$ by $f^2$ with $\int_M f^2 d v=1$ and $\int_X|\nabla f|^2 d\mu<\infty$, the above inequality is equivalent to the following Stam type Log-Sobolev inequality: for all nice $f\neq 0$ with $\int_X |\nabla f |^2_w d\mu<+\infty$, it holds
\begin{equation*}
	\int_X f^2 \log f^2 d\mu \leq \frac{N}{2} \log \left(\frac{4}{\gamma_N} \int_X|\nabla f|^2 d\mu\right).
\end{equation*}

\end{proof}

\section{Rigidity theorem  on non-collapsing RCD$(0, N)$ space}

In Remark 8.3 in Li-Li \cite{LiLi2024TMJ}, S. Li and Li pointed out the following

\begin{remark}\label{rem8.3}{
Note that when $M=\mathbb{R}^N$, $\kappa=1$, the Stam  logarithmic  Sobolev inequality reads 
\begin{eqnarray}
\mathcal{N}(f)I(f)\geq 2\pi eN. \label{NXY5}
\end{eqnarray}
Indeed, the  Entropy Isoperimetric Inequality $(\ref{NXY5})$  is equivalent to the following  logarithmic Sobolev inequality which was first proved by 
Stam \cite{Stam}: for any smooth probability density function $f$ on $\mathbb{R}^N$ with $\int_{\mathbb{R}^N} {|\nabla f|^2\over f} dx<\infty$, it holds
 \begin{eqnarray}\label{StamLSI}
\int_{\mathbb{R}^N} f\log f dx\leq {N\over 2}\log \left({1\over 2\pi e N} \int_{\mathbb{R}^N}{|\nabla f|^2\over f}dx\right). \label{LSI1}
\end{eqnarray}
Replacing $f$ by $f^2$ with $\int_{\mathbb{R}^N} f^2dx=1$,  the above LSI   is equivalent to the Stam logarithmic Sobolev inequality: for any smooth $f$ on $\mathbb{R}^N$ with $\int_{\mathbb{R}^N} f^2dx=1$ and $\int_{\mathbb{R}^N} |\nabla f|^2 dx<\infty$, it holds

\begin{eqnarray*}
\int_{\mathbb{R}^N} f^2\log f^2 dx\leq {N\over 2}\log \left({2\over \pi e N} \int_{\mathbb{R}^N} |\nabla f|^2dx\right). 
\end{eqnarray*}
In the case where $M$ is an $N$-dimensional complete Riemannian manifold with 
non-negative Ricci curvature on which the Stam logarithmic Sobolev inequality $(\ref{stamLogS})$ 
holds with the same constant $\gamma_N=2\pi e N$ as on the $N$-dimensional Euclidean space,  then $M$ must be  isometric to $\mathbb{R}^N$. See \cite{BCL, Ni1}.}\end{remark}

In this section, inspired by the above remark, we will prove the rigidity theorem for the Stam  logarithmic  Sobolev inequality with the sharp constant on non-collapsing RCD$(0, N)$ spaces. 

\medskip

We first recall the Bishop-Gromov inequality on metric measure spaces  \cite{Sturm2006}. It says that, for $0<r<R$ and $x \in X$,

\begin{equation}\label{BG}
\frac{\mu(B(x,R))}{\mu(B(x,r))} \leq\left(\frac{R}{r}\right)^N.
\end{equation}

Moreover, Gigli and De Philippis  \cite{Gigli2016} studies the case when equality holds in \eqref{BG} . To state their result, we require the following definition:
\begin{definition}\cite{Gigli2016}
 A metric measure space $(X, d, \mu)$ is a $(0, N-1)$-cone built over a metric measure space $\left(Y, d_Y, \mu_Y\right)$ if the following holds:
 
(i) $X=[0, \infty) \times Y /\{0\} \times Y$,

(ii) $d([r, y],[s, z])=r^2+s^2-2 r s \cos \left(d_Y(y, z) \wedge \pi\right)$,

(iii) $\mu(\mathrm{d} r \mathrm{~d} y)=r^{N-1} \mathrm{~d} r \mu_Y(\mathrm{~d} y)$.
	
\end{definition}
\begin{theorem}\cite{Gigli2016}
	Let $x \in X$. Suppose that the equality holds in \eqref{BG} for any $0<r<R$.
	
(i) If $N \geq 2$, then $(X, d, \mu)$ is $(0, N-1)$-cone over an $\operatorname{RCD}(N-2, N-1)$ space and $x$ is the vertex of the cone.

(ii) If $N<2$, then $(X, d, \mu)$ is isomorphic to either $\left([0, \infty), d_{\text {Eucl }}, x^{N-1} \mathrm{~d} x\right)$ or  $\left(\mathbb{R}, d_{\mathrm{Eucl}},|x|^{N-1} \mathrm{~d} x\right)$, where $d_{\text {Eucl }}$ is the canonical Euclidean distance. 

In both cases, $x \in X$ corresponds to 0 by the isomorphism.
\end{theorem}

We now state the last main result of this paper. 

\begin{theorem}\label{Rigidity}
	Let $(X, d, \mu)$ be a non-collapsing  RCD$(0,N)$ space and the Stam logarithmic Sobolev inequality \eqref{StamLSI}
holds with the sharp constant $\gamma_N=2\pi e N$ as on the $N$-dimensional Euclidean space. Then 

(i) If $N \geq 2$,  $(X, d, \mu)$ is $(0, N-1)$-cone over an $\operatorname{RCD}(N-2, N-1)$ space and $x$ is the vertex of the cone.

(ii) If $N<2$,  $(X, d, \mu)$ is isomorphic to either $\left([0, \infty), d_{\text {Eucl }}, x^{N-1} \mathrm{~d} x\right)$ or  $\left(\mathbb{R}, d_{\mathrm{Eucl}},|x|^{N-1} \mathrm{~d} x\right)$, where $d_{\text {Eucl }}$ is the canonical Euclidean distance. 
\end{theorem}
\begin{proof} By Theorem \ref{EII}, if the Stam logarithmic Sobolev inequality \eqref{StamLSI} holds with the sharp constant $\gamma_N=2\pi e N$, then 
$$2\pi e N\kappa^{N/2}\geq 2\pi eN.$$
By Theorem 1 in \cite{BKT} , $\gamma_N=2\pi eN \kappa^{\frac{N}{2}}$ is the sharp constant such that the LSI hols. The logarithmic Sobolev inequality \eqref{StamLSI} holds with the same constant $\gamma_N=2\pi e N$ means that
	\begin{equation}
		\frac{1}{2\pi eN \kappa^{\frac{N}{2}}}\leq \frac{1}{2\pi eN},
	\end{equation}
which is 
\begin{equation}\label{kappa}
	\kappa^{\frac{N}{2}}=\lim_{R\rightarrow\infty}\frac{\mu(B(x, R))}{\omega_N R^N}\geq 1,
\end{equation}
where $\omega_N$ denotes the volume of the unit ball in $\mathbb{R}^N$.

By the Bishop-Gromov inequality in \cite{Sturm2006}, for $r<R$, 
we have 
\begin{equation*}
	\frac{\mu(B(x,R))}{\mu(B(x,r))} \leq\left(\frac{R}{r}\right)^N.
\end{equation*}
Thus\begin{equation}\label{BGIneq}
	\frac{\mu(B(x,R))}{\omega_N R^N} \leq \frac{\mu(B(x,r))}{\omega_N r^N}.
\end{equation}
Under the non-collapsing condition on RCD$(0, N)$, we have
\begin{equation}\label{Noncollapsing1}
	\lim\limits_{r\rightarrow 0} \frac{\mu(B(x,r))}{\omega_N r^N}=1, \ \ \forall ~ x\in X.
\end{equation}
Hence 
\begin{equation}\label{Noncollapsing2}
\mu(B(x, R)) \leq \omega_N R^N\ \ \ \ \forall~ x\in X,\ R>0. 
\end{equation}
On the other hand, taking the limit $R\rightarrow \infty$ in $(\ref{BGIneq})$ and using \eqref{kappa}, we see that for any $r>0$, we have 
\begin{equation}\label{volume}
	1\leq \lim\limits_{R\rightarrow \infty} \frac{\mu(B(x,R))}{\omega_N R^N}\leq \frac{\mu(B(x,r))}{\omega_N r^N}, \ \ \forall ~ x\in X.
\end{equation}
Therefore 
\begin{equation}\label{Noncollapsing3}
\mu(B(x, r)) \leq \omega_N r^N\ \ \ \ \forall~ x\in X,\ r>0. 
\end{equation}
Combining $(\ref{Noncollapsing2})$ with $(\ref{Noncollapsing3})$, we have 
\begin{equation}\label{Noncollapsiing3}
\mu(B(x, r)) =\omega_N r^N\ \ \ \ \forall~ x\in X,\ r>0. 
\end{equation}
 By the volume rigidity theorem (Theorem 7.4), we complete the proof of Theorem \ref{Rigidity}.  
\end{proof}
%
%
% Let us mention that the above rigidity theorem is the same as the  following rigidity theorem for the $W$-entropy on RCD$^*(0, N)$ space proved by Kuwada-Li \cite{KL}.
% 
%\begin{theorem}[\cite{KL}] \label{KL22}
%	Suppose that the right upper derivative of $\mathcal{W}\left(P_t \mu, t\right)$ is 0 
%	at $t=$ $t_* \in(0, \infty)$, that is,
%$$
%\varlimsup_{t \downarrow t_*} \frac{W\left(P_t \mu, t\right)-W\left(P_{t_*} \mu, t_*\right)}{t-t_*}=0,
%$$
%for some $\mu \in \mathcal{P}_2(X)$ and $t_* \in(0, \infty)$. Then $(X, d, \mu)$ must be the RCD space  as stated in Theorem \ref{Rigidity}. Moreover,  we have
%	
%(i) $I(P_t\mu)=\frac{N}{2t}$ holds  for any $t \in\left(0, t_*\right]$.
%
%(ii) For $\mu$-a.e.and  $x_0 \in X$,
%\begin{equation}\label{Fishers}
%	I\left(P_t \delta_{x_0}\right)=\frac{N}{2 t},
%\end{equation}
%holds for $t \in\left(0, t_*\right]$. In particular, there exists $x_0 \in X$ satisfying \eqref{Fishers}.
%\end{theorem}
%
%

To end this section, we would like to raise an 
interesting problem whether we can extend the rigidity theorems (Theorem \ref{WnNK=0}, Theorem \ref{Lirigidity}, Theorem \ref{KLRigidity} and Theorem \ref{Rigidity}) from CD$(0, N)$ Riemannian manifolds, RCD$(0, N)$ and RCD$(0, n, N)$ spaces to RCD$(K, N)$ and RCD$(K, n, N)$ spaces for $K\in \mathbb{R}$.

\section{Two other entropy inequalities on RCD spaces}\label{sect5}

%
%The  first  one is the following 
%result which was first proved by Wang \cite{wang} on smooth Riemannian manifolds.  As in \cite{wang}, we introduce the  following $W_K$-entropy
%\begin{equation*}
%	W_K(u,t)= \frac{e^{2Kt}-1}{2K}\frac{dH}{dt}+H+\frac{N}{2}\log 4\pi e t +\frac{N}{2}\Big(\log \frac{2Kt}{e^{2Kt}}-1\Big)-e^{2Kt}.
%\end{equation*}
%\begin{theorem}
%	Let $(X, d, \mu)$ be a metric measure sapce satisfing the Riemannian curvature dimension condition
% RCD$(-K, N)$. Then
%	\begin{equation*}
%		\frac{dW_K}{dt}\leq -\frac{e^{2Kt}-1}{NK}\int_X \Big(\Delta\log u+\frac{NKe^{2Kt}}{e^{2Kt}-1}\Big)^2 u d\mu.
%	\end{equation*}
%	In particular,
%	\begin{equation*}
%		\frac{dW_K}{dt}\leq 0.
%	\end{equation*}
%\end{theorem}

In this section, we extend two other entropy formulas to the setting of RCD  spaces. 

\subsection{Extension of Wang's entropy formula to RCD spaces}

The  first  one is the following 
result which was first proved by Wang \cite{wang}.
\begin{theorem}[\cite{wang}] \label{Wang} Let $M$ be a closed Riemannian  manifold with $Ric\geq -K$, where 
	$K \geq 0$. Define
\begin{equation}\label{Wk}
W_K(u, \tau):=  \int_M\left[\frac{e^{2 K \tau}-1}{ 2K}|\nabla f|^2+f+\frac{n}{2}\left(\log \frac{2 K \tau}{e^{2 K \tau}-1}-{e^{2 K \tau}-1\over 2K}\right)-n\right]   \frac{e^{-f}}{(4 \pi \tau)^{\frac{n}{2}}} \mathrm{~d} \mu.
\end{equation}
Then
$$
{d\over dt}W_K(u, \tau)=-\int_M \frac{e^{2 K \tau}-1}{K}\left[\left\|\nabla_i \nabla_j f-\frac{K e^{2 K \tau}}{e^{2 K \tau}-1} g_{i j}\right
\|^2_{\operatorname{HS}}+R_{i j} f_i f_j+K|\nabla f|^2\right] u \mathrm{~d}\mu,
$$
where $\frac{\mathrm{d} \tau}{\mathrm{d} t}=1$. In particular, $W_K(f, \tau)$ is monotonic non-increasing along
 the heat equation $\partial_t f=\Delta f$  when  $\Ric\geq-K$.
\end{theorem}

The next result extends Theorem \ref{Wang} to RCD$(K, N)$ spaces. 

\begin{theorem}\label{WangRCD}
	Let $(X, d, \mu)$ be an
 RCD$(K, n, N)$ space and $u$ be a solution to heat equation. Let
 \begin{equation}\label{WK2}
	W_K(u,t)= \frac{e^{2Kt}-1}{2K}\frac{dH}{dt}+H-\frac{N}{2}\log (4\pi t) +\frac{N}{2}\left[\log \frac{2Kt}{e^{2Kt}-1}-(e^{2Kt}-1)\right]-N.
\end{equation}
 Then, we have
 \begin{equation*}
			\begin{split}
				\frac{d}{dt}W_K(u, t)
				&=-\frac{e^{2Kt}-1}{K}\int_X \left[\left\|\nabla^2\log u+\frac{K e^{2 K t}} {e^{2 K t}-1} {\bf g}\right\|_{\rm HS}^2+({\bf Ric_{N, n}}(\Delta)-K {\bf g})(\nabla\log u, \nabla \log u)\right]ud\mu\\
				& \hskip2cm -{(N-n)}\frac{e^{2Kt}-1}{K}\int_X \left[ {({\rm Tr}\nabla^2-\Delta)\log u\over N-n}-{K e^{2Kt}\over e^{2Kt}-1} \right]^2ud\mu.
				\end{split}
		\end{equation*}
		In particular, we have
	\begin{equation*}
		\frac{d}{dt}W_K(u,t)\leq -\frac{e^{2Kt}-1}{NK}\int_X \Big(\Delta\log u+\frac{NKe^{2Kt}}{e^{2Kt}-1}\Big)^2 u d\mu\leq 0.
	\end{equation*}
	Moreover, the $\frac{d}{dt}W_K(u,t)=0$  if and only if 
 \begin{equation*}
			\begin{split}
			\nabla^2\log u=\frac{K e^{2 K t}} {1-e^{2 K t}} {\bf g}, \ \ \ \ {\bf Ric_{N, n}}(\Delta)(\nabla \log u, \nabla \log u)=K|\nabla\log u|^2,
			\end{split}
		\end{equation*}
			and
			 \begin{equation*}
			\begin{split}
			{({\rm Tr}\nabla^2-\Delta)\log u\over N-n}={K e^{2Kt}\over e^{2Kt}-1}.
				\end{split}
		\end{equation*}	
		
	\end{theorem}

		In particular, when $K=0$, we recapture  \eqref{WfnKN} in Theorem \ref{WnNK=0}.   
		
\begin{theorem} (i.e., Theorem \ref{WnNK=0})
	Let $(X, d, \mu)$ be an
 RCD$(0, n, N)$ space and $u$ be a solution to heat equation. Let 
 \begin{equation}\label{WK2K=0}
	W(u,t)= 2t\frac{dH}{dt}+H-\frac{N}{2}\log (4\pi t) -N.
\end{equation}
 Then
 \begin{equation*}
			\begin{split}
				\frac{dW}{dt}
				&=-2t\int_X \left[\left\|\nabla^2\log u+{ {\bf g}\over 2t}\right\|_{\rm HS}^2+{\bf Ric_{N, n}}(\Delta)(\nabla\log u, \nabla \log u)\right]ud\mu\\
				& \hskip1.5cm \ \ -{2t\over N-n}\int_X \left[ ({\rm Tr}\nabla^2-\Delta)\log u-{N-n\over 2t} \right]^2ud\mu.
				\end{split}
		\end{equation*}
In particular, 
		\begin{equation*}
		\frac{dW}{dt}\leq -\frac{2t}{N}\int_X \Big(\Delta\log u+\frac{N}{2t}\Big)^2 u d\mu\leq 0.
	\end{equation*}
	Moreover, the $\frac{dW}{dt}(u,t)=0$  if and only if 
 \begin{equation*}
			\begin{split}
			\nabla^2\log u+ {{\bf g}\over 2t}=0, \ \ {\bf Ric_{N, n}}(\Delta)(\nabla \log u, \nabla \log u)=0,			\ \ ({\rm Tr}\nabla^2-\Delta)\log u={N-n\over 2t}.
				\end{split}
		\end{equation*}	
\end{theorem}

\noindent{\bf Proof of Theorem \ref{WangRCD}}. 
	By the weak Bochner formula on RCD$(K, n, N)$ spaces, we have
			\begin{equation*}
			\begin{split}
				\frac{dW_K}{dt}&= \frac{e^{2Kt}-1}{2K}\frac{d^2H}{dt^2}+\Big(e^{2Kt}+1\Big)\frac{dH}{dt}-N\frac{Ke^{4Kt}}{e^{2Kt}-1}\\
				&=-\frac{e^{2Kt}-1}{K}\int_X \Gamma_2(\log u, \log u)ud\mu+\Big(e^{2Kt}+1\Big)\int_X |\nabla \log u|_w^2ud\mu-N\frac{Ke^{4Kt}}{e^{2Kt}-1}\\
				&=-\frac{e^{2Kt}-1}{K}\int_X \left[\|\nabla^2\log u\|_{\rm HS}^2+{\bf Ric}(\nabla\log u, \nabla \log u)\right]ud\mu+\Big(e^{2Kt}+1\Big)\int_X |\nabla \log u|_w^2ud\mu-N\frac{Ke^{4Kt}}{e^{2Kt}-1}\\
			&=-\frac{e^{2Kt}-1}{K}\int_X \left[\left\|\nabla^2\log u+\frac{K e^{2 K t}} {e^{2 K t}-1} {\bf g}\right\|_{\rm HS}^2+{\bf Ric_{N, n}}(\Delta)(\nabla\log u, \nabla \log u)\right]ud\mu\\
				&\ \ -\frac{e^{2Kt}-1}{K} \int_X \left[-2\frac{K e^{2 K t}} {e^{2 K t}-1}{\rm Tr}\nabla^2 \log u-n\left(\frac{K e^{2 K t}}{e^{2 K t}-1}\right)^2+{|({\rm Tr}\nabla^2-\Delta)\log u|^2\over N-n}\right]ud\mu\\
				&\ \  \ \  \ -(e^{2Kt}+1)\int_X \Delta \log u ud\mu-N\frac{Ke^{4Kt}}{e^{2Kt}-1}\\
				&=-\frac{e^{2Kt}-1}{K}\int_X \left[\left\|\nabla^2\log u+\frac{K e^{2 K t}} {e^{2 K t}-1} {\bf g}\right\|_{\rm HS}^2+({\bf Ric_{N, n}}(\Delta)-K{\bf g})(\nabla\log u, \nabla \log u)\right]ud\mu\\
				&\ \  \ \  \ -\frac{e^{2Kt}-1}{K} \int_X \left[-2\frac{K e^{2 K t}} {e^{2 K t}-1}({\rm Tr}\nabla^2 -\Delta)\log u+{|({\rm Tr}\nabla^2-\Delta)\log u|^2\over N-n}\right]ud\mu-(N-n)\frac{Ke^{4Kt}}{e^{2Kt}-1}\\
				&=-\frac{e^{2Kt}-1}{K}\int_X \left[\left\|\nabla^2\log u+\frac{K e^{2 K t}} {e^{2 K t}-1} {\bf g}\right\|_{\rm HS}^2+({\bf Ric_{N, n}}(\Delta)-K{\bf g})(\nabla\log u, \nabla \log u)\right]ud\mu\\
				&\ \ \  \ \ \ -{(N-n)}\frac{e^{2Kt}-1}{K}\int_X \left[ {({\rm Tr}\nabla^2-\Delta)\log u\over N-n}-{K e^{2Kt}\over e^{2Kt}-1} \right]^2ud\mu.
				\end{split}
		\end{equation*}
%In particular, using the weak Bochner inequality we have
%				\begin{equation*}
%			\begin{split}
%				\frac{dW_K}{dt}
%%				&\leq -\frac{e^{2Kt}-1}{NK} \int_X \Big[u\Big(\Delta \log u\Big)^2d\mu-(e^{2Kt}-1)\int_X  u|\nabla \log u |_w^2\Big] d\mu\\
%%				&-(e^{2Kt}+1) \int_X u\Delta\log u  d\mu-\frac{N}{2}\frac{Ke^{4Kt}}{e^{2Kt}-1}\\
%%				&=-\frac{e^{2Kt}-1}{NK} \int_X u\Big(\Delta \log u+\frac{NKe^{2Kt}}{e^{2Kt}-1}\Big)^2d\mu+(e^{2Kt}-1)\int_X u\Delta\log u  d\mu\\
%%				&-(e^{2Kt}-1)\int_X u|\nabla \log u |_w^2 d\mu\\
%				\leq-\frac{e^{2Kt}-1}{NK} \int_X u\Big(\Delta \log u+\frac{NKe^{2Kt}}{e^{2Kt}-1}\Big)^2d\mu\leq 0.
%			\end{split}
%		\end{equation*}
	This finishes the proof of Theorem \ref{WangRCD}. \hfill $\square$

\subsection{Extension of Ye's entropy formula to RCD spaces}

In \cite{Ye}, Ye  introduced the logarithmic entropy functional  and proved its monotonicity on closed Ricci flow. In \cite{Wu},  
Wu extended Ye's result to compact Riemannian manifolds with nonnegative Ricci curvature.

\begin{theorem}[\cite{Ye, Wu}]
	Let $M$ be a compact Riemannian manifold, $u$ a positive solution to the heat equation $\partial_t u=\Delta u$, and $a$ be a constant such that $\int_M|\nabla u|^2 d \mu+a> 0$. 
	Define Ye's logarithmic entropy functional as follows 
$$
\mathcal{Y}_a(u, t):=-\int_M u^2 \log u^2 d \mu+\frac{N}{2} \log \left(\int_M|\nabla u|^2 d \mu+a\right)-4 a t .
$$
Then
$$
\begin{aligned}
{d\over dt}\mathcal{Y}_a(u, t)& \leq -\frac{n}{4 \omega} \int_M\left[\left\|-2 \frac{\nabla^2 u}{u}+2 
\frac{\nabla u \otimes \nabla u}{u^2}-\frac{4 \omega}{n} 
g\right\|^2_{\operatorname{HS}}+4\operatorname{Ric}\left(\frac{\nabla u}{u}, \frac{\nabla u}{u}\right)\right] u^2 d \mu \\
& =-\frac{n}{4 \omega} \int_M\left(\left|f_{i j}-\frac{4 \omega}{n} g_{i j}\right|^2_{\operatorname{HS}}+R_{i j} 
f_i f_j\right) \frac{e^{-f}}{(4 \pi t)^{n / 2}} d \mu,
\end{aligned}
$$
where $u=(4 \pi t)^{-n / 4} e^{-f / 2}$ and $\omega:=\omega(u(x, t), a)=\int_M|\nabla u|^2 d \mu+a>0$.
\end{theorem}

The following result extends Ye and Wu's result to RCD$(K, n, N)$ spaces. 

\begin{theorem}\label{log}
	{Let $(X, d, \mu)$ be an
 RCD$(K, n, N)$ and $u$ be a positive solution to the heat equation $\partial_t u=\Delta u$. Assume that $a$ is a constant such that $\frac{1}{4}\int_X \frac{|\nabla u |_w^2}{u}d\mu +a>0$ and $aK\leq 0$.} 
Define Ye's logarithmic entropy $\mathcal{Y}_a(u,t)$ as follows
\begin{equation*}
	\mathcal{Y}_a(u,t):=-\int_X u \log u d\mu +\frac{N}{2}\log\Big(\frac{1}{4}\int_X \frac{|\nabla u |_w^2}{u} +a\Big)+(NK-4a)t.
\end{equation*}
Then, we have
	\begin{equation*}
		\frac{d\mathcal{Y}_a}{dt}\leq - \frac{1}{4\omega}\int_X \Big[\Delta\log u+4\omega\Big]^2 u d\mu+{aNK\over w},
	\end{equation*}
	where $\omega=\frac{1}{4}\int_X \frac{|\nabla u |_w^2}{u}d\mu +a$.
	In particular, when $K=0$, it holds
	\begin{equation*}
		\frac{d\mathcal{Y}_a}{dt}\leq - \frac{1}{4\omega}\int_X \Big[\Delta\log u+4\omega\Big]^2 u d\mu\leq 0.
	\end{equation*}
\end{theorem}

We give two proofs of Theorem \ref{log}.  The first one follows the same argument as used in the proof of Theorem 5.1 in \cite{LiLi2024TMJ}, where S. Li and the first named author proved the Shannon entropy power formula on complete Riemannian manifolds with CD$(K, N)$ condition. The second one uses the same argument as used in Ye \cite{Ye} and Wu \cite{Wu}.  
%Recall 
%\begin{theorem}\label{EDIKM2}
%	Let $(X, d, \mu)$ be an 
% RCD$(K, N)$ space and $u$ be a solution to the heat equation $\partial_t u=\Delta u$.  Let $H(u)=-\int_X u \log u d\mu$ be the Shannon entropy.  
%	 Then  the Entropy Differential Inequality (EDI$(K,N)$) holds 
%	\begin{equation}\label{EDIK}
%		\frac{d^2}{dt^2}H(u)+\frac{2}{N}\Big(\frac{d}{dt}H(u)\Big)^2+2K\frac{d}{dt}H(u)\leq 0,
%	\end{equation}
%	\end{theorem}

	We first prove the following result on RCD$(K, n, N)$ space, which was first proved by S.Li and the first named author \cite{LiLi2024TMJ} on complete Riemannian manifolds with bounded geometry condition.

\begin{theorem}\label{HRiccati} On any RCD$(K, n, N)$ space, we have

\begin{equation}\label{LLHK}
\begin{split}
&H''+{2\over N}H'^2+2KH'\\
&=
-2\int_X \left[\left\|\nabla^2\log u-{{\rm Tr}\nabla^2\log u\over n}\right\|_{\rm HS}^2+({\bf Ric_{N, n}}(\Delta)-K{\bf g})(\nabla \log u, \nabla \log u)\right] ud\mu\\
& \hskip0.5cm  -{2(N-n)\over Nn}\int_X \left[\Delta \log u+{N\over N-n}({\rm Tr}\nabla^2-\Delta)\log  u\right]^2ud\mu -{2\over N}\int_X \left[\Delta \log u-\int_X \Delta \log u ud\mu\right]^2 ud\mu.
	\end{split}
\end{equation}
In particular, on RCD$(K, n, N)$ space, we have

\begin{equation*}
\begin{split}
H''+{2\over N}H'^2+2KH'\leq -{2\over N}\int_X \left[\Delta \log u-\int_X \Delta \log u ud\mu\right]^2 ud\mu.
	\end{split}
\end{equation*}

\end{theorem}
\begin{proof}

The proof is as the same as in \cite{LiLi2024TMJ}. By the second order entropy dissipation formula, we have

\begin{equation}\label{HDF2nd}
\begin{split}
	H''
%	&=-2\int_X \Gamma_2(\log u, \log u) ud\mu\\
%	&=-2\int_X \left[\left\|\nabla^2\log u\right\|_{\rm HS}^2+{\bf Ric}(\nabla \log u, \nabla \log u)\right] 
%	ud\mu\\
%	&=-2\int_X \left[\left\|\nabla^2\log u-{{\rm Tr}\nabla^2\log u\over n}\right\|_{\rm HS}^2+{|{\rm Tr}\nabla^2\log u|^2\over n}\right] ud\mu\\
%	& \ \ -2\int_X \left[ {\bf Ric_{N, n}}(\Delta)(\nabla \log u, \nabla \log u)+{|({\rm Tr}\nabla^2-\Delta)\log u|^2\over N-n}\right] ud\mu\\
	&=-2\int_X \left[\left\|\nabla^2\log u-{{\rm Tr}\nabla^2\log u\over n}\right\|_{\rm HS}^2+{\bf Ric_{N, n}}(\Delta)(\nabla \log u, \nabla \log u)\right] ud\mu\\
& \ \hskip2cm -2\int_X \left[{|{\rm Tr}\nabla^2\log u|^2\over n}+{|({\rm Tr}\nabla^2-\Delta)\log u|^2\over N-n}\right] 
ud\mu.
	\end{split}
\end{equation}
Using 
$$(a+b)^2={a^2\over 1+\varepsilon}-{b^2\over \varepsilon}+{\varepsilon\over 1+\varepsilon}\left(a+{1+\varepsilon\over \varepsilon}b\right)^2$$
and taking $a=\Delta \log u$, $b={\rm Tr}\nabla^2\log u-\Delta \log u$ and $\varepsilon={N-n\over n}$, we have\begin{equation*}
\begin{split}
{|{\rm Tr}\nabla^2 \log u|^2\over n}&={|\Delta \log u+({\rm Tr}\nabla^2-\Delta)\log  u|^2\over n}\\
&={|\Delta \log u|^2\over N}-{|({\rm Tr}\nabla^2-\Delta)\log  u|^2\over N-n}+{N-n\over Nn}\left[\Delta \log u+{N\over N-n}({\rm Tr}\nabla^2-\Delta)\log  u\right]^2.
\end{split}
\end{equation*}
Substituting it into \eqref{HDF2nd}, we  can derive \eqref{LLHK}. 

%
%\begin{equation*}
%\begin{split}
%&H''+{2\over N}H'^2+2KH'\\
%&=
%-2\int_X \left[\left\|\nabla^2\log u-{{\rm Tr}\nabla^2\log u\over n}\right\|_{\rm HS}^2+({\bf Ric_{N, n}}(\Delta)-K{\bf g})(\nabla \log u, \nabla \log u)\right] ud\mu\\
%& \hskip1cm  -2{N-n\over N}\int_X \left[\Delta \log u+{N\over N-n}({\rm Tr}\nabla^2-\Delta)\log  u\right]^2ud\mu-{2\over N}\int_X \left[\Delta \log u-\int_X \Delta \log u ud\mu\right]^2ud\mu.
%	\end{split}
%\end{equation*}
%Note that 
%
%\begin{equation*}
%\begin{split}
%\int_X |\Delta \log u|^2ud\mu-\left(\int_X \Delta \log u ud\mu\right)^2=\int_X \left[\Delta \log u -\int_X \Delta \log u ud\mu\right]^2ud\mu.
%	\end{split}
%\end{equation*}
%Substituting it in to the above equality, we complete the proof of Theorem \ref{HRiccati}.
%

\end{proof}

%
%\medskip
%In particular, on RCD$(0, n, N)$ space, 
%we have
%\begin{align*}
%H''+{2\over N}{H'}^2&=
%-\frac{2}{N} \int_X \Big[\Delta \log u-\int_Xu  \Delta\log u d\mu \Big]^2ud\mu-2\int_X
%\Ric_{N, n}(\Delta) (\nabla \log u, \nabla \log u) ud\mu\\
%&\hskip0.5cm -2\int_M \left\|\nabla^2 \log u-\frac{{\rm Tr}\nabla^2\log u} {n}g\right\|^2_{\rm HS}ud\mu
%-{2(N-n)\over Nn}
%\int_M \Big[\Delta\log u+{N\over N-n}({\rm Tr}\nabla^2-\Delta) \log u\Big]^2ud\mu.
%\end{align*}
%using
%
%\begin{align*}
%	\int_X \Big[\Delta \log u-\int_Xu  \Delta\log u d\mu \Big]^2ud\mu+16a^2
%	=\int_X \Big[\Delta \log u-\int_Xu  \Delta\log u d\mu +4a\Big]^2ud\mu,
%\end{align*}
%we have
%
%\begin{equation}\label{32a}
%\begin{align*}
%	&H''+\frac{2}{N}{H'}^2-{32a^2\over N}\\
%	& =
%-\frac{2}{N} \int_X \Big[\Delta \log u-\int_Xu  \Delta\log u d\mu+4a \Big]^2ud\mu-2\int_X
%\Ric_{N, n}(\Delta) (\nabla \log u, \nabla \log u) ud\mu\\
%&\hskip0.5cm -2\int_M \left\|\nabla^2 \log u-\frac{{\rm Tr}\nabla^2\log u} {n}g\right\|^2_{\rm HS}ud\mu
%-{2(N-n)\over Nn}
%\int_M \Big[\Delta\log u+{N\over N-n}({\rm Tr}\nabla^2-\Delta) \log u\Big]^2ud\mu.
%%&= -\frac{2}{N} \int_X \Big[\Delta \log u+4w \Big]^2ud\mu-2\int_X \Ric_{N, n}(\Delta) (\nabla \log u, \nabla \log u) ud\mu-2\int_M \left\|\nabla^2 \log u-\frac{{\rm Tr}\nabla^2\log u} {n}g\right\|^2_{\rm HS}ud\mu\\
%%&\hskip0.5cm-{2(N-n)\over Nn}
%%\int_M \Big[\Delta\log u+{N\over N-n}({\rm Tr}\nabla^2-\Delta) \log u\Big]^2ud\mu.
%\end{align*}
%\end{equation}

\noindent{\bf Proof of Theorem \ref{log} using \eqref{LLHK} in Theorem \ref{HRiccati}}. 
By the definition of $\mathcal{Y}_a(u, t)$
\begin{equation*}
	\mathcal{Y}_a(u,t):=H (u(t))+\frac{N}{2}\log\Big(\frac{1}{4}{d\over dt}H(u(t)) +a\Big)+(NK-4a)t,
\end{equation*}
we have
\begin{equation*}
\begin{split}
	{d\over dt}\mathcal{Y}_a(u,t)
	%&={d\over dt}H(u(t)) +\frac{N}{2}{d\over dt}\log\Big({d\over dt}\frac{1}{4}H(u(t)) +a\Big)+NK-4a\\
%	&={d\over dt}H(u(t)) +\frac{N}{2} {{d^2\over dt^2}H(u(t)) \over {d\over dt}H(u(t)) +4a}+NK-4a\\
%	&={N\over 2}\left[{2\over N}H'+{H''\over H'+4a}\right]+NK-4a\\
%	&={N\over 2(H'+4a)}\left[{2\over N}(H'+4a)H'+H''\right]+NK-4a\\
%	&={N\over 2(H'+4a)}\left[H''+{2\over N}H'^2+{8a\over N}H'+{2\over N}(NK-4a)(H'+4a)\right]\\
	&={N\over 2(H'+4a)}\left[H''+{2\over N}H'^2+2KH'+{8a(NK-4a)\over N}\right].
\end{split}
\end{equation*}
%Therefore
%
%\begin{equation*}
%\begin{split}
%	{2(H'+4a)\over N}{d\over dt}\mathcal{Y}_a(u,t)&
%	=\left[H''+{2\over N}H'^2-{32a^2\over N}\right]\\
%\end{split}
%\end{equation*}
Using 

\begin{align*}
	\int_X \Big[\Delta \log u-\int_Xu  \Delta\log u d\mu \Big]^2ud\mu+16a^2
	=\int_X \Big[\Delta \log u-\int_Xu  \Delta\log u d\mu +4a\Big]^2ud\mu,
\end{align*}
we have
%Suppose that $w={1\over 4}H'+a={1\over 4}\int_X {|\nabla u|^2\over u}d\mu+a>0$. Then 
\begin{align*}
	{d\over dt}\mathcal{Y}_a(u, t)
	%&={N\over 2(H'+4a)}\left[H''+{2\over N}H'^2-{32a^2\over N}\right]\\
	%={N\over 8w}\left[H''+{2\over N}H'^2-{32a^2\over N}\right]\\
%	& =
%-\frac{1}{4w} \int_X \Big[\Delta \log u-\int_Xu  \Delta\log u d\mu+4a \Big]^2ud\mu-{N\over 4w}\int_X
%\Ric_{N, n}(\Delta) (\nabla \log u, \nabla \log u) ud\mu\\
%&\hskip0.5cm -{N\over 4w}\int_M \left\|\nabla^2 \log u-\frac{{\rm Tr}\nabla^2\log u} {n}g\right\|^2_{\rm HS}ud\mu
%-{N-n\over 4nw}
%\int_M \Big[\Delta\log u+{N\over N-n}({\rm Tr}\nabla^2-\Delta) \log u\Big]^2ud\mu\\
&= -\frac{1}{4w} \int_X \Big[\Delta \log u+4w \Big]^2ud\mu-{N\over 4w}\int_X ({\bf Ric_{N, n}}(\Delta)-K{\bf g}) (\nabla \log u, \nabla \log u) ud\mu\\
&\hskip1cm -{N\over 4w}\int_M \left\|\nabla^2 \log u-\frac{{\rm Tr}\nabla^2\log u} {n}g\right\|^2_{\rm HS}ud\mu\\
&\hskip1.5cm-{N-n\over 4Nnw}
\int_M \Big[\Delta\log u+{N\over N-n}({\rm Tr}\nabla^2-\Delta) \log u\Big]^2ud\mu+{aNK\over w}.
\end{align*}
In particular, on RCD$(0, n, N)$ space, we have
\begin{align*}
	{d\over dt}\mathcal{Y}_a(u, t)\leq
	-\frac{1}{4w} \int_X \Big[\Delta \log u+4w \Big]^2ud\mu\leq 0.
	\end{align*}
This finishes the proof of Theorem \ref{log}. \hfill $\square$

\medskip

For the completeness of paper, we give the second proof of Theorem \ref{log} for the particular case $K=0$ using the same argument as given in  \cite{Ye} (see also Wu\cite{Wu}). Note that the $W$-entropy  can be rewritten as
\begin{equation*}
	W(f,\tau)=-\int_X u\log ud\mu+ 4\tau \Big(\frac{1}{4}\int_X \frac{|\nabla u |_w^2}{u}d\mu+a\Big)-\frac{N}{2}\log(4\pi e\tau)-4a\tau-{N\over 2}.
\end{equation*}

\begin{lemma}[\cite{Ye, Wu}]
	Let $u\in W^{1,2}(X)$ with $\int_X u d\mu=1$.
	Then the minimum of the function
	$$h(s)=s\Big(\frac{1}{4}\int_X \frac{|\nabla u |_w^2}{u}d\mu +a\Big)-\frac{N}{2}\log s,\quad\ s>0$$
is given by 
	$$\min h=\frac{N}{2}\log\Big(\frac{1}{4}\int_X \frac{|\nabla u |_w^2}{u}d\mu +a\Big)
	+\frac{N}{2}\left(1-\log\frac{N}{2}\right),$$
	{and the unique minimum point is achieved at}
	$$s=\frac{N}{2}\Big(\frac{1}{4}\int_X \frac{|\nabla u |_w^2}{u} d\mu+a\Big)^{-1}.$$
\end{lemma}
\begin{lemma}[\cite{Ye, Wu}]
	Let $a$ be such that $w:=\frac{1}{4}\int\frac{|\nabla u |^2_w}{u}d\mu+a> 0$. Then, for each $\tau \geq 0$, we have 
	\begin{equation}\label{Lemma2}
	W(f,\tau)\geq -\int_X u \log u d\mu +\frac{N}{2}\log
	\Big(\frac{1}{4}\int_X\frac{|\nabla u |^2_w}{u}d\mu +a\Big)-4a\tau +b(N),
	\end{equation}
	where $$b(N):=-\frac{N}{2}\log \pi -\frac{N}{2}\left(1+\log \frac{N}{2}\right),$$
	Moreover
	$$W(f,\frac{N}{8\omega})=-\int_X u \log u d\mu+\frac{N}{2}\log
	\Big(\frac{1}{4}\int_X\frac{|\nabla u |^2_w}{u}d\mu+a \Big)-\frac{Na}{2\omega}+b(N).$$
	\end{lemma}

\noindent{\bf Second proof of Theorem \ref{log} following \cite{Ye, Wu}}.  
	Let $t_1\leq t_2$, $t\in[t_1,t_2] $ and   $\sigma>0$. Let 
	\begin{equation*}
		\tau(t):=t-t_1+\sigma.
	\end{equation*}
	 By  \eqref{WW}, we have 
	\begin{equation}\label{dW}
		\frac{dW}{dt}\leq -\frac{2\tau}{N} \int_X \Big(\Delta \log u+\frac{N}{2\tau}\Big)^2ud\mu.
	\end{equation}
	 Integrating \eqref{dW}  from $t_1$ to $t_2$, we obtain
	\begin{equation*}
		W(u(t_2),\tau(t_2))-W(u(t_1),\tau(t_1))\leq -\frac{2}{N} 
		\int^{t_2}_{t_1}\int_X \tau\Big(\Delta \log u+\frac{N}{2\tau}\Big)^2ud\mu dt.
	\end{equation*}
	Taking $\sigma=\frac{N}{8\omega}$,
 we have
	\begin{equation}\label{log1}
	\begin{split}
		W(u(t_2),t_2-t_1+\sigma)&\leq W(u(t_1),\sigma)-\frac{2}{N} 
		\int^{t_2}_{t_1}\int_X \tau\Big(\Delta \log u+\frac{N}{2\tau}\Big)^2ud\mu dt \\
			&=-\int_X u \log u d\mu\Big\vert_{t_1}+\frac{N}{2}\log\Big(\frac{1}{4}
			\int_X\frac{|\nabla u |^2_w}{u}d\mu+a\Big)\Big\vert_{t_1}-4a\sigma+b(N)\\
			&\hskip2cm -\frac{2}{N} 
		\int^{t_2}_{t_1}\int_X \tau\Big(\Delta \log u+\frac{N}{2\tau}\Big)^2ud\mu dt.
	\end{split}
	\end{equation}
	 On the other hand, the inequality \eqref{Lemma2} implies that
	\begin{equation}\label{log2}
		W(u(t_2),t_2-t_1+\sigma)\geq -\int_X 
		u\log u d\mu\Big\vert_{t_2}+\frac{N}{2}\log\Big(\frac{1}{4}
		\int_X\frac{|\nabla u |^2_w}{u}d\mu+a\Big)\Big\vert_{t_2}-4a(t_2-t_1+\sigma)+b(N).
	\end{equation}
	Combining \eqref{log1} with \eqref{log2}, we get
	\begin{equation*}
		\mathcal{Y}_a(u(t_2),t_2)\leq \mathcal{Y}_a(u(t_1),t_1) -
		\frac{2}{N} \int^{t_2}_{t_1}\int_X \tau\Big(\Delta \log u+\frac{N}{2\tau}\Big)^2ud\mu dt,
	\end{equation*}
	and hence
	\begin{equation*}
		\lim_{t_2\rightarrow t_1+}\frac{\mathcal{Y}_a(u(t_2),t_2)-
		\mathcal{Y}_a(u(t_1),t_1)}{t_2-t_1}\leq -{2\over N}\lim_{t_2\rightarrow t_1+}\frac{1}{t_2-t_1}
		\int^{t_2}_{t_1}\int_X \tau\Big(\Delta \log u+\frac{N}{2\tau}\Big)^2ud\mu dt.
	\end{equation*}
	Since $t_1$ is arbitrary and $4\tau =\frac{N}{2\omega}$, it follows that
	\begin{equation*}
	\begin{split}
		\frac{d\mathcal{Y}_a}{dt}\leq -{2\tau\over N}\int_X 
		\Big(\Delta \log u+\frac{N}{2\tau}\Big)^2ud\mu=- \frac{1}{4\omega}\int_X \Big(\Delta\log u+4\omega\Big)^2 u d\mu\leq 0.
	\end{split}
	\end{equation*}
	\hfill $\square$

\vspace{5mm}

Xiang-Dong Li, State Key Laboratory of Mathematical Sciences, Academy of Mathematics and Systems Science, Chinese Academy of Sciences, No. 55, Zhongguancun East Road, Beijing, 100190, China,  
and 
 School of Mathematics, University of Chinese Academy of Sciences, Beijing, 100049, China. Email: xdli@amt.ac.cn

\vspace{5mm}
Enrui Zhang,  School of Mathematics, University of Chinese Academy of Sciences, Beijing, 100049, China. Email: zhangenrui@amss.ac.cn

\end{document}